\documentclass{article}

\usepackage[utf8]{inputenc} 
\usepackage[T1]{fontenc}    
\usepackage[colorlinks=true,linkcolor=blue, urlcolor=blue, citecolor=blue]{hyperref}   
\usepackage{url}            
\usepackage{booktabs}       
\usepackage{amsfonts}       
\usepackage{amsmath, amsthm, amssymb}
\usepackage{nicefrac}       
\usepackage{microtype}      
\usepackage{xcolor}         
\usepackage{multirow}
\usepackage[margin=1in]{geometry}
\usepackage{relsize}
\usepackage{dsfont}

\usepackage{graphicx}

\newcommand{\D}{\textrm{d}}
\newcommand{\R}{\mathbb{R}}
\newcommand{\X}{\mathbf{X}}
\newcommand{\Y}{\mathbf{Y}}
\newcommand{\E}{\mathbf{E}}

\newcommand{\bbP}{\mathbb{P}}

\newcommand{\cA}{\mathcal{A}}

\newcommand{\cR}{\mathcal{R}}

\newcommand{\EE}{\mathbb{E}}

\newcommand{\pr}{\mathbb{P}}

\newcommand{\var}{\textrm{var}}

\newcommand{\one}{\mathds{1}}

\newcommand{\iid}{\overset{\textrm{iid}}{\sim}}

\newcommand{\eps}{\varepsilon} 
\newcommand{\Ct}{C_{*}}

\newcommand{\lin}{\text{lin}}
\newcommand{\maxlin}{\text{max-lin}}
\newcommand{\trunc}{\mathsf{trunc}}

\usepackage{enumerate}
\newcommand{\chisqlin}{\text{trunc-$\chi^2$}}
\newcommand{\chisqmax}{\text{max-trunc-$\chi^2$}}
\newcommand{\linmax}{\text{max-lin}}

\usepackage{enumerate}

\newtheorem{theorem}{Theorem}
\newtheorem{lemma}{Lemma}
\newtheorem{definition}{Definition}
\newtheorem{corollary}{Corollary}

\newtheorem{proposition}{Proposition}

\theoremstyle{remark}

\usepackage{xr-hyper}
\usepackage{lmodern}
\usepackage{natbib}

\title{Minimax optimal submatrix detection: Sharp non-asymptotic rates}

\author{%
  Parker Knight and Julien Chhor\footnote{Direct correspondence to \texttt{pknight@g.harvard.edu}. The authors contributed equally to this work.}  \\
Harvard University, Boston, USA and \\ Toulouse School of Economics, University of Toulouse Capitole, Toulouse, France
}

\begin{document}

\maketitle

\begin{abstract}
  Given an observation $\mathbf Y \in \R^{d_1\times d_2}$ from the model $\mathbf Y = \mathbf X + \mathbf E$ where $\mathbf X$ is constant and $\mathbf E$ has i.i.d. $N(0,1)$ entries, we consider the problem of detecting a planted submatrix in the mean matrix $\mathbf X$. Specifically, we aim to distinguish the null hypothesis $\mathbf X = 0$ from the alternative hypothesis in which $\mathbf X$ is non-zero only on a submatrix of size $s_1 \times s_2$ with elevated entries bounded below by $\mu>0$.
  We establish a minimax lower bound characterizing how large $\mu$ must be to ensure that the two hypotheses are distinguishable with high probability. 
  Furthermore, we derive novel minimax-optimal tests achieving the lower bound, and describe extensions of these tests that are adaptive to unknown sparsity levels $s_1$ and $s_2$. 
  In contrast with previous work, which required restrictive assumptions on $s_1,s_2, d_1$ and $d_2$, our non-asymptotic upper and lower bounds match for any configuration of these parameters.
  
\end{abstract}

\addtocontents{toc}{\protect\setcounter{tocdepth}{-1}}

\section{Introduction}\label{sec:intro}

Since the groundbreaking work of Ingster \citep{ingster1982minimax, ingster1987minimax,ingster1996some, ingster1998minimax, ingster2012nonparametric}, a vast literature has emerged aiming to establish fundamental statistical limits of testing hypotheses under structured alternatives \citep{baraud2002non, arias2008searching, arias2011global,arias2014community, verzelen2015community, collier2017minimax, ingster2010detection, berthet2013optimal, perry2018optimality, liu2021minimax}. A notable line of research within this literature focuses on the \textit{submatrix detection problem}, in which in the statistician observes a noisy matrix and seeks to determine, by way of statistical hypothesis testing, whether the matrix admits a potentially sparse submatrix of entries with elevated mean. Submatrix detection arises naturally in practical high-dimensional data analysis, in particular as the biclustering problem in genomics \citep{cheng2000biclustering, shabalin2009finding, xie2019time, balakrishnan2011statistical} and as community detection in statistical network analysis \citep{dey2022community, arias2014community, fortunato202220}. The rigorous study of submatrix detection was famously initiated by Butucea and Ingster \citep{butucea2013detection}, in which the authors derive the precise minimax detection boundary for detecting a sparse submatrix in homoscedastic Gaussian noise in a particular high-dimensional asymptotic regime. The precision and tightness of the results of \citet{butucea2013detection}, as well as those of subsequent works \citep{arias2014community, verzelen2015community, dadon2024detection, luo2022tensor}, explicitly rely on the form of the asymptotic regime considered by the authors. To our knowledge, there are no existing results establishing the non-asymptotic minimax rate of submatrix detection. 
Moreover, prior work has relied on restrictive assumptions on the shape of the observed matrix and the planted submatrix, which severely hampers the statistical community's understanding of this problem in regimes not captured by such assumptions.

We aim to close this gap by studying the submatrix detection problem in an entirely non-asymptotic setting. We observe a matrix $\Y \in \R^{d_1\times d_2}$ drawn from the model
\begin{equation}\label{eq_model}
    \Y = \X  +\E
\end{equation}
where $\X$ is a constant mean matrix and the entries of the noise matrix $\E$ satisfy $\E_{ij}\iid N(0,1)$ for $(i,j) \in [d_1] \times [d_2]$. Our goal is to test for whether or not $\X$ admits a submatrix of elevated entries. We formalize this problem in the language of minimax hypothesis testing.

Let $s_1, s_2 \in \mathbb N$ with $s_1 \leq d_1$ and $s_2 \leq d_2$. For any $k,n \in \mathbb N$, let $\mathcal{P}_k(n)$ denote the set of all subsets of $[n]$ with cardinality $k$.
For any $S_1 \in \mathcal{P}_{s_1}(d_1)$, $S_2 \in \mathcal{P}_{s_2}(d_2)$ and any $\mu > 0$, we define
\begin{align*}
    \Theta(S_1, S_2, \mu) = \left\{\mathbf M \in \R^{d_1\times d_2} ~\text{ s.t. }~ \Bigg\{\begin{array}{ll}
        \!\forall (i,j) \in S_1\!\times\! S_2: \, \mathbf M_{ij} \geq \mu\\
        \!\forall (i,j) \notin S_1\!\times\! S_2: \, \mathbf M_{ij} = 0
    \end{array}
    \right\}.
\end{align*}
For any $\mu>0$, let 
\begin{align*}
    \Theta(s_1,s_2, d_1, d_2, \mu) = \bigcup_{\substack{S_1 \in \mathcal{P}_{s_1}(d_1)\\ S_2 \in \mathcal{P}_{s_2}(d_2)}} \Theta(S_1, S_2, \mu).
\end{align*}
For any matrix $\X \in \Theta(s_1,s_2, d_1, d_2, \mu)$, the quantities $s_1$ and $s_2$ denote the row and column sparsity levels of $\mathbf{X}$ respectively, and $\mu$ represents the strength of the signal contained in the nonzero entries of $\X$. We consider the testing problem
\begin{align}
    H_0: \forall (i,j) \in [d_1]\times [d_2]: \X_{ij} =  0 \qquad \text{ against } \qquad  H_1(\mu) : \X \in \Theta(s_1,s_2, d_1, d_2, \mu).\label{eq_testing_problem}
\end{align}
A \textit{test} is a measurable function of the data $\Y$ taking its values in $\{0,1\}$. We measure the quality of a test $\Delta$ by its \textit{risk}, defined as the sum of its type I and type II errors.  Here and throughout the paper, we denote by $\mathbb P_\X$ the probability distribution of $\Y$ under the mean parameter $\X \in \R^{d_1 \times d_2}$ and let $\mathbb P_{0}$ denote the distribution of $\Y$ under the null hypothesis.

\begin{definition}
     The risk of a test $\Delta$ is defined as
\begin{equation}\label{eq_testing_risk}
    \cR(\Delta, \mu) = \mathbb P_{0}(\Delta = 1) + \sup_{\X \in \Theta(s_1,s_2,d_1,d_2, \mu)} \mathbb P_\X(\Delta = 0).
\end{equation}
\end{definition}
\noindent Furthermore, we define the \textit{minimax risk} as the infimal risk over all tests for the problem (\ref{eq_testing_problem}).
\begin{definition}
    The \textit{minimax risk} associated with~\eqref{eq_testing_problem} is defined as
\begin{equation}\label{eq_minimax_risk}
    \cR^*(s_1,s_2,d_1,d_2, \mu) = \inf_{\Delta} \cR(\Delta, \mu),
\end{equation}
where the infimum is taken over all tests $\Delta$.
\end{definition}

\noindent We measure the difficulty of the testing problem (\ref{eq_testing_problem}) via the \textit{minimax separation rate} $\mu^*$ defined as follows.
\begin{definition}
Let $\eta \in (0,1)$ be given. The minimax separation rate of problem (\ref{eq_testing_problem}) at the desired level of risk $\eta$ is defined as
    \begin{equation}\label{eq_minimax_rate}
    \mu^*(s_1, s_2, d_1, d_2) = \inf \left\{\mu>0: \cR^*(s_1,s_2,d_1,d_2,\mu) < \eta\right\}.
\end{equation}
\end{definition}

We will consider $\eta \in (0,1)$ as a fixed constant throughout the paper. We will also write $\mu^*$ instead of $\mu^*(s_1,s_2, d_1, d_2)$ whenever no ambiguity arises, but we emphasize that our goal is to obtain a sharp expression of this quantity of interest as a function of $s_1, s_2, d_1$ and $d_2$. The minimax separation rate $\mu^*$ is a non-asymptotic function of the tuple $(s_1, s_2, d_1, d_2, \mu)$ that provides us with necessary and sufficient conditions on the signal strength $\mu$ such that consistent detection is possible with risk controlled to be at most $\eta$. Whenever $\mu$ satisfies $\mu > \mu^*$, then a test $\Delta^*$ exists such that $\cR(\Delta^*, \mu) \leq \eta$. Conversely, if $\mu < \mu^*$, no such test exists. In the present work, we aim to derive the value of $\mu^*$ up to absolute multiplicative constants, which will grant us a complete characterization of the statistical hardness of the testing problem (\ref{eq_testing_problem}) across all values of $s_1, s_2, d_1, d_2$.

\subsection{Prior work}

\citet{butucea2013detection} established the first formal decision-theoretic results pertaining to submatrix detection in Gaussian noise. They consider the asymptotic setting with $s_1, s_2, d_1$, and $d_2$ all tending to infinity with $s_1 = o(d_1)$ and $s_2 = o(d_2)$, and impose that 
\begin{enumerate}
    \item $s_1 \log(d_1/s_1) \asymp s_2 \log(d_2/s_2)$ 
    \item $\log \log (d_1/s_1) = o(\log(d_2/s_2))$ and $\log \log (d_2/s_2) = o(\log(d_1/s_1))$.
\end{enumerate}
In this regime, \citet{butucea2013detection} show that the asymptotic minimax rate of separation in this regime satisfies
\begin{equation}\label{eq:bi-intro}
    \big(\mu^*\big)^2 = \left[\frac{d_1d_2}{s_1^2s_2^2} \land 2 \Bigg(\frac{1}{s_2}\log\left(\frac{d_1}{s_1}\right) + \frac{1}{s_1}\log\left(\frac{d_2}{s_2}\right)\Bigg)\right](1+o(1)).
\end{equation}
This expression establishes a clear phase transition in the form of $\mu^*$, which we elucidate by temporarily restricting our attention to the asymptotically square submatrix setting $s_1 \asymp s_2, d_j \asymp s_j^{A_j}$ where $A_j > 1$ for $j \in \{1,2\}$. The minimax rate in this setting simplifies to:
\begin{align*}
    \mu^* &\asymp s_1^{-2 + (A_1 + A_2)/2} \quad \quad \quad \quad \quad  \; \; \;\text{if $A_1 + A_2 \leq 3$}, \\
    \mu^* &\asymp \sqrt{\frac{2(A_1 + A_2 - 2)\log(s_1)}{s_1}} \quad \text{otherwise.}
\end{align*}
The former expression corresponds to the \textit{dense regime}, in which $s_1$ and $s_2$ grow quickly relative to $d_1$ and $d_2$ respectively, and the latter to the \textit{sparse regime}. This phase transition in $\mu^*$ is analogous to classical results in the sparse vector testing literature \citep{ingster1996some, donoho2004higher}, but the result of Butucea and Ingster carries with it the additional complexity introduced by the product-structure of the alternative hypothesis in \eqref{eq_testing_problem} (namely, the decoupling of $A_1$ and $A_2$). Furthermore, \eqref{eq:bi-intro} is exact, and establishes the precise minimax constant in the asymptotic regime that they consider.

Following this seminal result of \cite{butucea2013detection}, a wave of results were published establishing the statistical properties of submatrix detection beyond the homoscedastic Gaussian setting. Arias-Castro and Verzelen considered submatrix detection in Bernoulli noise \citep{arias2014community, verzelen2015community}, motivated by the problem of community detection in Erdos-Renyi graphs, and establish the minimax rates of detecting a square submatrix in the adjacency matrix of a random graph with elevated connection probability in the dense \citep{arias2014community} and sparse \citep{verzelen2015community} graph settings respectively. The work of \citet{luo2022tensor} extended the results of \citet{butucea2013detection} to the multi-modal tensor setting, and prove the minimax rate of detecting a subtensor in a tensor of standard Gaussian noise.  More recently, \citet{dadon2024detection} provided upper and lower bounds on the minimax rate of detecting multiple submatrices in a square Gaussian matrix that are tight up to log factors. \citet{oren2026inhomogeneous} extended these results to the inhomogeneous setting, in which the planted submatrices may have heterogeneous means or covariance matrices.

A parallel line of research investigates computational aspects of the submatrix detection problem; namely, whether or not the information-theoretic lower bounds establishes by \citet{butucea2013detection} are attainable by (possibly randomized) polynomial time algorithms. Under the planted clique hypothesis, \citet{ma2015computational} established that no randomized polynomial time test can detect a square submatrix of size $s \times s$ in a $d \times d$ Gaussian matrix if $\mu = o(d s^{-2 - \delta})$ for $\delta > 0$ and $\sqrt{\log d} \lesssim \mu^{-1}$. These conditions on $\mu$ are strictly stronger than the lower bound of \citet{butucea2013detection} when $\sqrt{s} \gg \log d$, revealing a regime in which detection is statistically possible, but computationally intractable. \citet{hajek2015computational} established an analogous statistical-computational gap in the community detection problem. These results were generalized by \citet{brennan2019universality}, who derived computational lower bounds for detecting a submatrix of $\mathcal{Q}$-distributed entries in a matrix of $\mathcal{P}$-distributed data which hold for any $(\mathcal{Q}, \mathcal{P})$ pair with a finite KL divergence. Computational barriers to submatrix localization, defined as the problem of finding a planted submatrix rather than merely detecting its presence, has also received widespread attention in the statistics community over the past several years; we refer interested readers to \citet{balakrishnan2011statistical, chen2016statistical, cai2017computational, hajek2017information, hajek2018submatrix, hajek2016semidefinite}.

Our work also builds on and contributes to recent developments in the sparse Gaussian sequence model (GSM) testing literature, which dates back to \citet{baraud2002non}. \citet{collier2017minimax} established non-asymptotic minimax rates for estimating linear and quadratic functionals on the sparse GSM, which they used to derive the corresponding minimax rates of signal detection under sparsity. The concurrent work \citet{collier2018optimal} presented optimal rates for adaptive estimation of linear functions when the sparsity and noise variance are unknown. \citet{10.3150/19-BEJ1180} presented minimax rates for estimating general $\ell_p$ norms (i.e., nonlinear functionals) of the mean parameter under the same sparse GSM. \citet{carpentier2019adaptive} considered the estimation of the degree of sparsity $s$ given draws from a $d$-dimensional GSM. Minimax rates for sparse changepoint detection in the GSM where established by \citet{liu2021minimax}. Recently, \citet{kotekal2023minimax} described minimax rates of signal detection in the sparse GSM with correlated observations, and \citet{laurent2012non,chhor2024sparse} derived analogous results under heteroskedastic noise profiles.

Finally, we remark that the results in this paper complement those of our concurrent work \citet{chhor2026optimal}, which considers the problem of detecting a planted community in a bipartite graph. This corresponds to the submatrix detection problem under Bernoulli noise, which presents entirely different behavior from the Gaussian setting. Moreover, the results of \citet{chhor2026optimal} are limited by restrictive assumptions on the observed graph's baseline connection probability, and thus fail to completely characterize the statistical hardness of community detection in bipartite graphs. 
These restrictive assumptions also cause difficulty in obtaining adaptation to the size of the hidden communities. In the present paper, we completely resolve the minimax rate of submatrix detection in Gaussian noise and describe optimal adaptive tests, which constitutes a significant advance over our own companion work. Furthermore, we emphasize that since we place ourselves in the Gaussian setting, the optimal testing procedure described in Section \ref{sec_ub} relies on fundamentally different test statistics from those used in the bipartite graph setting.
 
\subsection{Our contributions}
In this paper, we reopen the investigation into the statistical limits of submatrix detection in Gaussian noise, and make the following contributions.
 \begin{enumerate}
     \item We fully resolve the non-asymptotic  minimax rate $\mu^*$ defined in \eqref{eq_minimax_rate}, and show that it satisfies
     \begin{align*}
    \big(\mu^*\big)^2 \asymp &\Bigg[\frac{1}{s_1} \log\!\Big(1+ \frac{d_2}{s_2^2} \log\!\Big(\!e\scalebox{.92}{$\displaystyle{d_1 \choose s_1}$}\Big)\,\Big) + \frac{1}{s_2} \log\!\Big(1+ \frac{d_1}{s_1^2} \log\!\Big(\!e\scalebox{.92}{$\displaystyle{d_2 \choose s_2}$}\Big)\,\Big)\Bigg] \\[5pt]
    &\land \, \Bigg[\frac{d_1}{s_1^2} \log\left(1+\frac{d_2}{s_2^2}\right)\Bigg] \, \land \, \Bigg[\frac{d_2}{s_2^2} \log\left(1+\frac{d_1}{s_1^2}\right)\Bigg].
\end{align*}
    Our results generalize the rates of sparse signal detection in the vector case~\citep{collier2017minimax} and also encompass the seminal results of Butucea and Ingster up to constants (Corollary \ref{cor_match_bi}), which required the stringent balancedness assumption $s_1\log(d_1/s_1) \asymp s_2 \log(d_2/s_2)$. 
     In contrast to the previous literature, our bounds always match up to multiplicative constants for any configuration of $d_1,d_2,s_1$ and $s_2$. 
     This allows us to provide a comprehensive picture of all possible regimes and phase transitions, especially in highly imbalanced settings $s_1 \ll s_2$ or $d_1 \ll d_2$, which, to the best of our knowledge, were previously absent from the literature. 

     \item 
     Our proof of the lower bound on $\mu^*$ in Theorem \ref{thm_main} relies on a refined use of the second moment method. However, our argument departs significantly from the lower bound techniques developed in~\cite{butucea2013detection} and~\cite{arias2014community}. Rather than analyze a truncated likelihood ratio statistic, we directly control the variance of the likelihood ratio under all possible regimes of sparsity. This requires several auxiliary results for controlling the moment-generating function of the product of binomial random variables which may be of independent interest.

     \item Our upper bound on the minimax rate requires the development of new testing procedures. We use a combination of linear tests and truncated $\chi^2$-based tests, first studied by \citet{collier2017minimax} for optimal testing in the sparse Gaussian sequence model, and carefully designed scanning procedures to achieve the optimal submatrix detection rates across all values of $s_1, s_2, d_1$, and $d_2$. To our knowledge, the tests described in this paper (as well as in \citet{chhor2026optimal}) constitute the first multi-dimensional extensions of the truncated $\chi^2$ method now widely-used in the testing literature \citep{collier2017minimax, kotekal2023minimax, liu2021minimax, chhor2024sparse}, and we anticipate that the fundamental ideas behind the construction of our optimal tests will inspire the development of new methods for optimal sparse signal detection in matrix- or tensor-structured data. Furthermore, we describe an adaptive extension of our testing procedure that is minimax optimal without knowledge of the size of the planted submatrix $(s_1, s_2)$.
 \end{enumerate}


The manuscript is structured as follows. In Section \ref{sec_main}, we present our main theorem. Section \ref{sec_ub} presents the construction of our optimal testing procedure and gives an expanded statement of the upper bound included in Theorem \ref{thm_main} (Proposition \ref{prop_ub}). We then describe the adaptive extension of our optimal test and present a guarantee on its risk (Proposition \ref{prop_ada_ub}). In Section \ref{sec_comparison}, we provide a detailed comparison of our results to those in the existing literature. Finally, Section \ref{sec_discussion} summarizes our contributions and outlines the many promising avenues for future work implied by our results.

\subsection{Notation}

The following notation will be used throughout the paper. For $d \in \mathbb N$, let $[d] := {1,...,d}$. We use $\mathcal P_s(d)$ to denote the set of subsets of $[d]$ of size $s$. For $a,b \in \R$,
denote $a \lor b:= \max\{a,b\}$ and $a \land b =: \min\{a,b\}$.  We will use $a \lesssim b$ if there exists a constant $C > 0$ depending on $\eta$ such that $a \leq C b$. We say $a \asymp b$ if $a \lesssim b$ and $b \lesssim a$. For two sets $A_1$ and $A_2$, we denote $A_1 \times A_2 = \{(i,j) : i \in A_1, j \in A_2\}$ as the Cartesian product of $A_1$ and $A_2$. For a finite set $A$, we use $|A|$ to denote the cardinality of $A$. We use $\one_{\{\cdot\}}$ as the indicator function, meaning that $\one_{A} = 1$ if the event $A$ occurs and $\one_{A} = 0$ otherwise. For a matrix $\mathbf X$, we use $\X_{ij}$ to denote its $(i,j)_{\text{th}}$ entry. Given two probability distributions $\bbP$ and $\mathbb Q$, we use $\text{TV}(\bbP, \mathbb Q) = \sup_A|\bbP(A) - \mathbb Q(A)|$ to denote the total variation distance between $\bbP$ and $\mathbb Q$. We use $c, C, \bar{C}, C', c_1, c_2, c_3$, and $c_4$ to denote constants whose value may change between instances of their usage in the paper. 
We will also use the convention that $[a,b] = \emptyset$ for any two real numbers $a,b$ such that $a>b$. 
For any two integers $n,k \in \mathbb N\cup \{0\}$, we denote by ${n \choose k}$ the binomial coefficient, equal to $\frac{n!}{k! (n-k)!}$.

\section{Main result}\label{sec_main}

 We begin this section by introducing notation that will serve us in the statements of the results that follow. For any $s_1,s_2, d_1, d_2 \in \mathbb N$, we define the following quantities:
\begin{align}
    \psi(s_1,s_2,d_1,d_2) &= \frac{1}{s_1} \log\left(1+ \frac{d_2}{s_2^2} \log\left(e{d_1 \choose s_1}\right)\right), \label{eq_def_psi}\\[5pt]
    \phi(s_1,s_2, d_1,d_2) &= 
        \frac{d_1}{s_1^2} \log\left(1+\frac{d_2}{s_2^2}\right)\label{eq_def_phi}
\end{align}
To make our notation concise, we write
\begin{align*}
    &\phi_{12} = \phi(s_1,s_2, d_1,d_2),
    &\phi_{21} = \phi(s_2,s_1, d_2,d_1),\,\\
    &\psi_{12} = \psi(s_1,s_2, d_1,d_2),
    &\psi_{21} = \psi(s_2,s_1, d_2,d_1).
\end{align*}
The following theorem is our main result, which describes the precise form of the minimax rate of separation $\mu^*$ as defined in (\ref{eq_minimax_rate}) up to constants.
\begin{theorem}\label{thm_main}
    Let $\eta \in (0,1)$ be given. Then, the following claims hold. 
    \begin{enumerate}
        \item There exists a constant $c > 0$ that depends only on $\eta$ such that
        \begin{equation}
            (\mu^*)^2 \geq c \cdot \big[\left( \psi_{12} + \psi_{21}\right) \land \phi_{12} \land \phi_{21} \big].
        \end{equation}
        \item There exists a constant $C' > 0$ that depends only on $\eta$ such that if $\mu^2 \geq C' \cdot [(\psi_{12} + \psi_{21}) \land \phi_{12} \land \phi_{21}]$, then the test $\Delta^*$ defined in \eqref{eq:delta-optimal} satisfies
        $\cR(\Delta^*, \mu) \leq \eta$.
        This implies that 
        \begin{equation}
            (\mu^*)^2 \leq C' \cdot \big[\left( \psi_{12} + \psi_{21}\right) \land \phi_{12} \land \phi_{21} \big].
        \end{equation}
    \end{enumerate}
\end{theorem}
Theorem \ref{thm_main} establishes that the minimax rate satisfies
\begin{equation}\label{eq_our_rate}
    (\mu^*)^2 \asymp \left( \psi_{12} + \psi_{21}\right) \land \phi_{12} \land \phi_{21}.
\end{equation}
The upper bound of Theorem \ref{thm_main} is achieved by an entirely novel testing procedure constructed from a combination of carefully aggregated truncated $\chi^2$ and linear test statistics. We defer the details to Section \ref{sec_ub}. The proof of the lower bound in Theorem \ref{thm_main} relies on a precise application of the second moment method \citep{ingster1982minimax}.  We lower bound the minimax risk with the risk of a simple versus simple testing problem, in which the alternative hypothesis is defined by integrating over our parameter space with a uniform prior. We then invoke the Neyman-Pearson lemma and control the second moment of the resulting likelihood ratio statistic to show that the testing risk can be made arbitrarily large for $\mu \leq c\mu^*$, where $c > 0$ is a small constant that depends on $\eta$.  Notably, our proof substantially departs from that of \citet{butucea2013detection}. Rather than employ the well-established truncated second moment method \citep{butucea2013detection, arias2014community, verzelen2015community}, we perform an elementary but extremely detailed analysis of the second moment of the likelihood ratio test across all possible regimes of $(s_1, s_2, d_1, d_2)$. We provide the proof in detail in the supplement.

\section{Upper bounds}\label{sec_ub}
In this section, we describe the novel tests that we combine to achieve the upper bound in Theorem \ref{thm_main}. For $Z \sim N(0,1)$ and any $\tau>0$, define 
\begin{equation}\label{eq:betatau}
    \nu_{\tau} = \EE[Z^2 \big|\, |Z|>\tau]  
\end{equation}
Given the threshold $\tau_{\chisqlin, 1} = \sqrt{C\log(1 + \frac{d_2}{s_2^2})}$ for constant $C > 0$ to be chosen later, we define the \textit{truncated $\chi^2$ test statistic} as
\begin{equation}\label{eq_trunc_lin}
    t_{\chisqlin, 1} = \sum_{j = 1}^{d_2}\Big((\bar{\Y}_j)^2 - \nu_{\tau_{\chisqlin, 1}}\Big)\one\big(|\bar{\Y}_j| > \tau_{\chisqlin, 1}\big)
\end{equation}
for $\bar{\Y}_j = \frac{1}{\sqrt{d_1}}\sum_{i = 1}^{d_1}\Y_{ij}$. For $h > 0$, the truncated $\chi^2$ test is defined as 
\begin{equation}
\Delta_{\chisqlin, 1}^h = \one(t_{\chisqlin} > h).
\end{equation}
For any $J_1 \in \mathcal{P}_{s_1}(d_1)$ and $j \in \{1, ..., d_2\}$, define 
\[\bar{\Y}_{J_1, j} = \frac{1}{\sqrt{s_1}}\sum_{i \in J_1}\Y_{ij}.\]
For $\tau_{\chisqmax, 1} = \sqrt{C\log\Big(1 + \frac{d_2}{s_2^2}\log (e{d_1 \choose s_1})\Big)}$ and $C > 0$ to be chosen later, we define the \textit{Bonferroni corrected truncated $\chi^2$ test statistic} as
\begin{equation}\label{eq_trunc_max}
    t_{\chisqmax, 1} = \max_{J_1 \in \mathcal P_{s_1}\!(d_1)}\sum_{j = 1}^{d_2}\big((\bar{\Y}_{J_1, j})^2 - \nu_{\tau_{\chisqmax, 1}}\big)\one(|\bar{\Y}_{J_1, j}| > \tau_{\chisqmax, 1}).
\end{equation}
The corresponding test is defined as 
\begin{equation}
\Delta_{\chisqmax, 1}^h = \one(t_{\chisqmax} > h),
\end{equation}
for $h > 0$. We define $\Delta_{\chisqlin, 2}^h$ and $\Delta_{\chisqmax, 2}^h$ by swapping the roles of $s_1$ and $d_1$ with those of $s_2$ and $d_2$ in the construction of the test statistics \eqref{eq_trunc_lin} and \eqref{eq_trunc_max} as well as the thresholds $\tau_{\chisqlin, 1}$ and $\tau_{\chisqmax, 1}$, respectively. Next, we define the \textit{linear test statistic} as
\begin{equation}
    t_{\lin} = \frac{1}{\sqrt{d_1d_2}}\sum_{i, j}\Y_{ij},
\end{equation}
and the corresponding linear test is
\begin{equation}
    \Delta^h_{\lin} = \one(t_\lin > h),
\end{equation}
for $h > 0$.
Finally, we define the \textit{Bonferroni corrected linear test statistic} as
\begin{equation}\label{eq_lin_max}
    \displaystyle t_{\linmax, 1} = \max_{J_1 \in \mathcal P_{s_1}\!(d_1)}\frac{1}{\sqrt{d_2}}\sum_{j = 1}^{d_2}\bar{\Y}_{J_1, j}.
\end{equation}
The corresponding test is defined as 
\begin{equation}
\Delta_{\linmax, 1}^h = \one(t_{\linmax, 1} > h),
\end{equation}
for $h > 0$. We will aggregate the linear tests and the truncated $\chi^2$ tests as follows:
\begin{align}
    &\Delta_a^{h_1,h_2} = \begin{cases}
        \Delta^{h_1}_{\chisqlin, 1} & \text{if  $\frac{d_2}{s_2^2} \geq 1$,} \\ \Delta^{h_2}_{\lin} & \text{otherwise}
    \end{cases} \qquad \text{ and } 
    \qquad \Delta_b^{h_1',h_2'} = \begin{cases}
        \Delta^{h_1'}_{\chisqlin, 2} & \text{if  $\frac{d_1}{s_1^2} \geq 1$,} \\ \Delta^{h_2'}_{\lin} & \text{otherwise.}
    \end{cases} \label{def_Delta_a_and_b}
\end{align}
We similarly aggregate the Bonferroni corrected linear and truncated $\chi^2$ tests as follows:
\begin{align*}
    &\Delta_c^{h_3,h_4} = \begin{cases}
        \Delta^{h_3}_{\chisqmax, 1} & \text{if  $\frac{d_2}{s_2^2}\log\!\left(e{d_1 \choose s_1}\right) \geq 1$,} \\ \Delta^{h_4}_{\linmax, 1} & \text{otherwise}
    \end{cases} \quad \text{ and } 
    \quad \Delta_d^{h_3',h_4'} = \begin{cases}
        \Delta^{h_3'}_{\chisqmax, 2} & \text{if  $\frac{d_1}{s_1^2}\log\!\left(e{d_2 \choose s_2}\right) \geq 1$,} \\ \Delta^{h_4'}_{\linmax, 2} & \text{otherwise}.
    \end{cases} 
\end{align*}
The optimal test consists of rejecting $H_0$ whenever one of the tests defined above rejects $H_0$
\begin{equation}\label{def_opt_test}
    \Delta^*_0 =  \Delta_a^{h_1, h_2} \lor \Delta_b^{h'_1, h'_2} \lor \Delta_c^{h_3, h_4} \lor  \Delta_d^{h_3', h_4'}.
\end{equation}
However, we will prove a stronger result, highlighting which test rejects the null hypothesis in any specific regime. To do so, we define
\begin{equation}\label{eq_def_beta}
    \beta(s_1, s_2, d_1, d_2) = \frac{1}{s_1 s_2}\log\left({d_2 \choose s_2}\right)\one{\left\{\frac{d_1}{s_1^2}\log\left(e{d_2 \choose s_2}\right) > 1\right\}},
\end{equation}
and let $\beta_{12} = \beta(s_1, s_2, d_1, d_2)$ and $\beta_{21} = \beta(s_2, s_1, d_2, d_1)$ for concision. We also define
\begin{equation}
    \tilde{R} := \tilde{R}(s_1, s_2, d_1, d_2) = \big(\psi_{12} + \beta_{21}\big) \land \big(\psi_{21} + \beta_{12}\big) \land \phi_{12} \land \phi_{21}.
\end{equation}
Our final test $\Delta^*$ is constructed according to which term dominates in the expression $\tilde{R}$:
\begin{equation}\label{eq:delta-optimal}
    \Delta^* = \begin{cases}
        \Delta_a^{h_1, h_2} & \text{ if } \tilde R = \phi_{12}, \\
        \Delta_b^{h'_1, h'_2} & \text{ if } \tilde R = \phi_{21}, \\
        \Delta_c^{h_3, h_4} & \text{ if } \tilde R = \psi_{12} + \beta_{21},\\
        \Delta_d^{h_3', h_4'} & \text{ if } \tilde R = \psi_{21} + \beta_{12}.
    \end{cases}
\end{equation}
\begin{proposition}\label{prop_ub}
    Let $\eta \in (0,1)$ be given. There exist constants $C, C_{h_1}, C_{h'_1}, C_{h_2}, C_{h_3}, C_{h_3'}, C_{h_4}, C_{h_4'} > 0$ such that if
    \[\mu^2 \geq C\big[(\psi_{12} + \psi_{21}) \land \phi_{12} \land \phi_{21}\big],\]
    then the test $\Delta^*$ implemented with the cutoffs
    \begin{align*}
        h_1 &= C_{h_1}s_2\log\big(1 + \frac{d_2}{s_2^2}\big) \\
        h'_1 &= C_{h'_1}s_1\log\big(1 + \frac{d_1}{s_1^2}\big) \\
        h_2 &= h'_2 = C_{h_2} \\
        h_3 &= C_{h_3}\Big(s_2\log\Big(1 + \frac{d_2}{s_2^2}\log \Big(e{d_1 \choose s_1}\Big)\Big) + \log \Big(e {d_1 \choose s_1}\Big)\Big) \\
        h'_3 &= C_{h_3'}\Big(s_1\log\Big(1 + \frac{d_1}{s_1^2}\log \Big(e{d_2 \choose s_2}\Big)\Big) + \log \Big(e{d_2 \choose s_2}\Big)\Big)\\
        h_4 &=  C_{h_4}\sqrt{\log\left(e {d_1 \choose s_1}\right)} \\
        h'_4 &=  C_{h_4'}\sqrt{\log\left(e {d_2 \choose s_2}\right)},
    \end{align*}
    satisfies
    \[\cR(\Delta^*, \mu) \leq \eta.\]
\end{proposition}
Proposition \ref{prop_ub} establishes the upper bound in Theorem \ref{thm_main}. We provide the proof in the supplement. Our optimal test $\Delta^*$ is constructed from four constituent tests: the linear test, the truncated $\chi^2$ test, and the Bonferroni corrected linear and truncated $\chi^2$ tests. The linear test is elementary and was studied in \citet{butucea2013detection}, and the Bonferroni corrected linear test is a straightforward extension used to handle certain regimes of imbalanced sparsity. The truncated $\chi^2$ and Bonferroni corrected truncated $\chi^2$ tests are, to our knowledge, entirely novel to the submatrix detection literature. We provide additional discussion of these tests in Section \ref{sec_discussion}.

\subsection{Adaptivity to unknown sparsity}

The definition of our optimal test $\Delta^*$ requires knowledge of the row sparsity $s_1$ and column sparsity  $s_2$, which is typically unavailable in practice. Here, we construct a test that achieves the minimax rate and is adaptive to unknown sparsity using a straightforward scanning procedure.

For any $(s_1, s_2) \in [d_1] \times [d_2]$, we will let $\tau_{\chisqlin, 1}(s_2)$ and $\tau_{\chisqlin, 2}(s_1)$ simply denote $\tau_{\chisqlin, 1}$ and $\tau_{\chisqlin, 2}$ with the dependence on the sparsity level made clear. We also denote the sparsity-dependent truncated $\chi^2$ test statistics $t_{\chisqlin, 1}(s_2)$ and $t_{\chisqlin, 2}(s_1)$ and tests  $\Delta^{h_1}_{\chisqlin, 1}(s_2)$ and $\Delta^{h'_1}_{\chisqlin, 2}(s_1)$, as well as the sparsity dependent Bonferroni corrected linear test statistics $t_{\linmax, 1}(s_1)$ $t_{\linmax, 2}(s_2)$ and tests $\Delta^{h_4}_{\linmax, 1}(s_1)$ and $\Delta^{h'_4}_{\linmax, 2}(s_2)$, in the same manner. We define the adjusted threshold value
\begin{equation}
    \tau_{\chisqmax, 1}(s_1, s_2) = \sqrt{C\log\left(1 + \frac{d_2}{s_2^2}\log \left({d_1 \choose s_1}\log_2(d_1)\log_2(d_2)\right)\right)}
\end{equation}
for a constant $C > 0$ to be chosen later. We then define $t_{\chisqmax, 1}(s_1, s_2)$ according to \eqref{eq_trunc_max} with this newly defined threshold, and let $\Delta^{h_3}_{\chisqmax, 1}(s_1, s_2) = \one(t_{\chisqmax, 1}(s_1, s_2) > h_3)$. We define $\tau_{\chisqmax, 2}(s_1, s_2)$, $t_{\chisqmax, 2}(s_1, s_2)$, and $\Delta^{h'_3}_{\chisqmax, 2}(s_1, s_2)$ by simply swapping the roles of $s_1$ and $s_2$ and $d_1$ and $d_2$ as before. Note that we do not need to introduce any additional notation to handle the linear test, as it does not require knowledge of the sparsity level. With these constituent tests well-defined, we let $\Delta^*(s_1, s_2)$ denote the aggregate test constructed according to \eqref{eq:delta-optimal} with the dependence on sparsity made explicit, and finally, we let $\Omega$ denote a dyadic partition of the set $[d_1] \times [d_2]$:
\begin{equation}
    \Omega = \Bigg\{\Big(\frac{d_1}{2^{m_1}}, \frac{d_2}{2^{m_2}}\Big): m_1 \in [\log_2(d_1)], m_2 \in [\log_2(d_2)]  \Bigg\}.
\end{equation}
Our proposed optimal adaptive test is defined as follows:
\begin{equation}
    \Delta^{*, \text{ada}} = \max_{(s_1, s_2) \in \Omega}\Delta^*(s_1, s_2).
\end{equation}
In the statement of the following proposition, we let 
\begin{align*}
    &\phi_{12} = \phi(s^*_1,s^*_2, d_1,d_2),
    &\phi_{21} = \phi(s^*_2,s^*_1, d_2,d_1),\,\\
    &\psi_{12} = \psi(s^*_1,s^*_2, d_1,d_2),
    &\psi_{21} = \psi(s^*_2,s^*_1, d_2,d_1),
\end{align*}
where $s_1^*$ and $s_2^*$ denote the true row and column sparsities under the alternative hypothesis respectively.
\begin{proposition}\label{prop_ada_ub}
    Let $\eta \in (0,1)$ be given. Suppose that $d_1 \land d_2 \geq 8$, $s^*_1 \land s_2^* \geq 3$, and that there exists a constant $C' > 0$ such that $s_i^* + \log \log d_j \geq C' \log \log d_i$ for $i, j \in \{1,2\}, i \neq j$. Then there exists a constant $C > 0$ and cutoffs $h_1, h_1', h_2, h_3,h_3', h_4, h_4' > 0$ such that if
    \[\mu^2 \geq C\big[(\psi_{12} + \psi_{21}) \land \phi_{12} \land \phi_{21}\big],\]
    then the test $\Delta^{*, \text{ada}}$ implemented with $h_1, h_1', h_2, h_3, h_3', h_4$ and $h_4'$ satisfies
    \[\cR(\Delta^{*, \text{ada}}, \mu) \leq \eta.\]
\end{proposition}
We defer explication of the cutoffs $h_1, h_1', h_2, h_3, h_3', h_4$ and $h_4'$ to the supplementary material. Proposition \ref{prop_ada_ub} establishes that minimax optimal submatrix detection is possible without knowledge of the row and column sparsity levels $s_1^*$ and $s_2^*$, under the assumptions $d_1 \land d_2 \geq 8$, $s^*_1 \land s_2^* \geq 3$ and $s_i^* + \log \log d_{j} \gtrsim\log \log d_i$ for $i \neq j$ with $i,j \in \{1,2\}$. While $s_i^* + \log \log d_j \gtrsim\log \log d_i$ is a weak condition on the shapes of the observed matrix and planted submatrix, the determination of whether cost-free adaptation is possible without this condition is an outstanding open problem that we leave to future work.

\section{Comparison with existing results}\label{sec_comparison}
\subsection{Relation to \citet{butucea2013detection} and subsequent work}

    Theorem \ref{thm_main} recovers (up to absolute constants) the results of \cite{butucea2013detection}. We recall the asymptotic setting adopted by \cite{butucea2013detection} here for convenience. They assert that $s_1, s_2, d_1,$ and $d_2$ all tend to infinity with $s_1 = o(d_1)$ and $s_2 = o(d_2)$. Furthermore, their lower bound on the minimax rate (Theorem 2.2 in \cite{butucea2013detection}) requires the following additional assumptions:
    \begin{enumerate}[(a)]
        \item $\frac{\log \log\big( \frac{d_1}{s_1}\big)}{\log \big(\frac{d_2}{s_2}\big)} \lor \frac{\log \log\big( \frac{d_2}{s_2}\big)}{\log \big(\frac{d_1}{s_1}\big)} \rightarrow 0$.
        \item $s_1 \log \big(\frac{d_1}{s_1}\big) \asymp s_2 \log \big(\frac{d_2}{s_2}\big)$.
    \end{enumerate}
    To fix our reference points for this discussion, we will refer to these conditions as Assumptions (a) and (b). Under these assumptions, \cite{butucea2013detection} prove that
    \begin{equation}\label{eq_bi_rate}
        (\mu^*)^2 = \left[\frac{d_1d_2}{s_1^2s_2^2} \land 2 \Bigg(\frac{1}{s_2}\log\left(\frac{d_1}{s_1}\right) + \frac{1}{s_1}\log\left(\frac{d_2}{s_2}\right)\Bigg)\right](1+o(1)).
    \end{equation}
    In comparison, we derive the following corollary from Theorem \ref{thm_main}.
    \begin{corollary}\label{cor_match_bi}
        Suppose that the following conditions hold:
        \begin{enumerate}
        \item $e \leq \frac{d_1}{s_1} \land \frac{d_2}{s_2}$
        \item $\frac{\log \log\big( \frac{d_1}{s_1}\big)}{\log \big(\frac{d_2}{s_2}\big)} \lor \frac{\log \log\big( \frac{d_2}{s_2}\big)}{\log \big(\frac{d_1}{s_1}\big)} \leq 1$.
        \item $s_1 \log \big(\frac{d_1}{s_1}\big) \asymp s_2 \log \big(\frac{d_2}{s_2}\big)$.
    \end{enumerate}
        Then the minimax rate of separation $\mu^*$ satisfies
        \begin{equation}
            (\mu^*)^2 \asymp \frac{d_1d_2}{s_1^2s_2^2} \land \Bigg(\frac{1}{s_2}\log\left(\frac{d_1}{s_1}\right) + \frac{1}{s_1}\log\left(\frac{d_2}{s_2}\right)\Bigg).
        \end{equation}
    \end{corollary}    
    This corollary demonstrates that our rate~\eqref{eq_our_rate} agrees with~\eqref{eq_bi_rate} in the setting of \cite{butucea2013detection}. 
    For comparison, their result pins down the sharp asymptotic constant under Assumptions $(a)$ and $(b)$ above, whereas our results are only expressed up to universal multiplicative constants. 
    However, we emphasize that their result requires Assumption (b), namely $s_1 \log(d_1/s_1) \asymp s_2 \log(d_2/s_2)$, which imposes the strong restriction that the hidden submatrix has comparable number of rows and columns up to logarithmic factors. 
    In contrast, our theorem covers any configuration of $s_1, s_2, d_1,$ and $d_2$ and provides entirely non-asymptotic upper and lower bounds presented in Theorem \ref{thm_main}. We are thus able to characterize the statistical hardness of submatrix detection in the \textit{imbalanced submatrix} regime (see Section \ref{sec_imbalanced}) and capture the precise behavior of the detection boundary around $\frac{d_1}{s_1^2}, \frac{d_2}{s_2^2} \asymp 1$ and $\frac{d_2}{s_2^2} \log(e{d_1 \choose s_1}), \; \frac{d_1}{s_1^2} \log(e{d_2 \choose s_2}) \asymp 1$. This degree of precision allows us to derive new phase transitions in the minimax rate that, to our knowledge, have not been described anywhere in the submatrix detection literature until now. To illustrate this in a special case, Corollary \ref{cor_pt} presents an expression of $\mu^*$ derived from Theorem \ref{thm_main} that notably does not follow from \eqref{eq_bi_rate}
    \begin{corollary}\label{cor_pt}
        Suppose that $s_1^2 \geq \bar{c}d_1 s_2$ for a constant $\bar{c} > 0$ and $s_j \leq c_j d_j$ for $j \in \{1, 2\}$ where $c_1,c_2 > 0$ are sufficiently small constants. Additionally, that $\frac{d_1}{s_1} \geq e \log(\frac{d_2}{s_2})$ and that there exists a constant $\alpha > 0$ such that $d_2 \geq s_2^{2 + \alpha}$. Then it holds
\begin{equation}\label{eq_pt}
    (\mu^*)^2 \asymp \frac{1}{s_2}\log\left(1 + \frac{d_1s_2}{s_1^2}\log(d_2)\right).
\end{equation}
    \end{corollary}
    In particular,~\eqref{eq_pt} reveals a phase transition at $s_1^2 \asymp d_1s_2\log(d_2)$ which is left obscured by the results of all prior work \citep{butucea2013detection, ma2015computational, luo2022tensor, dadon2024detection}. The case $s_1^2 > d_1s_2\log(d_2)$, in which we can linearize the right-hand side of \eqref{eq_pt}, is of particular interest as it contradicts Assumption (b) of \citet{butucea2013detection}. We provide additional discussion in the sequel.

    \subsection{Imbalanced submatrix regime}\label{sec_imbalanced}

    Assumption (b) of \citet{butucea2013detection} asserts that $s_1 \log(d_1/s_1) \asymp s_2 \log(d_2/s_2)$, which limits the scope of the their results to the \textit{balanced submatrix} setting in which the number of rows $s_1$ and columns $s_2$ must be of the same order up to a logarithmic term. As our Theorem \ref{thm_main} is free of any analogous balancedness assumption, we are able to characterize the minimax rate of submatrix detection across all regimes of the shape of the planted submatrix, most notably in any \textit{imbalanced submatrix} regime in which Assumption (b) is violated. In Corollary \ref{cor_simplify_newregime}, we simplify \eqref{eq_our_rate} in a regime that is explicitly excluded by Assumption (b) of \cite{butucea2013detection}.
\begin{corollary}\label{cor_simplify_newregime}
        Suppose that $s_1 = s_2^2$ and $d_1 = d_2^2$, and that there exists a constant $\alpha > 0$ such that $d_1 \geq s_1^{2 + \alpha}$. Then it holds
       \begin{equation}\label{eq_newregime_rate}
            (\mu^*)^2 \asymp \frac{\log \left(d_1\right)}{\sqrt{s_1}}.
        \end{equation}
    \end{corollary}
    Lifting ourselves into an asymptotic setting with $d_1, d_2, s_1, s_2 \rightarrow \infty$, the conditions $s_1 = s_2^2$ and $d_1 = d_2^2$ of Corollary \ref{cor_simplify_newregime} imply that $s_2\log(d_2/s_2) = o(s_1\log(d_1/s_1))$, which places us outside the balanced submatrix regime considered by \cite{butucea2013detection}. We remark that the rate \eqref{eq_newregime_rate} is in fact achieved by the scan test of \citet{butucea2013detection} under the conditions of Corollary \ref{cor_simplify_newregime}. However, without Assumption (b), the lower bound of \citet{butucea2013detection} (Theorem 2.2) is loose, and they are thus not able to recover the form of the minimax rate given in \eqref{eq_newregime_rate}.

    Furthermore, violations of Assumption (b) can lead to cases in which the results of \citet{butucea2013detection} are explicitly suboptimal. The following corollary, which follows directly from Corollary \ref{cor_pt}, leads us towards such a case.
    
    \begin{corollary}\label{cor_suboptimal_regime}
        Suppose that $s_1^2 > d_1 s_2 \log(d_2)$ and $s_j \leq c_j d_j$ for $j \in \{1, 2\}$ where $c_1,c_2 > 0$ are sufficiently small constants. Additionally, suppose that that $\frac{d_1}{s_1} \geq e \log(\frac{d_2}{s_2})$ and that there exists a constant $\alpha > 0$ such that $d_2 \geq s_2^{2 + \alpha}$. Then it holds
\begin{equation}\label{eq_suboptimal_regime}
    (\mu^*)^2 \asymp \frac{d_1}{s_1^2}\log(d_2).
\end{equation}
    \end{corollary}
    Here, \eqref{eq_suboptimal_regime} follows from simply linearizing the logarithm on the right-hand side of \eqref{eq_pt}. For the sake of discussion, we again place ourselves in an asymptotic setting with $d_1, d_2, s_1, s_2 \rightarrow \infty$ and $s_1 = o(d_1), s_2 = o(d_2)$. In this case, the condition $s_1^2 > d_1 s_2 \log(d_2)$ implies a violation of Assumption (b). To illustrate this, suppose that $s_1^2 > d_1 s_2 \log(d_2)$ and that $s_1\log(d_1/s_1) \asymp s_2 \log(d_2/s_2)$. Then it holds
    \begin{align}
        s_2 \log(d_2/s_2) &\asymp s_1\log(d_1/s_1) \\
        &> \sqrt{d_1s_2\log(d_2/s_2)} \log(d_1/s_1),
    \end{align}
    which implies
    \begin{equation}\label{eq_suboptimal_implication}
         \sqrt{s_2 \log(d_2/s_2)} > \sqrt{d_1}\log(d_1/s_1).
    \end{equation}
    Since  $s_1\log(d_1/s_1) \asymp s_2 \log(d_2/s_2)$ by assumption, \eqref{eq_suboptimal_implication} implies that $\sqrt{s_1 \log(d_1/s_1)} > \sqrt{d_1}\log(d_1/s_1)$, which clearly admits a contradiction in the setting $s_1 = o(d_1)$. Moreover, under the stated conditions it holds
    \begin{equation}
        (\mu^*)^2 \asymp \frac{d_1}{s_1^2}\log(d_2) < \frac{1}{s_2} \leq \frac{1}{s_2}\log\left(\frac{d_1}{s_1}\right) + \frac{1}{s_1}\log\left(\frac{d_2}{s_2}\right),
    \end{equation}
    where the final expression on the right-hand side is precisely the rate (i.e., the minimal signal strength necessary for consistent detection) attained by the scan test of \citet{butucea2013detection} (Theorem 2.1). If we additionally assume that $s_2^{\alpha} > \log(d_2)$ where $\alpha$ is obtained from Corollary \ref{cor_suboptimal_regime}, it holds
    \begin{equation}
        (\mu^*)^2 \asymp \frac{d_1}{s_1^2}\log(d_2) < \frac{d_1}{s_1^2}s_2^{\alpha} \leq \frac{d_1d_2}{s_1^2s_2^2},
    \end{equation}
    where here the expression on the right-hand side is the rate achieved by the linear test. These calculations suggest that the minimax rate established by \citet{butucea2013detection} under Assumptions (a) and (b) may in fact be strictly greater than the true minimax rate in the setting of Corollary \ref{cor_suboptimal_regime}. We formalize this conclusion in the following proposition. Here, we use  
\begin{equation}
    \mu^2_{\text{BI}} =  \frac{d_1d_2}{s_1^2s_2^2} \land 2 \Bigg(\frac{1}{s_2}\log\left(\frac{d_1}{s_1}\right) + \frac{1}{s_1}\log\left(\frac{d_2}{s_2}\right)\Bigg),
\end{equation}
to denote the rate of \citet{butucea2013detection}.

    \begin{proposition}\label{prop_suboptimal}
        Suppose that $d_1, d_2, s_1, s_2 \rightarrow \infty$ and $s_1 = o(d_1), s_2 = o(d_2)$. Furthermore, suppose that the conditions of Corollary 4 hold, and that $\log(d_2) = o(s_2^{\alpha})$. Then it holds
        \[\frac{\mu^*}{\mu_{\text{BI}}} \rightarrow 0.\]
    \end{proposition}
    Proposition \ref{prop_suboptimal} encapsulates one of the primary consequences of our Theorem \ref{thm_main}; namely, that there exist regimes in which the prior state-of-the-art theoretical results for submatrix detection are suboptimal with respect to the true minimax rate of detection. The reliance of \citet{butucea2013detection} on Assumptions (a) and (b) leads them to paint an incomplete picture of the detection boundary, and analogous assumptions similarly constrain all subsequent works \citep{arias2014community,verzelen2015community,ma2015computational, luo2022tensor, dadon2024detection, oren2026inhomogeneous}. The generality of Theorem \ref{thm_main} allows us to completely settle the minimax rate of submatrix detection, and Proposition \ref{prop_ub} ensures that we are able to achieve this rate with a minimax optimal test for any values of $s_1, s_2, d_1,$ and $d_2$.
\subsection{Case $s_1 = 1$}\label{subsubsec_s1=1}

To further illustrate that our results are free from any balancedness condition, we evaluate our rate in the case where the hidden submatrix is the most imbalanced, that is, when taking $s_1 = 1$. 
We also allow the ambient matrix to have arbitrarily many rows ($d_1$) and columns $(d_2)$. 
The case $s_1 = 1$ is known to be closely related to testing in $\ell_\infty$ norm (see e.g.~\cite{chhor2024sparse,kotekal2026locally}), and also provides some helpful insight into our results. 
Our rate can be shown to simplify as in Table~\ref{table_s1=1} below.
\begin{table}[h]
\centering
\begin{tabular}{|c|c|c|c|}
\hline
\textbf{ $\log (ed_1)$ vs $s_2$ } & \textbf{dense or sparse} & \textbf{rate} & \textbf{ optimal test } \\ \hline
$\log (ed_1) > s_2 \phantom{\Bigg|}$ & any & $\displaystyle(\mu^*)^2\asymp\log\Big(\frac{ed_2}{s_2}\Big) + \frac{\log(ed_1)}{s_2}$ & $\Delta_{\chisqmax, 1}^{h_3}$ \\ \hline
$\log (ed_1) < s_2\phantom{\Bigg|}$ & $\displaystyle\frac{d_2 \log (ed_1)}{s_2^2} \leq 1$ (dense) & $\displaystyle \hspace{-1mm}(\mu^*)^2 \, \asymp \, \frac{d_2}{s_2^2} \, \log (ed_1)$ & $\Delta^{h_1'}_{\maxlin, 1}$ \\
\hline
$\log (ed_1) < s_2\phantom{\Bigg|}$ & $\displaystyle\frac{d_2 \log (ed_1)}{s_2^2} > 1$ (sparse) & $\displaystyle(\mu^*)^2\asymp\log\!\Big(1\hspace{-.5mm}+\hspace{-.5mm} \frac{d_2\log(ed_1)}{s_2^2}\Big)$ & $\Delta_{\chisqmax, 1}^{h_3}$ \\ \hline
\end{tabular}
\caption{Regimes of separation when $s_1 = 1$. The proof is deferred to the supplement.}
\label{table_s1=1}
\end{table}

In this case, the mean matrix $\mathbf X \in \R^{d_1 \times d_2}$ should be viewed as a collection of $d_1$ row vectors of size $d_2$, among which at most \textit{one} is $s_2$-sparse and all others are zero. 
The problem is therefore to detect exactly one such vector given a noisy observation $\mathbf Y$, which is equivalent to the following multiple testing problem
\begin{align}
\begin{array}{l}
     H_{0}^{(i)}: \,\big(\mathbf X_{i j}\big)_j = 0_{\R^{d_2}} \qquad \text{ against }\\[10pt]
     H_{1}^{(i)}: \exists S \in \mathcal{P}_{s_2}(d_2) ~\text{ s.t. } \Bigg\{\begin{array}{ll}
        \!\!\forall j \in S: \, \mathbf X_{i j} \geq \mu,\\[-2pt]
        \!\!\forall j \notin S: \, \mathbf X_{i j} = 0 ,
    \end{array} 
\end{array}
       \qquad \forall i = 1,\dots, d_1. \label{eq_multiple_testing}
\end{align}
With high probability under $H_0$, all null hypotheses $H_0^{(i)}$ should be accepted. 
Hence, a natural procedure is a Bonferroni correction of the optimal procedures for each decision problem $H_0^{(i)}$ versus $H_1^{(i)}$. 
In Table~\ref{table_comparison_tests_s1=1} below, we summarize the optimal tests for any individual testing problem $H_0^{(i)}$ against $H_1^{(i)}$ and for the joint testing problem $H_0$ against $H_1$.
\begin{table}[h]
\centering
\resizebox{1\textwidth}{!}{
\begin{tabular}{|c|c|c|}
\hline
 $\phantom{\Bigg|}$& \textbf{$H_0^{(i)}$ against $H_1^{(i)}$} & \textbf{$H_0$ against $H_1$} \\ \hline
$\substack{\displaystyle \text{Optimal}\\[3pt]
\displaystyle \text{rejection}\\[3pt]
\displaystyle\text{rule}}$ $\phantom{\huge{\Bigg|}}\hspace{-1.95mm}$ & ${\Large\Bigg\{}\begin{array}{ll}
    \displaystyle\sum_{j=1}^{d_2} \big(\mathbf Y_{ij}^2 \!-\! \alpha\big)\mathds 1_{\scalebox{.8}{$|\mathbf Y_{ij}|\!>\!\lambda$}\!} \,\geq\, h_{s} & \text{if } s_2^2 < d_2\\[5pt]
    \displaystyle \sum_{j=1}^{d_2} \mathbf Y_{ij} \geq h_d & \text{otherwise}
\end{array}$ & ${\Large\Bigg\{}\begin{array}{ll}
        \!\!\displaystyle\max_{i \in [d_1]} \, \sum_{j = 1}^{d_2}\!\big(\mathbf Y_{ij}^2 \!-\! \tilde \alpha\big)\mathds 1_{\scalebox{.8}{$|\mathbf Y_{ij}|\!>\!\widetilde \lambda$}} \,\geq\, \tilde h_s & \text{if } s_2^2 \hspace{-.3mm}<\hspace{-.3mm} d_{(\text{eff}\hspace{.3mm})} \\[1pt]
        \!\!\displaystyle\max_{i \in [d_1]} \, \sum_{j = 1}^{d_2}\mathbf Y_{ij} \geq \tilde h_d & \text{otherwise}
    \end{array}$\\[15pt] \hline
$\substack{\displaystyle\text{Effective}\\[4pt]\displaystyle \text{dimension}}$ $\phantom{\huge{\Big|}}$& $d_2$ & $d_{(\text{eff}\hspace{.3mm})} = d_2 \log(ed_1)$\\[12pt] \hline
$\substack{\displaystyle \text{Truncation}\\[4pt] \displaystyle\text{thresholds}}$ $\phantom{{\Large\Bigg|}}\hspace{-1.95mm}$ & $\displaystyle\lambda = C\Big[\log\!\Big(1+\frac{d_2}{s_2^2}\Big)\,\Big]^{1/2}$ & $\displaystyle \hspace{-1mm}\tilde \lambda = C\left[\log\Big(1 + \frac{d_{(\text{eff}\hspace{.3mm})}}{s_2^2}\Big)\right]^{1/2}$ \\
\hline
$\substack{\displaystyle \text{Centering}\\[4pt] \displaystyle\text{parameters}}$ $\phantom{{\Large\Bigg|}}\hspace{-1.95mm}$ & $\displaystyle\alpha = \mathbb E\big(X^2 \big| \,|X|>\lambda\big)$ ~for~ $X \sim N(0,1)$ & $\displaystyle\tilde \alpha = \mathbb E\big(X^2 \big| \,|X|>\tilde \lambda\big)$ ~for~ $X \sim N(0,1)$ \\
\hline
$\substack{\displaystyle \text{Detection}\\[4pt] \displaystyle\text{thresholds}}$ $\phantom{\huge{\bigg|}}\hspace{-1.95mm}$& $\Bigg\{\begin{array}{ll}
    \displaystyle h_s = C_\eta \, s_2\log\!\Big(1+\frac{d_2}{s_2^2}\Big)\\
    \displaystyle h_d = C_\eta \, d_2^{1/2}
\end{array}$ & $\Bigg\{\begin{array}{ll}
    \displaystyle \tilde h_s = C_\eta \, s_2\log\!\Big(1 + \frac{d_{(\text{eff}\hspace{.3mm})}}{s_2^2}\Big) + C_\eta\log (e d_1)\\
    \displaystyle \tilde h_d = C_\eta \big[d_{(\text{eff}\hspace{.3mm})}\big]^{1/2}.
\end{array}$ \\[20pt] 
\hline
\end{tabular}
}
\caption{Optimal tests for $H_0^{(i)}$ against $H_1^{(i)}$ and for the joint testing problem $H_0$ against $H_1$.}
\label{table_comparison_tests_s1=1}
\end{table}

This comparison showcases a formal analogy between the two tests, but highlights that the truncation threshold $\tilde \lambda$ as well as the rejection thresholds $\tilde h_s, \tilde h_d$ must be carefully adjusted by logarithmic factors to control the joint Type~1 error of the multiple testing problem~\eqref{eq_multiple_testing} simultaneously for all $i \in [d_1]$. 
We also note that the phase transition between the dense and sparse regimes occurs around $s_2^2 \asymp d_2 \log(ed_1)$ rather than $s_2^2 \asymp d_2$ as in the classical vector case. 
In this setting, the Bonferroni-corrected $\chi^2$ test is required whenever $s_2^2<d_2\log(ed_1)$ and $s_1 = 1$, which corresponds to the severely sparse case where the hidden submatrix has simultaneously few rows and few columns. 
This phenomenon also extends to the $s_1>1$ case and is further discussed in Section~\ref{sec_discussion_tests}.

Note that in the first line of Table~\ref{table_s1=1}, when $s_2 \leq \log(ed_1)$, the condition $s_2^2 \leq d_2 \log(ed_1)$ is always satisfied, meaning that the test $\Delta^{h_3}_{\chisqmax, 1}$ is optimal. 
Additionally, the rate in Table~\ref{table_s1=1} degenerates into 
\begin{align*}
    (\mu^*)^2 \asymp \frac{1}{s_1} \log\left(\frac{ed_2}{s_2}\right) + \frac{1}{s_2} \log\left(\frac{ed_1}{s_1}\right),
\end{align*}
which is known to be the detection rate of the scan test statistic employed in~\cite{butucea2013detection}. 
However, as argued in Proposition~\ref{prop_suboptimal}, a simple combination of this scan test and linear test is generally suboptimal, and more refined procedures, as developed in the present paper, are required to achieve minimax optimality.


\subsection{Connection with the vector case: $s_1 \asymp d_1$}\label{subsec_vector_case}

    Our results in the matrix case allow us to directly recover the minimax rates of sparse signal detection in the vector case \citep{collier2017minimax}. 
    Indeed, it suffices to evaluate our results for $d_1 = s_1 = 1$. According to Table~\ref{table_s1=1} above, we can directly deduce that 
    \begin{align*}
        (\mu^*)^2 \asymp \log\left(1+ \frac{d_2}{s_2^2}\right).
    \end{align*}
    More generally, the problem of submatrix detection can be reduced to the vector case whenever $s_1 \asymp d_1$ or $s_2 \asymp d_2$. 
    To illustrate this, assume for now that $s_1 = d_1$. Inside each column of the observed matrix, the entries are either i.i.d. with distribution $N(0,1)$ or i.i.d. with distribution $N(\mu, 1)$. 
    It is well-known that a sufficient statistic for $\theta \in \mathbb{R}$ based on $n$ i.i.d.\ observations from $N(\theta,1)$ is the sum of these $n$ variables. 
    Therefore, the row vector obtained as the column-wise sum of the entries of the observed matrix is a sufficient statistic for the mean matrix $\X$, which ensures that techniques from the vector case developed in~\cite{collier2017minimax} can be applied to obtain the minimax separation rate $\mu^*$.

%
%

\section{Discussion}\label{sec_discussion}

\subsection{Discussion on the proposed tests}\label{sec_discussion_tests}
Our test $\Delta^*$ is a careful combination of several testing procedures, which substantially depart from test statistics proposed in the literature for submatrix detection problems. 

\cite{butucea2013detection} (resp.~\cite{baraud2002non}) use a \textit{scan} test statistic, which which consists of scanning over all possible submatrices (resp. subvectors) of size $s_1 \times s_2$ (resp. of size $s$) and rejecting if one of them has entries summing to an unusually large value. 
Analogous procedures have been used in related community detection settings~\citep{arias2014community,rotenberg2024planted}. 
Notably, our results never make use of the scan test, and highlight that it can always be successfully replaced by a Bonferroni corrected linear or truncated $\chi^2$ test, with lower time complexity\footnote{ Computing $\Delta^{h_3}_{\chisqmax, 1}$ requires $O(s_1d_2{d_1 \choose s_1})$ operations, rather than $O(s_1s_2{d_1 \choose s_1}{d_2 \choose s_2})$ for the scan test. }. 

Our two novel tests $\Delta_{\chisqlin, 1}^h$ and $\Delta^{h_3}_{\chisqmax, 1}$ build on the truncated $\chi^2$ test recently developed in the Gaussian sequence model. 
Given a vector $X \in \R^d$, the truncated $\chi^2$ test statistic is given by
\begin{align*}
    T = \sum_{j=1}^d (X_j^2 - \nu_a) \one{\big(|X_j| \geq a\big)}
\end{align*}
for some $a>0$ and where $\nu_a = \mathbb E\left(Z^2 \big|\,|Z| \geq a\right)$. 
The corresponding test that rejects for large values of $T$ is known to be optimal for detecting sparse alternatives in the Gaussian sequence model~\citep{collier2017minimax}. 
To adapt this test to the matrix case where $\mathbf Y \in \R^{d_1 \times d_2}$ the truncated $\chi^2$ test, we use 
\begin{align*}
    \sum_{i=1}^{d_1} \Big(\Big(\sum_{j=1}^{d_2} \mathbf Y_{ij}\Big)^2 - \nu_a\Big)\one{\big(\sum_{j=1}^{d_2} \mathbf Y_{ij} \geq a\big)}
\end{align*}
for some suitable truncation threshold $a>0$ and re-centering parameter $\nu_a$. 
This test statistic can be successfully applied to achieve optimality in regimes where $\tilde R \asymp \phi_{21}$ and $s_1^2 < d_1$ (see~\eqref{eq:delta-optimal} and~\eqref{def_Delta_a_and_b}). 
However, this test can also become suboptimal in other important sparse regimes. 
Specifically, its main limitation is due to the terms $\sum_{j=1}^{d_2} \mathbf Y_{ij}$ for $i\in [d_1]$, which completely eliminate the sparse structure of the corresponding vectors $(Y_{ij})_{j}$, and thereby prevent us from leveraging sparsity along the columns. 
We overcome this difficulty by proposing a Bonferroni correction of tests statistics of the form
\begin{align*}
    \sum_{i=1}^{d_1} \Big(\Big(\sum_{j\in S} \mathbf Y_{ij}\Big)^2 - \nu_a\Big)\one{\big(\sum_{j\in S} \mathbf Y_{ij} \geq a\big)}, \qquad \forall S \in \mathcal{P}_{s_2}(d_2),
\end{align*}
which preserve the sparse structure along the columns.
The rest of this section is devoted to discussing the differences between these two truncated $\chi^2$ test statistics.

Heuristically, our truncated $\chi^2$ test statistic $t_{\chisqlin, 1}$ defined in~\eqref{eq_trunc_lin} is optimal when the submatrix has a large number of rows and a small number of columns. 
In contrast, the Bonferroni-corrected truncated $\chi^2$ test statistic $t_{\chisqmax, 1}$ defined in~\eqref{eq_trunc_max} is optimal when the number of rows and of columns are both small. 
To see this, assume first that $s_2 = d_2$ and $s_1 \ll d_1$ for the sake of discussion. 
We have justified in Section~\ref{subsec_vector_case} that the column vector obtained as the row-wise sum of $\mathbf Y$ is a sufficient statistic of the data for the mean matrix $\mathbf X$. 
Therefore, when the number of rows or columns is large enough, the problem can be reduced to a vector-based problem, for which the linear test or truncated chi-square test are known to be optimal depending on whether $s_1 \lessgtr \sqrt{d_1}$~\citep{collier2017minimax}. This justifies the construction of our test statistic $t_{\chisqlin, 1}$, which consists of summing row-wise or column-wise, and applying the linear test if $s_1 \geq \sqrt{d_1}$ or the vector-based truncated $\chi^2$ test otherwise.

When the hidden submatrix has few rows and columns, collapsing the observation matrix into a single vector (by summing over rows or columns) destroys too much information from the data. 
In this regime, a natural idea would be to apply the scan test proposed in~\cite{butucea2013detection}. Unfortunately, this procedure is not optimal in most regimes, generalizing the suboptimality reported in~\cite{baraud2002non} (see Proposition~\ref{prop_suboptimal}). 
Instead, we propose a non-trivial refinement based on a Bonferroni correction of carefully selected  truncated chi-square test statistics. 
Our procedure, the Bonferroni corrected truncated $\chi^2$ test as formally defined in equation (\ref{eq_trunc_max}), is obtained by scanning and summing over subsets along \textit{one} dimension of the matrix and truncating along the other dimension. 
In some regimes, this procedure fails and needs to be applied with the dimensions flipped, that is, by scanning over columns and truncating along rows. 
This procedure captures subtle concentration effects of elevated submatrices with extremely small number of rows and columns, achieving minimax optimality in such regimes. 


\subsection{Conclusion and future work}
In this paper, we have completely resolved the minimax rate of separation for detecting a planted submatrix in homoscedastic Gaussian noise. Our results reveal new phase transitions in the minimax rate that were previously unknown to the literature, and allow us to describe regimes in which the previous state-of-the-art results are in fact suboptimal. Our upper bound relies on the construction of new testing procedures, the ideas behind which we anticipate will lead to the development of new optimal procedures for similar detection and estimation problems in matrix- and tensor-structured data.

Our results present many opportunities for future research. We briefly discuss a few of these opportunities below.
\begin{itemize}
    \item \underline{Computational gaps}: The submatrix detection admits a well-known statistical-computation gap; that is, there exists a regime where detection is statistically possible but computationally infeasible \citep{ma2015computational}. However, the results of \citet{ma2015computational} and follow-up work \citep{brennan2019universality, brennan2018reducibility} hold only in the square matrix and submatrix regime (with $d_1 = d_2, s_1 = s_2$), and it is naturally of interest to investigate the existence of such gaps without this restriction. We employ the Bonferroni corrected linear and truncated $\chi^2$ tests whenever $(\mu^*) \asymp \psi_{12} + \psi_{21}$, which require maximizing over all subsets of $[d_1]$ (resp. $[d_2]$) of size $s_1$ (resp. $s_2$). These procedures can be intractable when ${d_1 \choose s_1}$ or ${d_2 \choose s_2}$ is exponential in one of the parameters, and we conjecture that constraining ourselves to the use of polynomial time testing procedures will lead to suboptimal statistical performance in this regime. Establishing this phenomenon formally is an important future research direction.
    \item \underline{Extension to tensors}: The work of \citet{luo2022tensor} in part extends the results of \citet{butucea2013detection} to the tensor setting, in which the statistician observes a multi-modal tensor in Gaussian noise and attempts to discern whether there exists a planted sub-tensor of elevated mean. Their results similarly require balancedness conditions on the size of the planted sub-tensor, adapted from \citet{butucea2013detection} to the multi-modal setting (see Theorem 9 in \citet{luo2022tensor} and its assumptions). Extending our Theorem \ref{thm_main} to the tensor setting is of great interest, and we conjecture that such a result would reveal new phenomena in the sub-tensor detection problem, in particular new phase transitions according to the relative shapes of the observed tensor and planted sub-tensor.
    \item \underline{Submatrix recovery}: In this paper, we have focused our attention on the submatrix detection problem, in which the statistician aims to detect the presence of a planted submatrix in noise. Submatrix recovery, where the goal is to estimate the support of the planted submatrix, is a statistically harder task that has also seen great research activity over the past decade \citep{hajek2018submatrix, cai2017computational,butucea2015sharp,banks2018information, hajek2016semidefinite, dadon2024detection, oren2026inhomogeneous}. Resolving the non-asymptotic minimax rate of submatrix support recovery for any configuration of $s_1, s_2, d_1$, and $d_2$ constitutes a valuable research direction, and any such result will likely rely on recent developments in the variable selection literature for the Gaussian sequence model \citep{butucea2018variable, ndaoud2020optimal, butucea2023variable}. We leave these investigations to future work.
\end{itemize}

\textbf{Acknowledgements:} This paper has been funded by the Agence Nationale de la Recherche under grant ANR-17-EURE-0010 (Investissements d'Avenir program). Parker Knight is support by an NSF Graduate Research Fellowship. The authors thank Rajarshi Mukherjee for drawing their attention to the submatrix detection problem. The authors would also like to thank the Rose Traveling Fellowship of the Harvard T.H. Chan School of Public Health for supporting an international visit between the authors. 

\bibliographystyle{apalike}
\bibliography{biblio}

\newpage

\appendix 

\begin{center}
    \LARGE{Supplementary Material}
\end{center}

\addtocontents{toc}{\protect\setcounter{tocdepth}{2}}
\tableofcontents

\section{Proof of lower bound}

In this section, we prove the lower bound on the minimax rate of separation $\mu^*$. Our general strategy is to lower bound the minimax risk of testing with the Bayes risk defined with respect to a well-chosen prior over the parameter space. We then invoke the Neyman-Pearson lemma and carefully control the second moment of the resulting likelihood ratio statistic. Throughout this section of the supplement, we will denote our desired minimax rate of separation as
\[R := (\psi_{12} + \psi_{21}) \land \phi_{12} \land \phi_{21}.\]

\bigskip \noindent Let $\eta \in (0,1)$ be given. Recall that the minimax risk is defined as
\begin{align*}
    \cR^*(s_1,s_2,d_1,d_2, \mu) &= \inf_\Delta \cR(\Delta, \mu) \\ 
    &= \inf_\Delta\Big\{\mathbb P_{0}(\Delta = 1) + \sup_{\X \in \Theta(s_1,s_2,d_1,d_2, \mu)} \mathbb P_\X(\Delta = 0)\Big\}.
\end{align*}
Define the reduced parameter space $\overline{\Theta}(s_1, s_2, d_1, d_2, \mu)$ as
\[\overline{\Theta}(s_1, s_2, d_1, d_2, \mu) = \{\X \in \Theta(s_1, s_2, d_1, d_2, \mu) : X_{ij} \neq 0 \implies X_{ij} = \mu\}.\]
We let $\pi$ denote the uniform distribution over $\overline{\Theta}(s_1, s_2, d_1, d_2, \mu)$, meaning that for any $\mathbf M \in \overline{\Theta}(s_1, s_2, d_1, d_2, \mu)$ it holds $\bbP_{\X \sim \pi}(\X = \mathbf M) = \frac{1}{{d_1 \choose s_1}{d_2 \choose s_2}}$. Using $\pi$ as our prior on $\overline{\Theta}(s_1, s_2, d_1, d_2, \mu)$, we define the mixture distribution $\mathbf{P}_\pi$ on $\R^{d_1\times d_2}$ as
\[\mathbf P_{\pi}(A)=\int_{\X \in \overline{\Theta}(s_1, s_2, d_1, d_2, \mu)}\bbP_\X(A)\pi(\D \X),\]
where $A$ is a measurable set. With this infrastructure in hand, we can lower bound the minimax risk as follows:
\begin{align*}
    \cR^*(s_1,s_2,d_1,d_2, \mu) &=\inf_\Delta\Big\{\mathbb P_{0}(\Delta = 1) + \sup_{\X \in \Theta(s_1,s_2,d_1,d_2, \mu)} \mathbb P_\X(\Delta = 0)\Big\} \\
    &\geq \inf_\Delta\Big\{\mathbb P_{0}(\Delta = 1) + \sup_{\X \in \overline{\Theta}(s_1,s_2,d_1,d_2, \mu)} \mathbb P_\X(\Delta = 0)\Big\} \\
    &\geq \inf_\Delta\Big\{\mathbb P_{0}(\Delta = 1) + \mathbf P_\pi(\Delta = 0)\Big\}.
\end{align*}
The final expression above is the minimax risk of a simple versus simple hypothesis testing problem, and we can characterize it precisely using the Neyman-Pearson lemma (\cite{rigollet2023high}, Lemma 4.3). Combining this result with standard equivalent formulations of the total variation distance (\cite{rigollet2023high}, Proposition 4.4), we have
\begin{align*}
    \inf_\Delta\Big\{\mathbb P_{0}(\Delta = 1) + \mathbf P_\pi(\Delta = 0)\Big\} &= 1 - \text{TV}(\bbP_0, \mathbf P_{\pi}) \\
    &= 1 - \frac12\int_{\Y \in \R^{d_1\times d_2}}|\D\bbP_0(\Y) - \D \mathbf P_\pi (\Y)| \\
    &= 1 - \frac12\int_{\Y \in \R^{d_1\times d_2}}\Big|\frac{\D\mathbf P_\pi (\Y)}{\D \bbP_0(\Y)} - 1\Big|\D \bbP_0(\Y),
\end{align*}
where $\D \bbP_0$ and $\D \mathbf P_\pi$ denote the Radon-Nikodym derivatives with respect to Lebesgue measure of $\bbP_0$ and $\mathbf P_\pi$ respectively. Letting $L_\pi = \frac{\D\mathbf P_\pi }{\D \bbP_0}$ denote the likelihood ratio, we apply the Cauchy-Schwarz inequality to obtain
\begin{align*}
    \int_{\Y \in \R^{d_1\times d_2}}\Big|\frac{\D\mathbf P_\pi (\Y)}{\D \bbP_0(\Y)} - 1\Big|\D \bbP_0(\Y) &= \int_{\Y \in \R^{d_1\times d_2}}\Big|L_\pi(\Y) - 1\Big|\D \bbP_0(\Y) \\
    &\leq \sqrt{\int_{\Y \in \R^{d_1\times d_2}}\Big(L_\pi(\Y) - 1\Big)^2\D \bbP_0(\Y)} \\
    &= \sqrt{\EE_0\big[L_\pi^2(\Y)\big] - 1},
\end{align*}
where we arrive at the final expression by expanding the square and using $\EE_0[L_\pi(\Y)] = 1$. This chain of calculations reveals
\[\cR^*(s_1, s_2, d_1, d_2) \geq 1 - \frac12\sqrt{\EE_0\big[L_\pi^2(\Y)\big] - 1}.\]
Therefore, to prove that $\cR^*(s_1, s_2, d_1, d_2) \geq 1 - \eta$, it suffices to show
\begin{align*}
    \EE_0\big[L_\pi^2\big] &\leq 1 + 4\eta^2 \\
    &= 1 + \eps,
\end{align*}
where we define $\eps := 4\eta^2$. By direct calculation, we have
\begin{align*}
    \EE_0\big[L_\pi^2\big] &= \int_{\Y}\frac{\big(\mathbf \D \mathbf P_\pi(\Y)\big)^2}{\D \bbP_0(\Y) } \\
    &= \int_{\Y}\frac{\int \D \bbP_\X(\Y)\pi(\D \X)\int \D \bbP_{\X'}(\Y)\pi(\D \X')}{\D \bbP_0(\Y) } \\
    &= \int_\X \int_{\X'}\Big[\int_\Y\frac{\D \bbP_\X(\Y)\D \bbP_{\X'}(\Y)}{\D \bbP_0(\Y)}\Big]\pi(\D \X)\pi(\D \X'), 
\end{align*}
where $\X$ and $\X'$ are two mean matrices drawn independently from $\pi$ with supports $S = S_1 \times S_2$ and $S' = S_1' \times S_2'$ respectively. Recall that the entries of $\Y$ are assumed to be independent Gaussian random variables with variance $1$. We can therefore write the above integrand as
\begin{align*}
    \int_\Y\frac{\D \bbP_\X(\Y)\D \bbP_{\X'}(\Y)}{\D \bbP_0(\Y)} &= \int_\Y\exp\Big[-\frac12\Big(\sum_{(i,j) \in S}(Y_{ij} - \mu)^2 +
    \sum_{(i,j) \notin S}Y_{ij}^2  \\
    &\quad \quad  + \sum_{(i,j) \in S'}(Y_{ij} - \mu)^2 + \sum_{(i,j) \notin S'}Y_{ij}^2 - \sum_{\substack{i \in [d_1], \\ j \in [d_2]}}Y_{ij}^2\Big)\Big] \\
    &= \int_\Y\exp\Big[-\frac12\Big(\sum_{\substack{i \in [d_1], \\ j \in [d_2]}}Y^2_{ij} + 2\sum_{(i,j \in )S \cap S'}\mu^2 + \sum_{(i,j) \in S \triangle S'}\mu^2 \\
    &\quad \quad - 4\sum_{(i,j) \in S \cap S'}Y_{ij}\mu - 2\sum_{(i,j) \in S \triangle S'}Y_{ij}\mu\Big)\Big] \\
    &= \int_\Y\exp\Big[-\frac12\Big(\sum_{(i,j) \in S \cap S'}(Y_{ij} - 2\mu)^2 + \sum_{(i,j) \in S \triangle S'}(Y_{ij} - \mu)^2 - 2|S \cap S'|\mu^2\Big)\Big] \\
    &= \exp\Big(\mu^2|S_1 \cap S_1'||S_2 \cap S'_2|\Big).
\end{align*}
It follows that
\begin{align*}
    \EE_0\big[L_\pi^2\big] &= \EE\Big[\exp\Big(\mu^2|S_1 \cap S_1'||S_2 \cap S'_2|\Big)\Big] \\
&=\EE[\exp(\mu^2W_1W_2)],
\end{align*}
where $W_j \sim \text{HyperGeometric}(d_j, s_j, s_j)$. By Lemma 3 of \citet{arias2011global}, each $W_j$ is stochastically dominated by a Bin$(d_j, \frac{s_j}{d_j - s_j})$ random variable, and therefore it suffices to show that
\begin{equation}\label{eq_mgf_goal}
    \EE[\exp(\mu^2XY)] \leq 1+\eps,
\end{equation}
where $X \sim \text{Bin}(s_1, \frac{s_1}{d_1 - s_1})$ and $Y \sim \text{Bin}(s_2, \frac{s_2}{d_2 - s_2})$. The remainder of the proof is structured as follows.
\begin{enumerate}
    \item In the \textit{general sparsity} setting, meaning that $s_1 \leq c_1 d_1$ and $s_2 \leq c_2 d_2$ for sufficiently small constants $c_1, c_2 > 0$, we show that there exists a constant $c_\mu > 0$ such that if $\mu^2 \leq c_\mu R$, then (\ref{eq_mgf_goal}) holds. The proof of this claim constitutes the primary technical difficulty of the derivation of our lower bound, and relies on a precise analysis of the moment generating function in the left hand side of (\ref{eq_mgf_goal}). The details are given in Section \ref{sec_pf_lb_general}, with the key lemmas collected in Section \ref{sec_lemmas_lb}.
    \item Otherwise, we place our selves in the \textit{very dense} setting and assume without loss of generality that $s_1 \geq c d_1$ for a constant $c \in (0,1)$. In this case, the problem roughly reduces to the sparse signal detection problem in a standard Gaussian sequence model. In Section \ref{sec_pf_lb_dense}, we prove that there exist constants $c_\mu,c' > 0$ such that if $\mu^2 \leq c_\mu \frac{d_1}{s_1^2}\log\big(1 + c'\frac{d_2}{s_2^2}\big)$, then (\ref{eq_mgf_goal}) holds. We will then show that $R \asymp \frac{d_1}{s_1^2}\log\big(1 + c'\frac{d_2}{s_2^2}\big)$ in this regime, which completes the proof.
\end{enumerate}
We remark that this proof strategy is similar to that employed in our companion work \citet{chhor2026optimal}. We state the proof in its entirety here for completeness, and to render the dependence of the results in Appendices \ref{app_ub} and \ref{app_additional_justifications} on the lemmas established in this section as clear as possible.
\subsection{Proof of lower bound in the general sparsity setting}\label{sec_pf_lb_general}
Throughout this proof, we assume that there exist sufficiently small constants $c_1, c_2 \in (0,1)$ which depend on $\eta$ such that $s_1 \leq c_1d_1$ and $s_2 \leq c_2d_2$. It this regime, it holds that $\frac{s_j}{d_j - s_j} \asymp \frac{s_j}{d_j}$, and therefore it suffices to control $\EE[\exp(\mu^2XY)]$ for $X \sim \text{Bin}(s_1, \frac{s_1}{d_1})$ and $Y \sim \text{Bin}(s_2, \frac{s_2}{d_2})$. We will also assume that $\frac{d_1}{s_1} \geq e \log(\frac{d_2}{s_2})$. In fact, this assumption may be made without loss of generality due to Lemma \ref{lem_2.6}. Recall that we aim to show that there exists a constant $c_\mu > 0$ such that if $\mu^2 \leq c_\mu R$, then $\EE[\exp(\mu^2XY)] \leq 1 + \eps$. We divide our analysis into two cases.

\bigskip 

\noindent\underline{Case 1:} Suppose that $s_1^2 \geq (2e)^{-4}d_1s_2$. Let $\Ct \geq 1$ and $c_{1, \mu} > 0$ be the constants obtained from applying Lemma \ref{lem_simplify_dense_allrates} with $\alpha = \eps/2$. We can write
\begin{align*}
\EE[\exp(\mu^2XY)] &= \EE[\exp(\mu^2XY)\one(X \leq \Ct \frac{s_1^2}{d_1})] + \EE[\exp(\mu^2XY)\one(X > \Ct \frac{s_1^2}{d_1})]
\end{align*}
By Lemma \ref{lem_simplify_dense_allrates}, if $\mu^2 \leq c_{1, \mu}R$, then $\EE[\exp(\mu^2XY)\one(X > \Ct \frac{s_1^2}{d_1})] < \eps/2$.
Now we turn our attention to the first term. Suppose that $\Ct \frac{s_1^2}{d_1} < 1$. Then
\begin{align*}
    \EE[\exp(\mu^2XY)\one(X \leq \Ct \frac{s_1^2}{d_1})] &= \bbP(X = 0) \\
    &\leq 1,
\end{align*}
and it holds that $\EE[\exp(\mu^2XY)] \leq 1+\eps/2$ and the proof is complete. Otherwise, suppose that $\Ct \frac{s_1^2}{d_1} \geq 1$. Then by Lemma \ref{lemma:lb-principal-term}, there exists a constant $c_{2, \mu} > 0$ such that if $\mu^2 \leq c_{2, \mu}R$, it holds $\EE[\exp(\mu^2XY)\one(X \leq \Ct \frac{s_1^2}{d_1})] < 1 + \eps/2$. Letting $c_\mu = \min(c_{1, \mu}, c_{2, \mu})$, it follows that if $\mu^2 \leq c_\mu R$, then
\[\EE[\exp(\mu^2XY)] \leq 1+\eps.\]
\bigskip 

\noindent \underline{Case 2:} Suppose that there exists a constant $\bar{c} \in (0, (2e)^{-4})$ such that $s_1^2 \leq \bar{c}d_1s_2$. We split our analysis into two sub-cases. Assume first that $s_2 < s_1\log\big(\frac{d_1s_2}{s_1^2}\big)$. For a constant $C_* \geq 1$ whose value will be determined later, we form the partition
\begin{align*}
    \EE[\exp(\mu^2XY)] &= \EE[\exp(\mu^2XY)\one(X \leq \Ct \frac{s_1^2}{d_1})] \\
    &\quad +  \EE[\exp(\mu^2XY)\one(\Ct \frac{s_1^2}{d_1} < X < s_2^{-1}\log\big(\frac{d_1s_2}{s_1^2}\big))] \\
    &\quad + \EE[\exp(\mu^2XY)\one(X \geq \Ct \frac{s_1^2}{d_1} \lor s_2^{-1}\log\big(\frac{d_1s_2}{s_1^2}\big) )] \\
    &= \text{I}(C_*) + \text{II}(C_*) + \text{III}(C_*).
\end{align*}
By Lemma \ref{lem_simplification_rate_max_test}, there exist constants $C_{1,*} \geq 1$ and $c_{1, \mu} > 0$ such that if $\mu^2 \leq c_{1, \mu}R$, then $\text{III}(C_{1,*}) < \eps/3$. Suppose that $\lceil C_{1,*}\frac{s_1^2}{d_1}\rceil \geq \lfloor s_2^{-1}\log \big(\frac{d_1s_2}{s_1^2}\big)\rfloor$. In this case, we take $C_* = C_{1,*}$, which gives us $\text{II}(C_*) = 0$. If $C_* \frac{s_1^2}{d_1} < 1$, it immediately follows that $\text{I}(C_*) \leq 1$ from an elementary calculation in the proof of Case 1, and thus $\EE[\exp(\mu^2XY)] \leq 1 + \eps/3$ which completes the proof. Otherwise, if $C_* \frac{s_1^2}{d_1} \geq 1$, then by Lemma \ref{lemma:lb-principal-term} there exists a constant $c_{2,\mu} > 0$ such that if $\mu^2 \leq c_{2,\mu}R$, then $\text{I}(C_*) \leq 1 + \frac23\eps$. In this case, we take $c_\mu = \min(c_{1, \mu}, c_{2, \mu})$ and for $\mu^2 \leq c_{\mu}R$ it holds that $\EE[\exp(\mu^2XY)]\leq 1 + \eps$. This completes the proof in the case $\lceil C_{1,*}\frac{s_1^2}{d_1}\rceil \geq \lfloor s_2^{-1}\log \big(\frac{d_1s_2}{s_1^2}\big)\rfloor$.

Now suppose that $\lceil C_{1,*}\frac{s_1^2}{d_1}\rceil \leq \lfloor s_2^{-1}\log \big(\frac{d_1s_2}{s_1^2}\big)\rfloor$. From here, we consider two further subcases. First, suppose that $2e\frac{s_1^2}{d_1} \geq \frac{s_2^2}{d_2}$. Then, by Lemma \ref{lem_simplify_s1d1big_truncchi2}, there exist constants $C_{2,*} \geq 1$ and $c_{2,\mu}$ such that if $\mu^2 \leq c_{2,\mu}R$, then $\text{II}(C_{2,*}) < \eps/3$. Now we take $C_* = \max(C_{1,*}, C_{2,*})$. If $C_* \frac{s_1^2}{d_1} < 1$, it immediately follows that $\text{I}(C_*) \leq 1$, and thus $\EE[\exp(\mu^2XY)] \leq 1 + \frac23\eps$ which completes the proof. Otherwise, if $C_* \frac{s_1^2}{d_1} \geq 1$, then by Lemma \ref{lemma:lb-principal-term} there exists a constant $c_{3,\mu} > 0$ such that if $\mu^2 \leq c_{3,\mu}R$, then $\text{I}(C_*) \leq 1 + \eps/3$. We then take $c_\mu = \min(c_{1, \mu}, c_{2, \mu}, c_{3,\mu})$ and for $\mu^2 \leq c_{\mu}R$ it holds that $\EE[\exp(\mu^2XY)]\leq 1 + \eps$. Next, suppose that $2e\frac{s_1^2}{d_1} \leq \frac{s_2^2}{d_2}$. In this case, we partition $\text{II}(C_*)$ as follows:
\begin{align*}
    \text{II}(C_*) &= \EE[\exp(\mu^2XY)\one(\Ct \frac{s_1^2}{d_1} < X < s_2^{-1}\log\big(\frac{d_1s_2}{s_1^2}\big))] \\
    &= \EE\left[\exp\big(\mu^2XY\big)\one\bigg(1 \vee \Ct \frac{s_1^2}{d_1} \leq X \leq \frac{\Ct s_2^2/d_2}{\log\big(\frac{d_1s_2^2}{s_1^2d_2}\big)}\wedge s_1\bigg)\right] \\
    &\quad +\EE\bigg[\exp(\mu^2XY)\one\bigg( 1 \vee \Ct \frac{s_1^2}{d_1} \vee  \frac{\Ct s_2^2/d_2}{\log\big(\frac{d_1s_2^2}{s_1^2d_2}\big)} \leq X \leq \frac{s_2}{\log\big(\frac{d_1s_2}{s_1^2}\big)}\bigg)\bigg] \\
    &= \text{II}^{(a)}(C_*) + \text{II}^{(b)}(C_*).
\end{align*}
By Lemmas \ref{lem_simplify_s2d2big_maxlin} and \ref{lem_simplify_s2d2big_truncchi2}, there exist constants $C_{2,*}, C_{3,*} \geq 1$ and $c_{2,\mu}, c_{3,\mu} > 0$ such that if $\mu^2 \leq \min(c_{2,\mu}, c_{3,\mu})R$, then $\text{II}^{(a)}(C_{2,*}) + \text{II}^{(b)}(C_{3,*}) < \eps/3$. We take $C_* = \max(C_{1,*}, C_{2,*}, C_{3,*})$. If $C_* \frac{s_1^2}{d_1} < 1$, then $\text{I}(C_*) \leq 1$ and thus $\EE[\exp(\mu^2XY)] \leq 1 + \frac23\eps$ which completes the proof. Otherwise, if $C_* \frac{s_1^2}{d_1} \geq 1$, then by Lemma \ref{lemma:lb-principal-term} there exists a constant $c_{4,\mu} > 0$ such that if $\mu^2 \leq c_{4,\mu}R$, then $\text{I}(C_*) \leq 1 + \eps/3$. We then take $c_\mu = \min(c_{1, \mu}, c_{2, \mu}, c_{3,\mu}, c_{4,\mu})$ and for $\mu^2 \leq c_{\mu}R$ it holds that $\EE[\exp(\mu^2XY)]\leq 1 + \eps$. This completes the proof in the case $s_2 < s_1 \log \big(\frac{d_1s_2}{s_1^2}\big)$.

We now turn our attention to the case $s_2 \geq s_1 \log \big(\frac{d_1s_2}{s_1^2}\big)$. In this case, we perform the partition
\begin{align*}
    \EE[\exp(\mu^2XY)] &= \EE[\exp(\mu^2XY)\one(X \leq \Ct \frac{s_1^2}{d_1})] \\
    &\quad +  \EE[\exp(\mu^2XY)\one(\Ct \frac{s_1^2}{d_1} < X \leq s_1\big))] \\
    &= \text{I}(C_*) + \text{II}(C_*)
\end{align*}
The remainder of the proof follows exactly as in the $s_2 < s_1 \log \big(\frac{d_1s_2}{s_1^2}\big)$ setting. If $2e\frac{s_1^2}{d_1} \geq \frac{s_2^2}{d_2}$, we control $\text{II}(C_*)$ using Lemma \ref{lem_simplify_s1d1big_truncchi2}; otherwise, we use Lemmas \ref{lem_simplify_s2d2big_maxlin} and \ref{lem_simplify_s2d2big_truncchi2}. We then control $\text{I}(C_*)$ using Lemma \ref{lemma:lb-principal-term} if needed. We omit the details for brevity. The proof of the lower bound in the general sparsity setting is complete.

\subsection{Proof of lower bound in the very dense setting}\label{sec_pf_lb_dense}

Now suppose that there exists a constant $c > 0$ such that $s_1 > cd_1$. Suppose that $\mu^2 \leq c\frac{d_1}{s_1^2}\log\big(1 + c'\frac{d_2}{s_2^2}\big)$ for $c' = \log(1 + \eps)$. By the definition of $c$, it immediately holds that
\[\mu^2 \leq c\frac{d_1}{s_1^2}\log\big(1 + c'\frac{d_2}{s_2^2}\big) \leq \frac{1}{s_1}\log\big(1 + c'\frac{d_2}{s_2^2}\big)\]
We divide our analysis into two simple cases. First, suppose that there exists a constant $\bar{c} < 1$ such that $s_2 < \bar{c}d_2$, which grants that $\frac{s_2}{d_2 - s_2} \asymp \frac{s_2}{d_2}$, and we may consider $Y \sim \text{Bin}(s_2, \frac{s_2}{d_2})$ instead of $Y \sim \text{Bin}(s_2, \frac{s_2}{d_2 - s_2})$. By direct calculation, we have
\begin{align*}
    \EE[\exp(\mu^2XY)] &\leq \EE[\exp(\mu^2s_1Y)] \quad \text{(since $X \leq s_1$ almost surely)} \\
    &\leq \EE\Big[\exp\Big(\log\big(1+c'\frac{d_2}{s_2^2}\big)Y\Big)\Big] \quad \text{(since $\mu^2\leq \frac{1}{s_1}\log\big(1 + c'\frac{d_2}{s_2^2}\big)$)} \\
    &= \Big(1 + \frac{s_2}{d_2}\big(e^{\log\big(1 + c'\frac{d_2}{s_2^2}\big)} - 1\big)\Big)^{s_2} \\
    &= \big(1 + \frac{c'}{s_2}\big)^{s_2} \\
    &\leq e^{c'} \\
    &= 1 + \eps.
\end{align*}
If no such $\bar{c}$ exists, then $s_2 \geq c d_2$ for a constant $1 \geq c > 0$, and we have
\begin{align*}
    \EE[\exp(\mu^2XY)] &\leq \exp(\mu^2s_1s_2) \\
    &\leq \exp\Big(\log\big(1+c'\frac{d_2}{s_2^2}\big)s_2\Big) \quad \text{(since $\mu^2\leq \frac{1}{s_1}\log\big(1 + c'\frac{d_2}{s_2^2}\big)$)} \\
    &= \big(1 + \frac{c'd_2}{s^2_2}\big)^{s_2} \\
    &\leq e^{c' \cdot c } \\
    &= 1 + \eps
\end{align*}
where we adjust the constant $c'$ to satisfy $c' = c^{-1}\log(1 + \eps)$. It remains to show that $\frac{d_1}{s_1^2}\log\big(1 + c'\frac{d_2}{s_2^2}\big) \asymp R$. Let $C' = \frac1c$. First, note that $\frac{d_1}{s_1^2} \leq \frac{C'}{s_1} \leq C'$, so hence if $R$ is defined with $C \geq C'$, it holds $R \leq \frac{d_1}{s_1^2}\log\big(1 + c'\frac{d_2}{s_2^2}\big)$ by definition. It remains to show that $\frac{d_1}{s_1^2}\log\big(1 + c'\frac{d_2}{s_2^2}\big) \lesssim R$. Since $\frac{d_1}{s_1} \leq C'$, we have
\begin{align*}
    \frac{d_1}{s_1^2}\log\big(1 + c'\frac{d_2}{s_2^2}\big) &\leq \frac{C'}{s_1}\log\big(1 + c'\frac{d_2}{s_2^2}\big) \\
    &\leq  \frac{C'}{s_1}\log\left(1 + c'\frac{d_2}{s_2^2}\log \left(e {d_1 \choose s_1}\right)\right) \\
    &= \psi_{12} \\
    &\leq \psi_{12} + \psi_{21}.
\end{align*}
Furthermore, using the inequality $\log(1 + x) \leq x$ for any $x > 0$, we have
\begin{align*}
    \frac{d_1}{s_1^2}\log\big(1 + c'\frac{d_2}{s_2^2}\big) 
    &\leq c'\frac{d_1d_2}{s_1^2s_2^2} \\
    &\leq c' C' \frac{d_2}{s_1s_2^2} \\
    &\leq C''\frac{d_2}{s_2^2}\log\Big(1 + \frac{1}{s_1}\Big) \quad \text{(since $x \lesssim \log(1 + x)$ for $x < 1$)}\\
    &\leq C''\frac{d_2}{s_2^2}\log\Big(1 + \frac{d_1}{k^2_1}\Big) \\
    &\leq C'' \phi_{21}.
\end{align*}
Clearly, $\frac{d_1}{s_1^2}\log\big(1 + c'\frac{d_2}{s_2^2}\big) \leq \phi_{12}$. Therefore we have $\frac{d_1}{s_1^2}\log\big(1 + c'\frac{d_2}{s_2^2}\big) \lesssim R$, and the proof is complete.

\subsection{Lemmas for proof of lower bound}\label{sec_lemmas_lb}

\begin{lemma}\label{lemma:lb-principal-term}
    Let $C > 0$ be a constant such that that $s_1^2 \geq d_1 / C$. Then for any $\alpha > 0$, there exists a constant $c_{\mu} > 0$ such that  if $\mu^2 \leq c_{\mu}\frac{d_1}{s_1^2}\log\big(1 + \frac{d_2}{s_2^2}\big)$, then
    \[\EE\Big[\exp(\mu^2XY)\one\big(X \leq C \frac{s_1^2}{d_1}\big)\Big] < 1 + \alpha.\]
\end{lemma}

\begin{proof}
   By direct calculation, we have
    \begin{align*}
        \EE\Big[\exp(\mu^2XY)\one\big(X \leq C \frac{s_1^2}{d_1}\big)\Big] &\leq \EE\Big[\exp(\mu^2C \frac{s_1^2}{d_1}Y)\Big] \\ 
        &= \Big(1 + \frac{s_2}{d_2}\big(e^{\mu^2C (s_1^2/d_1)} - 1\big)\Big)^{s_2} \\ 
        &\leq \exp\Big(\frac{s_2^2}{d_2}\big(e^{\mu^2C (s_1^2/d_1)} - 1\big)\Big).
    \end{align*}
    Now since $\mu^2 \leq c_{\mu}\frac{d_1}{s_1^2}\log\big(1 + \frac{d_2}{s_2^2}\big)$, we have 
    \begin{align*}
        \exp\Big(\frac{s_2^2}{d_2}\big(e^{\mu^2C (s_1^2/d_1)} - 1\big)\Big) &\leq \exp\Big(\frac{s_2^2}{d_2}\big(e^{c_{\mu}\frac{d_1}{s_1^2}\log\big(1 + \frac{d_2}{s_2^2}\big)C (s_1^2/d_1)} - 1\big)\Big) \\ 
        &= \exp\Big(\frac{s_2^2}{d_2}\big(\big(1 + \frac{d_2}{s_2^2}\big)^{c_{\mu}C} - 1\big)\Big) \\ 
        &\leq \exp\Big(\frac{s_2^2}{d_2}\big(c_{\mu}C\frac{d_2}{s_2^2}\big)\Big) \\
        &= \exp\big(c_\mu C\big)
    \end{align*}
    where the second inequality holds for $c_\mu < C^{-1}$, by the inequality $(1+x)^y \leq 1+y x$ which holds for any $x\geq 0$ and $y \in [0,1]$. This final expression is at most $1 + \alpha$ for $c_\mu$ taken sufficiently small.
\end{proof}

For any $C \geq 1$, we define the set of indices
\begin{align}
    \bar{\cA}_C = \left\{ \Big\lceil C\frac{s_1^2}{d_1}\Big\rceil, \dots , s_1\right\}\label{eq_def_A_C}
\end{align}
The results in this section will depend largely on the following lemma.

\begin{lemma}\label{lemma:geom-series}
    Suppose that $s_1 \leq \frac12d_1$. Then for any $\alpha > 0$, there exist constants $\Ct \geq 1$ and $c_\mu > 0$ such that for any $\cA \subset \bar{\cA}_{\Ct}$, if
    \[\mu^2 \leq \min_{k \in \cA}\frac{1}{k}\log\Bigg(1 + \frac{d_2}{s_2}\bigg(\exp\Big[\frac{c_\mu k}{s_2}\log\big(\frac{kd_1}{2es_1^2}\big)\Big] - 1\bigg)\Bigg)\]
    then 
    \[\EE\Big[\exp(\mu^2XY)\one\Big(X \in \cA \Big)\Big] < \alpha.\]

\end{lemma}

\begin{proof}
    We have 
    \begin{align*}
        \EE\Big[\exp(\mu^2XY)\one\Big(X  \in \cA\Big)\Big] &= \sum_{k \in \cA}\EE\Big[\exp(\mu^2kY)\Big]\Pr(X = k) \\
        &= \sum_{k \in \cA}\Big(1 + \frac{s_2}{d_2}\big(e^{k\mu^2} - 1\big)\Big)^{s_2}\Pr(X = k) \\
        &\leq \sum_{k \in \cA}\Big(1 + \frac{s_2}{d_2}\big(e^{k\mu^2} - 1\big)\Big)^{s_2}\Big(\frac{s^2_12e}{kd_1}\Big)^ke^{-s_1^2/d_1}
    \end{align*}
    where the inequality follows from Lemma \ref{lemma:probXk}. Since we have assumed that
     \[\mu^2 \leq \min_{k \in \cA}\frac{1}{k}\log\Bigg(1 + \frac{d_2}{s_2}\bigg(\exp\Big[\frac{c_\mu k}{s_2}\log\big(\frac{kd_1}{2es_1^2}\big)\Big] - 1\bigg)\Bigg),\]
    we can substitute this into the above calculation to obtain
    \begin{align*}
        &\sum_{k \in \cA}\Big(1 + \frac{s_2}{d_2}\big(e^{k\mu^2} - 1\big)\Big)^{s_2}\Big(\frac{s^2_12e}{kd_1}\Big)^k \\ 
        &\leq \sum_{k \in \cA}\Bigg(\Big(\frac{kd_1}{s_1^22e}\Big)^{c_\mu}\Bigg)^k\Bigg(\frac{s^2_12e}{kd_1}\Bigg)^k \\ 
        &= \sum_{k \in \cA}\Bigg(\Big(\frac{s^2_12e}{kd_1}\Big)^{1-c_\mu}\Bigg)^k.
    \end{align*}
    Since $\cA \subset \bar{\cA}_{\Ct}$, for every $k \in \cA$ we have     
    $$\frac{s_1^2}{kd_1} \leq \frac{1}{\Ct}.$$
    Therefore, by taking $\Ct$ sufficiently large, we can control the above sum with a geometric series
    \begin{align*}
        &\sum_{k \in \cA}\Bigg(\Big(\frac{s^2_12e}{kd_1}\Big)^{1-c_\mu}\Bigg)^k \\
        &\leq \sum_{k \in \cA}\Bigg(\Big(\frac{2e}{\Ct}\Big)^{1-c_\mu}\Bigg)^k \\
        &\leq \frac{\tilde c}{1 - \tilde c}
    \end{align*}
    where $\tilde c = (2e/\Ct)^{1 - c_\mu}$. This quantity can be made arbitrarily small by taking $c_\mu$ sufficiently small and $\Ct$ sufficiently large. This completes the proof.
\end{proof}

\begin{lemma}\label{lem_max_test_lower_bound}
    Suppose that there exists a constant $\bar{c} \in (0,(2e)^{-4})$ such that $s_1^2 \leq \bar{c}d_1s_2$, and $s_1 \leq \frac12d_1$. 
    Furthermore, suppose that $s_2 < s_1\log\big(\frac{d_1s_2}{s_1^2}\big)$. Then for any $\alpha > 0$, there exist  constants $c_\mu > 0$ and $C_* \geq 1$ such that if 
    $$\mu^2 \leq c_\mu\frac{1}{s_2}\log\left(\frac{s_2d_1}{2es_1^2} \log\left(\frac{d_2}{s_2}\right)\right),$$
    then 
    \[\EE\left[\exp(\mu^2XY)\one\Big(X \geq \frac{s_2}{\log\big(\frac{d_1s_2}{s_1^2}\big)}\vee   \frac{\Ct s_1^2}{d_1}\Big)\right] < \alpha.\]
    If we further assume that $\frac{d_2}{s_2} \geq 2(1 - e^{-c_\mu})^{-1}$ and that $s_1 < s_2\log(d_2/s_2)$, then if $\mu^2 \leq \frac{c_\mu}{s_2}\log\big(\frac{d_1}{2es_1}\big) + \frac{1}{s_1}\log\big((1-e^{-c_\mu})\frac{d_2}{s_2}\big)$ it holds
\[\EE\left[\exp(\mu^2XY)\one\Big(X > \frac{s_2}{\log\big(\frac{d_1s_2}{s_1^2}\big)}\vee   \frac{\Ct s_1^2}{d_1}\Big)\right] < \alpha.\]
\end{lemma}

\begin{proof}[Proof of Lemma~\ref{lem_max_test_lower_bound}]
     Note that the set $\cA = \{1 \vee \big\lceil \frac{s_2}{\log\big((d_1s_2)/s_1^2\big)}\big\rceil \vee \big\lceil \Ct \frac{s^2_1}{d_1}\big\rceil, ..., s_1\}$ is a subset of $\bar{\cA}_{\Ct}$, and to prove the claim it suffices to show
    \[\mu^2 \leq \min_{k \in \cA}\frac{1}{k}\log\Bigg(1 + \frac{d_2}{s_2}\bigg(\exp\Big[\frac{c_\mu k}{s_2}\log\big(\frac{kd_1}{2es_1^2}\big)\Big] - 1\bigg)\Bigg)\]
    and invoke Lemma \ref{lemma:geom-series}. To this end, for $k \in \cA$ we define 
    \begin{equation}\label{eq:g}
    g(k) = \frac{1}{k}\log\Bigg(1 + \frac{d_2}{s_2}\bigg(\exp\Big[\frac{c_\mu k}{s_2}\log\big(\frac{kd_1}{2es_1^2}\big)\Big] - 1\bigg)\Bigg).
    \end{equation}
    Notice that for $k \in \cA$, we have 
    \begin{align*}
        \frac{c_\mu k}{s_2}\log\big(\frac{kd_1}{2es_1^2}\big) &\geq \frac{c_\mu}{\log\big(\frac{s_2d_1}{s_1^2}\big)}\log\Big(\frac{s_2d_1}{2es_1^2\log\big(\frac{s_2d_1}{s_1^2}\big)}\Big) \\
        &= c_\mu\Big(1 - \frac{\log\big(2e\log(\frac{d_1s_2}{s_1^2})\big)}{\log\big(\frac{d_1s_2}{s_1^2}\big)}\Big).
    \end{align*}
    Recall that there exists a constant $\bar{c} \in (0, (2e)^{-4})$ such that $s_1^2 \leq \bar{c}d_1s_2$. In particular, this implies that $d_1s_2/s_1^2 \geq 16e^4$. Using this, as well as the bound $\log (x) / x \leq 1/2$ for $x > 1$, we have 
    \begin{align*}
        \frac{\log\big(2e\log(\frac{d_1s_2}{s_1^2})\big)}{\log\big(\frac{d_1s_2}{s_1^2}\big)} &= \frac{\log(2e)}{\log\big(\frac{d_1s_2}{s_1^2}\big)} + \frac{\log\big(\log(\frac{d_1s_2}{s_1^2})\big)}{\log\big(\frac{d_1s_2}{s_1^2}\big)} \\
        &\leq \frac{\log(2e)}{4\log(2e)} + \frac12 \\
        &= \frac34.
    \end{align*}
    Combining this with our calculation above, we have that, for any $k \in \cA$,
    \begin{align*}
        \frac{c_\mu k}{s_2}\log\big(\frac{kd_1}{2es_1^2}\big) &\geq \frac{c_\mu}{4} \\ 
        &> 0.
    \end{align*}
    Thus, $\exp\big(\frac{c_\mu k}{s_2}\log\big(\frac{kd_1}{2es_1^2}\big)\big)$ is bounded away from 1, and hence  for $c = 1 - e^{-c_\mu}$ we have
    \[f(k) =: \frac{1}{k}\log\Bigg(c\frac{d_2}{s_2}\exp\Big[\frac{c_\mu k}{s_2}\log\big(\frac{kd_1}{2es_1^2}\big)\Big]\Bigg) \leq g(k)\]
    for each $k \in \cA$. So the prove the claim, it suffices to show that $\mu^2$ is at most the minimum value of $f(k)$ over $\cA$. We can write $f(k)$ as 
    \[f(k) = \frac{1}{k}\log\Big(c\frac{d_2}{s_2}\Big) + \frac{c_\mu}{s_2}\log\Big(\frac{kd_1}{2es_1^2}\Big).\]
    By direct calculation, we have
    \[\frac{\D f}{\D k}(k) = -\frac{\log\big(c\frac{d_2}{s_2}\big)}{k^2} + \frac{c_\mu}{s_2k}\]
    from which we deduce that $f$ is minimized at $k^* = c_\mu^{-1}s_2\log\big(c\frac{d_2}{s_2}\big)$ and decreasing for $k < k^*$. Then if $\mu^2 < c_\mu\frac{1}{s_2}\log\big(\frac{s_2d_1\log(d_2/s_2)}{s_1^2}\big)$, it follows that 
    \begin{align*}
        \mu^2 &< c_\mu\frac{1}{s_2}\log\big(\frac{s_2d_1\log(d_2/s_2)}{2es_1^2}\big) \\
        &\leq c_\mu\frac{1}{s_2}\Big(\log\big(\frac{s_2d_1\log(d_2/s_2)}{2e c_\mu s_1^2}\big) + 1\Big) \\
        &= f(k^*) \\
        &\leq \min_{k \in \cA}f(k) \\
        &\leq \min_{k\in \cA}g(k)
    \end{align*}
    and we may apply Lemma \ref{lemma:geom-series} to complete the proof. Now if $c\frac{d_2}{s_2} \geq 2$ and $s_1 < s_2\log(d_2/s_2)$, then, for some sufficiently small $c_\mu$, we have
    \begin{align*}
        s_1 &< s_2\log(d_2/s_2) \\
        &< \frac{s_2}{c_\mu}\frac{\log (2)}{\log (2) + \log \big(\frac{2e}{c}\big)}\log\big(\frac{d_2}{s_2}\big) \\
        &\leq \frac{s_2\log\big(\frac{cd_2}{s_2}\big)}{c_\mu\log\big(\frac{2ed_2}{s_2}\big)}\log\big(\frac{d_2}{s_2}\big) \\
        &\leq \frac{s_2}{c_\mu}\log\big(c\frac{d_2}{s_2}\big) \\
        &= k^*
    \end{align*}
    where the third inequality uses the fact that $x \mapsto \frac{x}{x + c}$ is increasing in $x$. Therefore since $s_1 < k^*$, the minimum value of $f(k)$ over the set $\cA$ is attained at $k = s_1$. Thus for $\mu^2 < \frac{c_\mu}{s_2}\log\big(\frac{d_1}{2es_1}\big) + \frac{1}{s_1}\log\big(c\frac{d_2}{s_2}\big)$, we have 
    \begin{align*}
        \mu^2 &< \frac{c_\mu}{s_2}\log\big(\frac{d_1}{2es_1}\big) + \frac{1}{s_1}\log\big(c\frac{d_2}{s_2}\big) \\ 
        &= f(s_1) \\
        &= \min_{k \in \cA}f(k) \\
        &\leq \min_{k\in \cA}g(k).
    \end{align*}
    Applying Lemma \ref{lemma:geom-series} completes the proof.    
\end{proof}

%
%
%
\begin{lemma}\label{lem_lb_s1d1big_truncchi2}
    Suppose that $s_1^2 \leq \bar{c}d_1s_2$ for some $\bar{c} \in (0,(2e)^{-4})$ and $s_1 \leq \frac12d_1$. Furthermore, suppose that $\frac{s^2_2}{d_2} \leq 2e\frac{s_1^2}{d_1}$. For any small constant $c > 0$, we define the quantity
    \begin{align*}
        M(c) = \begin{cases}
            \frac{1}{s_1}\log\Big(1 + c \frac{d_2s_1}{s_2^2}\log(\frac{d_1}{2es_1})\Big) & \text{ if } s_1 \leq \frac{s_2}{\log\big(\frac{d_1s_2}{s_1^2}\big)}\\
            \frac{1}{s_2}\log\big(\frac{d_1s_2}{s_1^2}\big)\log\Big(1 + c \frac{d_2}{s_2\log\big(\frac{d_1s_2}{s_1^2}\big)}\log\big(\frac{d_1s_2}{2es_1^2\log(\frac{d_1s_2}{s_1^2})}\big)\Big) & \text{ otherwise.}
        \end{cases}
    \end{align*}
    Then for any $\alpha > 0$, there exist constants $c_\mu > 0$ and $C_* \geq 1$ such that if $\mu \leq M(c_\mu)$ and $1 \vee \Ct \frac{s_1^2}{d_1}\leq s_1 \wedge\frac{s_2}{\log\big(\frac{d_1s_2}{s_1^2}\big)}$, then 
    \[\EE\Big[\exp(\mu^2XY)\one\Big( 1 \vee \Ct \frac{s_1^2}{d_1} \leq X \leq s_1\wedge\frac{s_2}{\log\big(\frac{d_1s_2}{s_1^2}\big)}\Big)\Big] < \alpha.\]
    Moreover, it holds that    
    $$2M(c_\mu) \geq \frac{1}{s_1}\log\Big(1 + c_\mu \frac{d_2s_1}{s_2^2}\log(\frac{d_1}{2es_1})\Big) + \frac{\log\big(\frac{d_1s_2}{s_1^2}\big)}{s_2}\log\Big(1 + c_\mu \frac{d_2}{s_2\log\big(\frac{d_1s_2}{s_1^2}\big)}\log\big(\frac{d_1s_2}{2es_1^2\log(\frac{d_1s_2}{s_1^2})}\big)\Big).$$
\end{lemma}

\begin{proof}[Proof of Lemma~\ref{lem_lb_s1d1big_truncchi2}]
    Define $\cA = \{ 1 \vee \lceil\Ct \frac{s_1^2}{d_1} \rceil, ...s_1 \wedge \big\lfloor\frac{s_2}{\log\big(\frac{d_1s_2}{s_1^2}\big)} \big\rfloor \}$. Clearly $\cA \subset \bar{\cA}_{\Ct}$, and hence it suffices to show that $\mu^2 \leq \min_{k \in \cA}g(k)$ where $g(k)$ is defined as in (\ref{eq:g}). Using the inequality $e^x \geq 1 + x $, we have 
    \begin{align*}
        g(k) &= \frac{1}{k}\log\Bigg(1 + \frac{d_2}{s_2}\bigg(\exp\Big[\frac{c_\mu k}{s_2}\log\big(\frac{kd_1}{2es_1^2}\big)\Big] - 1\bigg)\Bigg) \\
        &\geq \frac{1}{k}\log\Bigg(1 + c_\mu\frac{d_2}{s^2_2}k\log\Big(\frac{kd_1}{2es_1^2}\Big)\Bigg) \\
        &=: f^{(1)}(k).
    \end{align*}
    
%
%
    We define $\widetilde{\mathcal{A}} = \{ 1 \vee \lceil\Ct \frac{s_1^2}{d_1} \rceil, \dots \}$ and we will now show that $f^{(1)}(k)$ is decreasing in $k$ over the set $\widetilde{\mathcal{A}}$. To do so, it suffices to consider the function $f(k) = \frac{1}{k}\log\Big(c_\mu\frac{d_2}{s^2_2}k\log\big(\frac{kd_1}{2es_1^2}\big)\Big)$ and apply Lemma \ref{lemma:log-plus-one}. Direct calculation reveals
    \[\frac{\D f}{\D k}(k) = -\frac{\log\big(c_\mu\frac{d_2}{s_2^2}\big)}{k^2} - \frac{\log(k)}{k^2} + \frac{1}{k^2} - \frac{\log \log \big(\frac{kd_1}{2es_1^2}\big)}{k^2} + \frac{1}{k^2\log\big(\frac{kd_1}{2es_1^2}\big)}.\]
    The right-hand side is less than zero if   
    \[1 + \frac{1}{\log\big(\frac{kd_1}{2es_1^2}\big)} < \log\Big(c_\mu\frac{d_2k}{s_2^2} \log \big(\frac{kd_1}{2es_1^2}\big)\Big).\]
    
    For $k \in \widetilde{\mathcal{A}}$, it holds 
    \begin{align*}
        \log\Big(c_\mu\frac{d_2k}{s_2^2}\log \big(\frac{kd_1}{2es_1^2}\big)\Big) 
        &\geq \log\Big(c_\mu \Ct \frac{d_2s_1^2}{s_2^2d_1}\log \big(\frac{\Ct}{2e}\big)\Big) \\
        &\geq \log\Big(\frac{c_\mu C_*}{2e} \log \big(\frac{\Ct}{2e}\big)\Big)\\
        &\geq 2
    \end{align*}
    where the second inequality uses the assumption $\frac{s_2^2}{d_2} \leq 2e\frac{s_1^2}{d_1}$ and the third holds as long as $\Ct$ is taken sufficiently large. Furthermore, for $k \in \widetilde{\mathcal{A}}$ we have 

    \begin{align*}
        1 + \frac{1}{\log \big(\frac{kd_1}{2es_1^2}\big)} &\leq 1 + \frac{1}{\log\big(\frac{\Ct}{2e}\big)} \\
        &< 2.
    \end{align*}

    Therefore, $f(k)$ is decreasing over $\widetilde{\mathcal{A}}$, and by Lemma \ref{lemma:log-plus-one}, $f^{(1)}(k)$ is decreasing over $\widetilde{\mathcal{A}}$ as well. 
    We also observe that 
    \begin{align*}
        M = \begin{cases} f^{(1)}(s_1) & \text{ if } s_1 \leq \frac{s_2}{\log(\frac{d_1s_2}{s_1^2})}\\
        f^{(1)}(\frac{s_2}{\log(\frac{d_1s_2}{s_1^2})}) & \text{ otherwise,}
        \end{cases}
    \end{align*}
    that is, $M= f^{(1)}\Big(s_1 \land \frac{s_2}{\log(\frac{d_1s_2}{s_1^2})}\Big)$. 
    Moreover, since $f^{(1)}$ is decreasing over $\widetilde{\mathcal{A}}$, we also have
    \begin{align*}
        M &= \max\bigg(f^{(1)}(s_1), ~f^{(1)}\bigg(\frac{s_2}{\log\big(\frac{d_1s_2}{s_1^2}\big)}\bigg)\bigg) \\
        &\geq \frac{1}{2} \left[f^{(1)}(s_1) + ~f^{(1)}\bigg(\frac{s_2}{\log\big(\frac{d_1s_2}{s_1^2}\big)}\bigg)\right]\\
        & = \frac{1}{2} \left[\frac{1}{s_1}\log\Big(1 + c_\mu \frac{d_2s_1}{s_2^2}\log(\frac{d_1}{2es_1})\Big) + \frac{\log\big(\frac{d_1s_2}{s_1^2}\big)}{s_2}\log\Big(1 + c_\mu \frac{d_2}{s_2\log\big(\frac{d_1s_2}{s_1^2}\big)}\log\big(\frac{d_1s_2}{2es_1^2\log(\frac{d_1s_2}{s_1^2})}\big)\Big)\right]
    \end{align*}
    as claimed. 
    Using our assumed upper bound on $\mu^2$, we have 
    \begin{align*}
        \mu^2 &\leq M =  f^{(1)}\bigg(s_1 \land \frac{s_2}{\log(\frac{d_1s_2}{s_1^2})}\bigg) \\
        &\leq \min_{k \in \cA}f^{(1)}(k) \\
        &\leq \min_{k\in \cA}g(k).
    \end{align*}
    We conclude the proof by invoking Lemma \ref{lemma:geom-series}. 
\end{proof}

\begin{lemma}\label{lem_lb_s2d2big_maxlin}
    Suppose that $s_1^2 \leq \bar{c}d_1s_2$ for some $\bar{c} \in (0, (2e)^{-4})$ and $s_1 < \frac12 d_1$. 
    Furthermore, suppose that $\frac{s^2_2}{d_2} \geq 2e\frac{s_1^2}{d_1}$ . Then for any $\alpha > 0$, there exist constants $c'_{\mu} > 0$ and $\Ct > 0$ such that if $1 \vee \lceil \Ct \frac{s_1^2}{d_1}\rceil \leq \big\lfloor \frac{\Ct s_2^2/d_2}{\log\big(\frac{d_1s_2^2}{s_1^2d_2}\big)}\big\rfloor\wedge s_1$ and
    \[\mu^2 \leq c'_\mu \frac{d_2}{s_2^2}\log\big(\frac{\Ct}{2e}\vee \frac{d_1}{2es^2_1}\big),\]
    then
    \[\EE\left[\exp\big(\mu^2XY\big)\one\bigg(1 \vee \Ct \frac{s_1^2}{d_1} \leq X \leq \frac{\Ct s_2^2/d_2}{\log\big(\frac{d_1s_2^2}{s_1^2d_2}\big)}\wedge s_1\bigg)\right] < \alpha.\]
\end{lemma}

\begin{proof}[Proof of Lemma~\ref{lem_lb_s2d2big_maxlin}]
    Let $g(k)$ be defined as in (\ref{eq:g}). For $\cA = \{1 \vee \lceil \Ct \frac{s_1^2}{d_1}\rceil, \dots ,  \Big\lfloor \frac{s_2^2/d_2}{\log\big(\frac{d_1s_2^2}{s_1^2d_2}\big)}\Big\rfloor\wedge s_1\}$, it suffices to show that $\mu^2 \leq \min_{k\in\cA}g(k)$. Using $e^x \geq x + 1,\, \forall x \in \R$, we have 
    \begin{align*}
        g(k) &= \frac{1}{k}\log\Bigg(1 + \frac{d_2}{s_2}\bigg(\exp\Big[\frac{c_\mu k}{s_2}\log\big(\frac{kd_1}{2es_1^2}\big)\Big] - 1\bigg)\Bigg) \\
        &\geq \frac{1}{k}\log\Bigg(1 + c_\mu\frac{d_2}{s^2_2}k\log\Big(\frac{kd_1}{2es_1^2}\Big)\Bigg).
    \end{align*}
    For $k \in \cA$,  we can control the term within the logarithm as
    \begin{align*}
        c_\mu\frac{d_2}{s^2_2}k\log\Big(\frac{kd_1}{2es_1^2}\Big) 
        &\leq c_\mu \Ct \frac{\log\Big(\frac{\Ct d_1s_2^2}{2es_1^2d_2\log\big(d_1s_2^2/(s_1^2d_2)\big)} \Big)}{\log\big(\frac{d_1s_2^2}{s_1^2d_2}\big)} \\
        &\leq c_\mu\Ct \big(\log (\Ct) + 1 \big) \\
        &< 2c_\mu\Ct \log (\Ct)
    \end{align*}
    where the second inequality follows from $\log(x/\log(x))/\log(x) \leq 1$ for $x \geq 2e$, and the final inequality holds for $\Ct > e$. Therefore, there exists a constant $c \in (0,1)$ which depends on $c_\mu$ and $\Ct$ such that 
    \begin{align*}
        g(k) &\geq \frac{1}{k}\log\Bigg(1 + c_\mu\frac{d_2}{s^2_2}k\log\Big(\frac{kd_1}{2es_1^2}\Big)\Bigg) \\
        &\geq cc_\mu \Ct \frac1k\frac{d_2}{s_2^2}k\log\big(\frac{kd_1}{2es_1^2}\big) \\
        &= c'_\mu\frac{d_2}{s_2^2}\log\big(\frac{kd_1}{2es_1^2}\big) \\
        &=: f(k)
    \end{align*}
    for $c'_\mu = c c_\mu \Ct$. The function $f(k)$ is clearly increasing in $k$, and hence is minimized at the smallest value of $k$ in $\cA$. 
    Now for $\mu^2 \leq c'_\mu \frac{d_2}{s_2^2}\log\big(\frac{\Ct}{2e}\vee \frac{d_1}{2es^2_1}\big)$, we have
    \begin{align*}
        \mu^2 &\leq c'_\mu \frac{d_2}{s_2^2}\log\big(\frac{\Ct}{2e}\vee \frac{d_1}{2es^2_1}\big) \\
        &\leq \min_{k \in \cA}f(k) \\
        &\leq \min_{k \in \cA}g(k).
    \end{align*}
    We apply Lemma \ref{lemma:geom-series} to complete the proof.
\end{proof}
\begin{lemma}\label{lem_lb_s2d2big_truncchi2}
    Suppose that $s_1^2 \leq \bar{c}d_1s_2$ for some $\bar{c} \in (0,(2e)^{-4})$ and $s_1 \leq \frac{1}{2e}d_1$. Furthermore, suppose that $\frac{s^2_2}{d_2} \geq 2e\frac{s_1^2}{d_1}$. Then for any $\alpha > 0$, there exist constants $c_\mu > 0$ and $C_* \geq 1$ such that if $1 \vee \Ct \frac{s_1^2}{d_1} \vee  \frac{\Ct s_2^2/d_2}{\log\big(\frac{d_1s_2^2}{s_1^2d_2}\big)} \leq s_1\wedge\frac{s_2}{\log\big(\frac{d_1s_2}{s_1^2}\big)}$ and
    $$\mu^2 \leq \frac12\bigg(\frac{1}{s_1}\log\Big(1 + c_\mu \frac{d_2s_1}{s_2^2}\log(\frac{d_1}{2es_1})\Big) + \frac{\log\big(\frac{d_1s_2}{s_1^2}\big)}{s_2}\log\Big(1 + c_\mu \frac{d_2}{s_2\log\big(\frac{d_1s_2}{s_1^2}\big)}\log\big(\frac{d_1s_2}{2es_1^2\log(\frac{d_1s_2}{s_1^2})}\big)\Big)\bigg),$$
    then 
    \[\EE\bigg[\exp(\mu^2XY)\one\bigg( 1 \vee \Ct \frac{s_1^2}{d_1} \vee  \frac{\Ct s_2^2/d_2}{\log\big(\frac{d_1s_2^2}{s_1^2d_2}\big)} \leq X \leq s_1\wedge \frac{s_2}{\log\big(\frac{d_1s_2}{s_1^2}\big)}\bigg)\bigg] < \alpha.\]
\end{lemma}

\begin{proof}[Proof of Lemma~\ref{lem_lb_s2d2big_truncchi2}]
    The structure of this proof is similar to that of Lemma \ref{lem_lb_s1d1big_truncchi2}. Define $\cA = \{ 1 \vee \lceil\Ct \frac{s_1^2}{d_1} \rceil \vee \Big\lceil \frac{\Ct s_2^2/d_2}{\log\big(\frac{d_1s_2^2}{s_1^2d_2}\big)}\Big\rceil, \dots , s_1 \wedge \lfloor\frac{s_2}{\log\big(\frac{d_1s_2}{s_1^2}\big)} \rfloor \}$. With $g(k)$ defined as in (\ref{eq:g}), we have 

    \begin{align*}
        g(k) &= \frac{1}{k}\log\Bigg(1 + \frac{d_2}{s_2}\bigg(\exp\Big[\frac{c_\mu k}{s_2}\log\big(\frac{kd_1}{2es_1^2}\big)\Big] - 1\bigg)\Bigg) \\
        &\geq \frac{1}{k}\log\Bigg(1 + c_\mu\frac{d_2}{s^2_2}k\log\Big(\frac{kd_1}{2es_1^2}\Big)\Bigg) \\
        &=: f^{(1)}(k).
    \end{align*}

    As in the proof of Lemma \ref{lem_lb_s1d1big_truncchi2}, will now show that $f^{(1)}(k)$ is decreasing in $k$ over the set $\cA$. It suffices to consider the function $f(k) =: \frac{1}{k}\log\Big(c_\mu\frac{d_2}{s^2_2}k\log\big(\frac{kd_1}{2es_1^2}\big)\Big)$ and apply Lemma \ref{lemma:log-plus-one}. By direct calculation, we have
    \[\frac{\D f}{\D k}(k) = -\frac{\log\big(c_\mu\frac{d_2}{s_2^2}\big)}{k^2} - \frac{\log(k)}{k^2} + \frac{1}{k^2} - \frac{\log \log \big(\frac{kd_1}{2es_1^2}\big)}{k^2} + \frac{1}{k^2\log\big(\frac{kd_1}{2es_1^2}\big)}.\]
    The right hand side is less than zero if 
    \[1 + \frac{1}{\log\big(\frac{kd_1}{2es_1^2}\big)} < \log\Big(c_\mu\frac{d_2k}{s_2^2}\log \big(\frac{kd_1}{2es_1^2}\big)\Big)\]
    For $k \in \cA$, it holds 
    \begin{align*}
        \log\Big(c_\mu\frac{d_2k}{s_2^2} \log \big(\frac{kd_1}{2es_1^2}\big)\Big) 
        &\geq \log\Bigg(\frac{c_\mu \Ct}{\log\big(\frac{d_1s_2^2}{s_1^2d_2}\big)} \log \Big(\frac{\Ct s_2^2d_1}{2es_1^2d_2\log\big(\frac{s_2^2d_1}{s_1^2d_2}\big)}\Big)\Bigg) \\
        &\geq \log\Big(\frac{c_\mu \Ct }{2} \Big) \\
        &\geq 2
    \end{align*}
    where the second inequality uses that $\frac{\log \log x}{\log x} < \frac12$ for $x > 1$ to conclude 
    $$\frac{\log\big(\frac{\Ct s_2^2d_1}{2ed_2s_1^2}\big) - \log \log \big(\frac{s_2^2d_1}{d_2s_1^2}\big)}{\log \big(\frac{d_1s_2^2}{s_1^2d_2}\big)} \geq \frac12,$$
    and the final inequality holds as long as $\Ct$ is taken sufficiently large such that $\Ct \geq 2e^2 / c_\mu$. Furthermore, for $k \in \cA$ we have 
    \begin{align*}
        1 + \frac{1}{\log \big(\frac{kd_1}{2es_1^2}\big)} &\leq 1 + \frac{1}{\log\big(\frac{\Ct}{2e}\big)} \\
        &< 2
    \end{align*}
    Therefore, $f(k)$ is decreasing over $\cA$, and by Lemma \ref{lemma:log-plus-one}, $f^{(1)}(k)$ is decreasing over $\cA$ as well. Using our assumed upper bound on $\mu^2$, we have 
    \begin{align*}
        \mu^2 &\leq \frac12\bigg(\frac{1}{s_1}\log\Big(1 + c_\mu \frac{d_2s_1}{s_2^2}\log(\frac{d_1}{2es_1})\Big) + \frac{\log\big(\frac{d_1s_2}{s_1^2}\big)}{s_2}\log\Big(1 + c_\mu \frac{d_2}{s_2\log\big(\frac{d_1s_2}{s_1^2}\big)}\log\big(\frac{d_1s_2}{2es_1^2\log(\frac{d_1s_2}{s_1^2})}\big)\Big)\bigg) \\
        &= \frac12\Big(f^{(1)}(s_1) + f^{(1)}\big(\frac{s_2}{\log(\frac{d_1s_2}{s_1^2})}\big)\Big) \\
        &\leq \max\Big(f^{(1)}(s_1), f^{(1)}\big(\frac{s_2}{\log(\frac{d_1s_2}{s_1^2})}\big)\Big) \\
        &\leq \min_{k \in \cA}f^{(1)}(k) \\
        &\leq \min_{k\in \cA}g(k).
    \end{align*}
    We conclude the proof by invoking Lemma~\ref{lemma:geom-series}. 
\end{proof}

\begin{lemma}\label{lem_lb_dense_allrates}
    Suppose that $s_1^2 \geq (2e)^{-4}d_1s_2$ and that $s_1 < c'd_1$ for some small enough constant $c'>0$.
    Then for any $\alpha > 0$, there exist constants $c_\mu > 0$ and $\Ct \geq 1$ such that the following claims hold. In what follows, we define $c = 1 - \exp\Big(-c_\mu \Ct (2e)^{-4}\log \Big(\frac{\Ct}{2e}\Big)\Big)$.
    \begin{enumerate}
        \item Suppose that $c_\mu^{-1}s_2\log\big(c\frac{d_2}{s_2}\big) \leq \Ct \frac{s_1^2}{d_1}$. Then if 
        \[\mu^2 < \Ct^{-1}\frac{d_1}{s_1^2}\log\Big(c \frac{d_2}{s_2}\Big) + \frac{c_\mu}{s_2}\log\Big(\frac{\Ct}{2e}\Big),\]
         it holds 
        \[\EE\Big[\exp(\mu^2XY)\one\big(X \geq \lceil  \Ct \frac{s_1^2}{d_1}\rceil\big)\Big] < \alpha.\]
        \item Suppose that $\Ct \frac{s_1^2}{d_1} < c_\mu^{-1}s_2\log\big(c\frac{d_2}{s_2}\big) \leq s_1$. Then if
        \[\mu^2 < \frac{c_\mu}{s_2}\log\Big(\frac{s_2d_1\log(c \frac{d_2}{s_2})}{c_\mu 2e s_1^2}\Big),\]
        it holds 
        \[\EE\Big[\exp(\mu^2XY)\one\big(X \geq \lceil  \Ct \frac{s_1^2}{d_1}\rceil\big)\Big] < \alpha.\]
        \item Suppose that $s_1 < c_\mu^{-1}s_2\log\big(c\frac{d_2}{s_2}\big)$.Then if
        \[\mu^2 < \frac{1}{s_1}\log\big(c \frac{d_2}{s_2}\big) + \frac{c_\mu}{s_2}\log\big(\frac{d_1}{2es_1}\big),\]
        it holds 
        \[\EE\Big[\exp(\mu^2XY)\one\big(X \geq \lceil  \Ct \frac{s_1^2}{d_1}\rceil\big)\Big] < \alpha.\]
    \end{enumerate}
\end{lemma}
\begin{proof}
    Define $\cA = \{\lceil  \Ct \frac{s_1^2}{d_1}\rceil, ..., s_1 \}$ and $g$ as in (\ref{eq:g}). For $k \in \cA$, it holds 
    \begin{align*}
        \frac{c_\mu k}{s_2}\log\Big(\frac{kd_1}{2es_1^2}\Big) &\geq \frac{c_\mu \Ct s_1^2}{s_2d_1}\log\Big(\frac{\Ct}{2e}\Big) \\
        &\geq c_\mu \Ct (2e)^{-4}\log \Big(\frac{\Ct}{2e}\Big) \\
        &> 0
    \end{align*}
    where the second inequality follows from the assumption $s_1^2 \geq (2e)^{-4}d_1s_2$ and the final inequality holds for $\Ct > 2e$. This implies that $\exp\Big(\frac{c_\mu k}{s_2}\log\Big(\frac{kd_1}{2es_1^2}\Big)\Big)$ is bounded away from 1 for $k \in \cA$. Therefore, for $c = 1 - \exp\Big(-c_\mu \Ct (2e)^{-4}\log \Big(\frac{\Ct}{2e}\Big)\Big)$, it holds
    \begin{align*}
        g(k) &= \frac{1}{k}\log\Bigg(1 + \frac{d_2}{s_2}\bigg(\exp\Big[\frac{c_\mu k}{s_2}\log\big(\frac{kd_1}{2es_1^2}\big)\Big] - 1\bigg)\Bigg)\\
        &\geq \frac{1}{k}\log\Bigg(1 + c\frac{d_2}{s_2}\exp\Big[\frac{c_\mu k}{s_2}\log\big(\frac{kd_1}{2es_1^2}\big)\Big]\Bigg)\\
        &\geq \frac{1}{k}\log\Bigg(c\frac{d_2}{s_2}\exp\Big[\frac{c_\mu k}{s_2}\log\big(\frac{kd_1}{2es_1^2}\big)\Big]\Bigg) \\
        &=: f(k).
    \end{align*}
    Therefore, it suffices to show that $\mu^2$ is at most the minimum of $f$ over $\cA$. We write $f(k)$ as 
    \[f(k) = \frac{1}{k}\log\Big(c\frac{d_2}{s_2}\Big) + \frac{c_\mu}{s_2}\log\Big(\frac{kd_1}{2es_1^2}\Big).\]
    By direct calculation, we have
    \[\frac{\D f}{\D k}(k) = -\frac{\log\big(c\frac{d_2}{s_2}\big)}{k^2} + \frac{c_\mu}{s_2k}\]
    from which we deduce that $f$ is minimized at $k^* = c_\mu^{-1}s_2\log\big(c\frac{d_2}{s_2}\big)$ over $\R^+$, decreasing for $k < k^*$, and increasing for $k > k^*$. We now proceed by cases.

    \begin{enumerate}
        \item Suppose that $c_\mu^{-1}s_2\log\big(c\frac{d_2}{s_2}\big) \leq \Ct \frac{s_1^2}{d_1}$. Then $f$ is increasing over $\mathcal{A}$, and by our assumption on $\mu^2$, it holds
        \begin{align*}\mu^2 &< \Ct^{-1}\frac{d_1}{s_1^2}\log\Big(c \frac{d_2}{s_2}\Big) + \frac{c_\mu}{s_2}\log\Big(\frac{\Ct}{2e}\Big) \\
        &\leq \min_{k \in \cA}f(k) \\
        &\leq \min_{k \in \cA}g(k).
        \end{align*}
        \item Suppose that $\Ct \frac{s_1^2}{d_1} < c_\mu^{-1}s_2\log\big(c\frac{d_2}{s_2}\big) \leq s_1$. Then by our assumption on $\mu^2$, we have
        \begin{align*}
            \mu^2 &< \frac{c_\mu}{s_2}\log\Big(\frac{s_2d_1\log(c \frac{d_2}{s_2})}{c_\mu 2e s_1^2}\Big)\\
        &\leq \min_{k \in \cA}f(k) \\
        &\leq \min_{k \in \cA}g(k).
        \end{align*}
        \item Suppose that $s_1 < c_\mu^{-1}s_2\log\big(c\frac{d_2}{s_2}\big)$. Then $f$ is decreasing over $\mathcal{A}$ and is minimized at $k = s_1$. Therefore by our assumption on $\mu^2$, we have 
        \begin{align*}
            \mu^2 &< \frac{1}{s_1}\log\big(c \frac{d_2}{s_2}\big) + \frac{c_\mu}{s_2}\log\big(\frac{d_1}{2es_1}\big)\\
        &= \min_{k \in \cA}f(k) \\
        &\leq \min_{k \in \cA}g(k).
        \end{align*}
    \end{enumerate}
    By Lemma \ref{lemma:geom-series}, the proof is complete.
\end{proof}

\subsubsection{Simplification of the rate}
For any $s_1,s_2, d_1, d_2 \in \mathbb N$, we define the following quantities
\begin{align}
    \psi(s_1,s_2,d_1,d_2) &= \frac{1}{s_1} \log\left(1+ \frac{d_2s_1}{s_2^2} \log\left(\frac{d_1}{s_1}\right)\right) \label{eq_def_psi}\\[10pt]
    \phi(s_1,s_2, d_1,d_2) &= \begin{cases}
        \frac{d_1}{s_1^2} \log\left(1+\frac{d_2}{s_2^2}\right)&  \text{ if } \frac{d_1}{s_1^2} \leq C\\
         \infty & \text{ otherwise.}
    \end{cases}\label{eq_def_phi}\\[10pt]
    \beta(s_1,s_2,d_1,d_2) &= \frac{1}{s_1} \log\left(\frac{d_2}{s_2}\right) \mathbf{1}_{\left\{\frac{d_1s_2}{s_1^2} \log\left(\frac{d_2}{s_2}\right)>1\right\}}.\label{eq_def_nu}
\end{align}
To alleviate the notation, we will write
\begin{align*}
    \phi_{12} &= \phi(s_1,s_2, d_1,d_2)\\
    \phi_{21} &= \phi(s_2,s_1, d_2,d_1)\\
    \psi_{12} &= \psi(s_1,s_2, d_1,d_2)\\
    \psi_{21} &= \psi(s_2,s_1, d_2,d_1)\\
    \beta_{12} &= \beta(s_1,s_2, d_1,d_2)\\
    \beta_{21} &= \beta(s_2,s_1, d_2,d_1).
\end{align*}
Below, we show that the lower bound can be rewritten as
\[R := R(s_1, s_2, d_1, d_2) = \Big(\psi_{12} + \psi_{21}\Big) \land \phi_{12}  \land \phi_{21}.\]

The following lemma simplifies the rate emerging from Lemma ~\ref{lem_lb_dense_allrates}.
\begin{lemma}\label{lem_simplify_dense_allrates}
    Assume that $s_1^2 \geq \bar c d_1 s_2$ for some constant $\bar c>0$ and $s_j \leq c' d_j, \,\forall j \in \{1,2\}$ for some sufficiently small constant $c'>0$. 
    Assume also that $\frac{d_1}{s_1} \geq e\log\left(\frac{d_2}{s_2}\right)$. 
    \begin{enumerate}
        \item 
        Then the following two properties hold    
    \begin{align}
        & \Big(\psi_{12} + \psi_{21}\Big) \land \phi_{12}
        \asymp \psi_{21} \land \phi_{12}\label{eq_simplification_dense_case}\\
        & \psi_{21} \land \phi_{12}
        \asymp \Big(\psi_{21} + \beta_{12}\Big) \land \phi_{12}.\label{eq_simplification_dense_case_upper_bound}
    \end{align}
    \item
    It follows that, for any $\alpha>0$, there exists a constant $c_\mu>0$ such that, whenever $\mu^2 \leq c_\mu R$, we have
    \begin{align*}
    \EE\Big[\exp(\mu^2XY)\; \one\Big(X \geq \Ct \frac{s_1^2}{d_1}\Big)\Big] < \alpha.
    \end{align*}
    \end{enumerate}
\end{lemma}
\begin{proof}[Proof of Lemma ~\ref{lem_simplify_dense_allrates}]
\phantom{}

\begin{enumerate}
    \item 
For any $s_1,s_2, d_1, d_2$, we let 
\begin{align*}
    \widetilde \phi_{12} = \frac{d_1}{s_1^2} \log\left(1+\frac{d_2}{s_2^2}\right).
\end{align*}
We start by showing that 
\begin{align*}
    \psi_{21} \land \phi_{12} \asymp \psi_{21} \land \widetilde \phi_{12}.
\end{align*}
This is clear if $\frac{d_1}{s_1^2} \leq C$ by definition of $\phi_{12}$.  
Assume now that $\frac{d_1}{s_1^2} > C$, which implies that $\phi_{12} = \infty$ and $\psi_{21} \land \phi_{12} = \psi_{21}$. 
By the assumption $s_1^2 \geq \bar c d_1 s_2$, we obtain $s_2 \leq \frac{s_1^2}{d_1\bar c} \leq \frac{1}{\bar c}$. 
Therefore, we have
\begin{align*}
    \phi_{12} = \frac{d_1}{s_1^2} \log\left(1+\frac{d_2}{s_2^2}\right) &\geq \frac{d_1}{s_1^2} \log\left(1+\bar c^2 d_2\right)& \\
    &\geq \bar c^2\frac{d_1}{s_1^2} \log\left(1+d_2\right) &\text{ by Lemma~\ref{lem_logs_constants}.(i)}\\
    & \geq \bar c^2\log\left(1+\frac{d_1}{s_1^2} s_2\log\left(\frac{d_2}{s_2}\right)\right)  & \text{ by Lemma~\ref{lem_logs_constants}.(ii)}\\
    & \geq \bar c^2\frac{1}{s_2}\log\left(1+\frac{d_1}{s_1^2} s_2\log\left(\frac{d_2}{s_2}\right)\right) &\\
    & = \bar c^2 \psi_{21}.&
\end{align*}
Therefore, it follows that
\begin{align*}
    \psi_{21} \land \phi_{12} &\asymp \psi_{21} \land \widetilde \phi_{12}\\
    \Big(\psi_{12} + \psi_{21}\Big) \land \phi_{12} &\asymp\Big(\psi_{12} + \psi_{21}\Big) \land \widetilde \phi_{12}\\
    \Big(\psi_{21} + \beta_{12}\Big) \land \phi_{12} & \asymp \Big(\psi_{21} + \beta_{12}\Big) \land \widetilde \phi_{12} 
\end{align*}
and the equations~\eqref{eq_simplification_dense_case} and~\eqref{eq_simplification_dense_case_upper_bound} become equivalent to proving
\begin{align}
\Big(\psi_{12} + \psi_{21}\Big) \land \widetilde \phi_{12}
        \asymp \psi_{21} \land \widetilde \phi_{12}.\label{eq_simplification_dense_case_tilde}\\
\psi_{21} \land \widetilde \phi_{12}
        \asymp \Big(\psi_{21} + \beta_{12}\Big) \land \widetilde \phi_{12}\label{eq_simplification_dense_case_upper_bound_tilde}
\end{align}
In the rest of the proof, we focus on proving~\eqref{eq_simplification_dense_case_tilde} and~\eqref{eq_simplification_dense_case_upper_bound_tilde}.
    
    Assume first that $\frac{d_1s_2}{s_1^2} \log\left(\frac{d_2}{s_2}\right) \leq 1$. 
    Using the inequality $\log(1+x) \geq x\log(2) $ for any $x \leq 1$, we have 
    \begin{align*}
        \psi_{21} = \frac{1}{s_2} \log\left(1+\frac{d_1s_2}{s_1^2} \log\left(\frac{d_2}{s_2}\right)\right) &\geq \log(2) \frac{d_1}{s_1^2} \log\left(\frac{d_2}{s_2}\right)\\
        & \geq \frac{\log(2)}{2} \frac{d_1}{s_1^2} \log\left(1+\frac{d_2}{s_2^2}\right) \quad \text{ if $c'$ is small enough}\\
        & = \frac{\log(2)}{2} \widetilde \phi_{12},
    \end{align*}
    which guarantees that both~\eqref{eq_simplification_dense_case_tilde} and~\eqref{eq_simplification_dense_case_upper_bound_tilde} hold. 
    From now on, we will therefore assume that $\frac{d_1s_2}{s_1^2} \log\left(\frac{d_2}{s_2}\right) > 1$. 
    Now, assume that $\frac{d_2}{s_2^2} \leq 1$. 
    This implies the following inequalities   
    \begin{align*}
        \psi_{21} = \frac{1}{s_2} \log\left(1+\frac{d_1s_2}{s_1^2} \log\left(\frac{d_2}{s_2}\right)\right) &>\frac{\log(2)}{s_2} \\
        & \geq \frac{\log(2)}{s_2} \, \frac{\bar c d_1s_2}{s_1^2} \, \frac{d_2}{s_2^2} \\
        & = \bar c \log(2) \frac{d_1d_2}{s_1^2 s_2^2} \\
        & \geq \bar c \log(2) \frac{d_1}{s_1^2} \log\left(1+\frac{d_2}{s_2^2}\right)\\
        & = \bar c \log(2) \widetilde \phi_{12},
    \end{align*}
    which ensures that~\eqref{eq_simplification_dense_case_tilde} and~\eqref{eq_simplification_dense_case_upper_bound_tilde} hold as well and yields the result. 
    Thus, we will now assume that $\frac{d_2}{s_2^2}>1$, which also implies $\frac{d_2s_1}{s_2^2} \log\left(\frac{d_1}{s_1}\right)>1$.

    Now, suppose for the sake of contradiction that $s_1 > \frac{d_1}{\log(d_1)}$. 
    Note that the assumption $\frac{d_1}{s_1} \geq e \log\left(\frac{d_2}{s_2}\right)$ implies that
    \begin{align*}
        \log(d_1) > e\log\left(\frac{d_2}{s_2}\right) > e\log(\sqrt{d_2}) = \frac{e}{2} \log(d_2),
    \end{align*}
    that is $d_1 > d_2^{e/2}.$ Therefore, 
    \begin{align*}
        \frac{s_2 d_1}{s_1^2} \log\left(\frac{d_2}{s_2}\right) &< \frac{s_2d_1}{d_1^2/\log^2(d_1)}\log\left(d_2\right) < \frac{\sqrt{d_2}}{d_1 }\cdot \log(d_2)\log^2(d_1)\\
        & < \frac{2}{e}\frac{\log^3(d_1)}{d_1^{1-1/e}} <1,
    \end{align*}
    which contradicts the assumed inequality $ \frac{s_2 d_1}{s_1^2} \log\left(\frac{d_2}{s_2}\right)>1$. 

    Therefore, it holds that $ s_1 \leq \frac{d_1}{\log(d_1)}$. 
    Since $\frac{d_1 s_2}{s_1^2} \log\left(\frac{d_2}{s_2}\right) > 1$, 
    we have
    \begin{align*}
        \psi_{12} = \frac{1}{s_1} \log\left(1+ \frac{d_2s_1}{s_2^2} \log\left(\frac{d_1}{s_1}\right)\right) &\leq \frac{1}{s_1} \log\left(2 \frac{d_2s_1}{s_2^2} \log\left(\frac{d_1}{s_1}\right)\right)\\
        & \leq \frac{\log(d_2/s_2) + 2\log(d_1)}{s_1}  \quad \text{ provided } s_1 \leq c' d_1\\
        & \leq \frac{1}{\bar cs_2} \frac{s_1}{ d_1}\left[\log\left(\frac{d_2}{s_2}\right) + 2\log(d_1)\right]\\
        & \leq \frac{1}{\bar c s_2} \left(1/e + 2\right)\\
        & \leq \frac{e^{-1} + 2}{\bar c \log(2)} \cdot \frac{1}{s_2} \log\left(1+ \frac{d_1 s_2}{s_1^2} \log\left(\frac{d_2}{s_2}\right)\right),\\
        & = \frac{e^{-1} + 2}{\bar c \log(2)} \psi_{21},
    \end{align*}
    which ensures that~\eqref{eq_simplification_dense_case_tilde} holds. 
    Moreover, we verify that equation~\eqref{eq_simplification_dense_case_upper_bound_tilde} holds as well. This is clear if $\frac{d_1 s_2}{s_1^2} \log\left(\frac{d_2}{s_2}\right) \leq 1$ since we have $\beta_{12} = 0$. Otherwise, we have
    \begin{align*}
        \psi_{21} = \frac{1}{s_2} \log\left(1+ \frac{d_1 s_2}{s_1^2} \log\left(\frac{d_2}{s_2}\right)\right) > \log(2) \frac{1}{s_2} \geq \frac{\log(2)\bar c}{s_1} \frac{d_1}{s_1} \geq \frac{e\log(2) \bar c}{s_1} \log\left(\frac{d_2}{s_2}\right) \geq e \bar c \log(2) \beta_{12},
    \end{align*}
    which ensures~\eqref{eq_simplification_dense_case_upper_bound_tilde} holds and concludes the proof of the first claim. 

    \item Now, let $\alpha>0$ and let $c_\mu, c>0$ be the constants defined in Lemma~\ref{lem_lb_dense_allrates}. 

    Assume first that $c_\mu^{-1}s_2\log\big(c\frac{d_2}{s_2}\big) \leq \Ct \frac{s_1^2}{d_1}$. 
    Recalling that $\frac{d_2}{s_2}\geq \frac{1}{c'}$ where $c'$ can be taken arbitrarily small, we have
    \begin{align*}
        R &\leq \phi_{12} = \frac{d_1}{s_1^2} \log\left(1+\frac{d_2}{s_2^2}\right) &\\
        & \leq \frac{d_1}{s_1^2} \log\left(2\frac{d_2}{s_2}\right) & \text{ provided $c'$ is small enough}\\
        & \leq 2\frac{d_1}{s_1^2} \log\left(c\frac{d_2}{s_2}\right) &\text{ provided $c'$ is small enough}\\
        &\leq \Ct^{-1}\frac{d_1}{s_1^2}\log\Big(c \frac{d_2}{s_2}\Big) + \frac{c_\mu}{s_2}\log\Big(\frac{\Ct}{2e}\Big),&
    \end{align*}
    which implies that if $\mu^2 \leq R$, then we have by Lemma~\ref{lem_lb_dense_allrates}.1. that 
    \[\EE\Big[\exp(\mu^2XY)\one\big(X \geq  \Ct \frac{s_1^2}{d_1} \big)\Big] < \alpha.\]
    Now, assume $\Ct \frac{s_1^2}{d_1} < c_\mu^{-1}s_2\log\big(c\frac{d_2}{s_2}\big) \leq s_1$. 
    Then we have 
    \begin{align*}
        R &\leq \psi_{21} = \frac{1}{s_2} \log\left( 1+ \frac{d_1s_2}{s_1^2} \log\left(\frac{d_2}{s_2}\right)\right)&\\
        & \leq \frac{1}{s_2} \log\left( 2\frac{d_1s_2}{s_1^2} \log\left(\frac{d_2}{s_2}\right)\right) & \text{ since } \frac{d_1s_2}{s_1^2} \geq 1\\
        &\leq \frac{1}{s_2}\log\Big(\frac{s_2d_1\log(c \frac{d_2}{s_2})}{c_\mu 2e s_1^2}\Big) &\text{ if $c'$ is small enough.}
    \end{align*}
    Therefore, whenever $\mu \leq c_\mu R$, we have by Lemma~\ref{lem_lb_dense_allrates}.2. that
    \[\EE\Big[\exp(\mu^2XY)\one\big(X \geq \lceil  \Ct \frac{s_1^2}{d_1}\rceil\big)\Big] < \alpha.\]

    Finally, assume that $s_1 < c_\mu^{-1}s_2\log\big(c\frac{d_2}{s_2}\big)$. 
    Then we have
    \begin{align*}
        \frac{1}{c_\mu}\log\left(c\frac{d_2}{s_2}\right) > \frac{s_1}{s_2} \geq \frac{1}{16e^4} \frac{d_1}{s_1} \geq \frac{1}{16e^3} \log\left(\frac{d_2}{s_2}\right),
    \end{align*}
    so that $\frac{s_2}{s_1} \log\left(\frac{d_2}{s_2}\right) \leq 16e^3$. 
    Therefore, we have
    \begin{align*}
        R &\leq \psi_{21} = \frac{1}{s_2} \log\left(1+\frac{d_1}{s_1^2}s_2\log\left(\frac{d_2}{s_2}\right)\right)\\
        & \leq \frac{1}{s_2} \log\left(1+16e^3\frac{d_1}{s_1}\right)\\
        & \leq \frac{1}{2s_2} \log\left(\frac{d_1}{2es_1}\right) \qquad \text{ if $c'$ is small enough} \\
        & \leq \frac{1}{c_\mu}\left[\frac{1}{s_1}\log\big(c \frac{d_2}{s_2}\big) + \frac{c_\mu}{s_2}\log\big(\frac{d_1}{2es_1}\big)\right]
    \end{align*}
    By Lemma~\ref{lem_lb_dense_allrates}.3, if $\mu^2 \leq c_\mu R \leq \frac{1}{s_1}\log\big(c \frac{d_2}{s_2}\big) + \frac{c_\mu}{s_2}\log\big(\frac{d_1}{2es_1}\big)$, we obtain
    \[\EE\Big[\exp(\mu^2XY)\one\big(X \geq \lceil  \Ct \frac{s_1^2}{d_1}\rceil\big)\Big] < \alpha,\]
    and the proof is complete.
    \end{enumerate}
\end{proof}
The Lemma below shows that $R$ is a lower bound on the rate emerging from Lemma~\ref{lem_lb_s1d1big_truncchi2}.
\begin{lemma}\label{lem_simplify_s1d1big_truncchi2}
    Suppose that $s_1^2 \leq \bar{c}d_1s_2$ for some $\bar{c} \in (0,(2e)^{-4})$ and $s_1 \leq c'd_1$, $s_2 \leq c' d_2$ for some sufficiently small $c'>0$. Furthermore, suppose that $\frac{s^2_2}{d_2} \leq 2e\frac{s_1^2}{d_1}$. 
    For some small enough constant $c_\mu$, we define the quantity 
    \begin{align*}
        M = \begin{cases}
            \frac{1}{s_1} \log\Big(1 + c_\mu \frac{d_2s_1}{s_2^2}\log(\frac{d_1}{2es_1})\Big) & \text{ if } s_1 \leq \frac{s_2}{\log\big(\frac{d_1s_2}{s_1^2}\big)}\\
            \frac{1}{s_2}\log\big(\frac{d_1s_2}{s_1^2}\big)\log\Big(1 + c_\mu \frac{d_2}{s_2\log\big(\frac{d_1s_2}{s_1^2}\big)}\log\big(\frac{d_1s_2}{2es_1^2\log(\frac{d_1s_2}{s_1^2})}\big)\Big) & \text{ otherwise.}
        \end{cases}
    \end{align*}
    \begin{enumerate}
        \item 
    Then for some sufficiently small constant $c>0$ depending only on $\mu$, it holds that $M \geq cR$. 
    Hence, for any $\alpha>0$, there exists a constant $c'_\mu>0$ such that, whenever $\mu^2 \leq c'_\mu R$, we have
    \[\EE\Big[\exp(\mu^2XY)\one\Big( 1 \vee \Ct \frac{s_1^2}{d_1} \leq X \leq s_1\wedge\frac{s_2}{\log\big(\frac{d_1s_2}{s_1^2}\big)}\Big)\Big] < \alpha.\]
    \item Moreover, we have
    \begin{align}
        R \geq \begin{cases}
            \frac{1}{2}(\psi_{12} + \beta_{21}) \land \phi_{12} \land \phi_{21} & \text{ if } s_1 \leq \frac{s_2}{\log\big(\frac{d_1s_2}{s_1^2}\big)}\\
            \frac{1}{2}(\psi_{21} + \beta_{12}) \land \phi_{12} \land \phi_{21} & \text{ otherwise. }
        \end{cases} \label{eq_lower_bound_M}
    \end{align}
    \end{enumerate}
\end{lemma}

\begin{proof}[Proof of Lemma ~\ref{lem_simplify_s1d1big_truncchi2}]
\phantom{ }

\begin{enumerate}
\item We start by showing that $M \geq c(\psi_{21} + \psi_{12})$. 
    Using  Lemma~\ref{lem_lb_s1d1big_truncchi2} and Lemma~\ref{lem_logs_constants}.(i), we obtain
\begin{align*}
    2M &\geq\frac{1}{s_1}\log\Big(1 + c_\mu \frac{d_2s_1}{s_2^2}\log(\frac{d_1}{2es_1})\Big) + \frac{\log\big(\frac{d_1s_2}{s_1^2}\big)}{s_2}\log\Big(1 + c_\mu \frac{d_2}{s_2\log\big(\frac{d_1s_2}{s_1^2}\big)}\log \big(\frac{d_1s_2}{2es_1^2\log(\frac{d_1s_2}{s_1^2})}\big)\Big) \\
    &\geq\frac{c_\mu}{s_1}\log\Big(1 + \frac{d_2s_1}{s_2^2}\log(\frac{d_1}{2es_1})\Big) + c_\mu\frac{\log\big(\frac{d_1s_2}{s_1^2}\big)}{s_2}\log\Big(1 +  \frac{d_2}{s_2\log\big(\frac{d_1s_2}{s_1^2}\big)}\log \big(\frac{d_1s_2}{2es_1^2\log(\frac{d_1s_2}{s_1^2})}\big)\Big) \\
    &\geq c_\mu \left[\frac{1}{2}\psi_{12} + \frac{\log\big(\frac{d_1s_2}{s_1^2}\big)}{s_2}\log\Big(1 + \frac{d_2}{s_2\log\big(\frac{d_1s_2}{s_1^2}\big)}\log \big(\frac{d_1s_2}{2es_1^2\log(\frac{d_1s_2}{s_1^2})}\big)\Big)\right].
\end{align*}
To obtain $M \geq c(\psi_{12} + \psi_{21})$, it now remains to show
\[c\psi_{21} \leq \frac{\log\big(\frac{d_1s_2}{s_1^2}\big)}{s_2}\log\Big(1 + \frac{d_2}{s_2\log\big(\frac{d_1s_2}{s_1^2}\big)}\log \big(\frac{d_1s_2}{2es_1^2\log(\frac{d_1s_2}{s_1^2})}\big)\Big).\]

 Since $\frac{d_1s_2}{s_1^2} \geq \bar{c}^{-1} \geq 2$, it follows that $\frac{1}{2}\log(1 + \frac{d_1s_2}{s_1^2}) \leq \log(\frac{d_1s_2}{s_1^2})$ by the inequality $\sqrt{1+x} \leq x$ that holds true for any $x \geq 2$. Furthermore, the inequality $\frac{d_1s_2}{s_1^2} \geq \bar{c}^{-1} \geq 16e^4$ implies that 
 \begin{align*}
 \frac{\log\Big(\frac{d_1s_2}{2es_1^2\log\big(\frac{d_1s_2}{s_1^2}\big)}\Big)}{\log \big(\frac{d_1s_2}{s_1^2}\big)} &= 1 - \frac{\log\Big(2e\log\big(\frac{d_1s_2}{s_1^2}\big)\Big)}{\log\big(\frac{d_1s_2}{s_1^2}\big)} \\
 &\geq \frac12.
 \end{align*}
Moreover, since $d_2/s_2 \geq 1/c' \geq 4$ provided $c' \leq 1/4$, we also have that $\log\left(\frac{d_2}{2s_2}\right) \geq \frac{1}{2} \log\left(\frac{d_2}{s_2}\right)$. Combining these observations yields 
\begin{align}
    \frac{\log\big(\frac{d_1s_2}{s_1^2}\big)}{s_2}\log\Big(1 +  \frac{d_2}{s_2\log\big(\frac{d_1s_2}{s_1^2}\big)}\log \big(\frac{d_1s_2}{2es_1^2\log(\frac{d_1s_2}{s_1^2})}\big)\Big) &\geq \frac{\frac{1}{2}\log\big(1 + \frac{d_1s_2}{s_1^2}\big)}{s_2}\log\Big(
\frac12\frac{d_2}{s_2}\Big)\nonumber\\
& \geq \frac14\frac{\log\big(1 + \frac{d_1s_2}{s_1^2}\big)}{s_2} \log\Big(
\frac{d_2}{s_2}\Big)\nonumber\\
& \geq \frac{1}{4} \frac{1}{s_2} \log\left(1 + \frac{d_1s_2}{s_1^2}\log\Big(
\frac{d_2}{s_2}\Big)\right) \text{ by Lemma~\ref{lem_logs_constants}.(ii)}\nonumber\\
& = \frac{1}{4} \psi_{21}.\label{eq_psi12/4}
\end{align}
It remains to invoke Lemma~\ref{lem_lb_s1d1big_truncchi2} to conclude the proof of the first claim. 

    \item 
We now prove equation~\eqref{eq_lower_bound_M}. 
Assume first that $s_1 \leq \frac{s_2}{\log\big(\frac{d_1s_2}{s_1^2}\big)}$. 
Then we obtain
\begin{align*}
    \frac{d_1}{s_1} \leq \frac{\frac{d_1s_2}{s_1^2}}{\log\big(\frac{d_1s_2}{s_1^2}\big)},\quad  \text{ hence } \quad \frac{d_1s_2}{s_1^2} \geq \frac{d_1}{s_1} \log\left(\frac{d_1}{s_1}\right), \quad \text{ i.e. } \quad \frac{1}{s_1} \geq \frac{1}{s_2} \log\left(\frac{d_1}{s_1}\right)
\end{align*}
by Lemma~\ref{lem_logs}.(ii). 
Note that, since $\frac{d_2s_1}{s_2^2} \geq \frac{d_1}{s_1} \geq \frac{1}{c'} $ can be made arbitrarily large by taking $c'$ small enough, we have
\begin{align*}
    \beta_{21} &\leq \frac{1}{s_2} \log\left(\frac{d_1}{s_1}\right) \leq \frac{1}{s_1} \leq \frac{1}{s_1}\log\left(1 + \frac{d_2s_1}{s_2^2}\log\left(\frac{d_1}{s_1}\right)\right)  = \psi_{12}.
\end{align*}
Hence,
\begin{align*}
    R &= (\psi_{12} + \psi_{21}) \land \phi_{12} \land \phi_{21} \\
    &\geq \psi_{12}  \land \phi_{12} \land \phi_{21} \\
    &\geq \frac{1}{2}(\psi_{12} + \beta_{21}) \land \phi_{12} \land \phi_{21}
\end{align*}
as claimed.
Assume now that $s_1 > \frac{s_2}{\log\big(\frac{d_1s_2}{s_1^2}\big)}$. 
By assumption, we have that $\frac{d_1s_2}{s_1^2} \geq \bar{c}^{-1} \geq 2e$.
Therefore, 
\begin{align*}
    \frac{\frac{d_1 s_2}{s_1^2}}{\log\left(\frac{d_1 s_2}{s_1^2}\right)}\leq \frac{d_1}{s_1}, \quad  \text{ hence } \quad \frac{d_1 s_2}{s_1^2} \leq 2 \frac{d_1}{s_1} \log\left(\frac{2d_1}{s_1}\right), \quad \text{ i.e. } \quad s_2 \leq 3 s_1 \log\left(\frac{d_1}{s_1}\right)
\end{align*}
by Lemma~\ref{lem_logs}.(i) and for $c'$ sufficiently small.
This yields
\begin{align*}
    \psi_{12} = \frac{1}{s_1} \log\Big(1 +  \frac{d_2s_1}{s_2^2}\log\left(\frac{d_1}{s_1}\right)\Big)
    & \geq \frac{1}{s_1} \log\left(1+\frac{d_2}{s_2}\right)\geq \frac{1}{s_1} \log\left(\frac{d_2}{s_2}\right) \geq \beta_{12}.
\end{align*}
Therefore, we obtain
\begin{align*}
    R &= (\psi_{12} + \psi_{21}) \land \phi_{12} \land \phi_{21} \\
    &\geq \psi_{21}  \land \phi_{12} \land \phi_{21} \\
    &\geq \frac{1}{2}(\psi_{21} + \beta_{12}) \land \phi_{12} \land \phi_{21}.
\end{align*}
This concludes the proof of equation~\eqref{eq_lower_bound_M}.
    \end{enumerate}
\end{proof}

The lemma below shows that $R$ is a lower bound on the rate emerging from Lemma~\ref{lem_lb_s2d2big_maxlin}.

\begin{lemma}\label{lem_simplify_s2d2big_maxlin}
    Suppose that $s_1^2 \leq \bar{c}d_1s_2$ for some $\bar{c} \in (0, (2e)^{-4})$ and $s_1 < \frac12 d_1$. 
    Furthermore, suppose that $\frac{s^2_2}{d_2} \geq 2e\frac{s_1^2}{d_1}$ and that $\lceil \Ct \frac{s_1^2}{d_1}\rceil \leq \Big\lfloor \frac{\Ct s_2^2/d_2}{\log\big(\frac{d_1s_2^2}{s_1^2d_2}\big)}\Big\rfloor\wedge s_1$. Then for any $\alpha > 0$, there exists a constant $c > 0$ such that if $\mu^2 \leq cR$, 
    \[\EE\left[\exp\big(\mu^2XY\big)\one\bigg(\Ct \frac{s_1^2}{d_1}  \leq X \leq  \frac{\Ct s_2^2/d_2}{\log\big(\frac{d_1s_2^2}{s_1^2d_2}\big)}\wedge s_1\bigg)\right] < \alpha.\]
\end{lemma}

\begin{proof}[Proof of Lemma ~\ref{lem_simplify_s2d2big_maxlin}]
    First, observe that the assumption $\lceil \Ct \frac{s_1^2}{d_1}\rceil \leq \big\lfloor \frac{\Ct s_2^2/d_2}{\log\big(\frac{d_1s_2^2}{s_1^2d_2}\big)}\big\rfloor\wedge s_1$ implies 
    $1 \leq \frac{\Ct s_2^2/d_2}{\log\big(\frac{d_1s_2^2}{s_1^2d_2}\big)}$. Rearranging terms, this implies 
    $$\frac{d_2}{s_2^2} \leq \Ct\log^{-1} \big(\frac{d_1s_2^2}{d_2s_1^2}\big) \leq \frac{\Ct}{\log(2e)},$$
     and therefore $\phi_{21} = \frac{d_2}{s_2^2}\log\big(1 + \frac{d_1}{s_1^2}\big)$ by adjusting the constant the definition of $\phi_{21}$ accordingly. Now, assume that $\mu^2 \leq \widetilde c_\mu R$ for some $\widetilde c_\mu>0$. Then, we have
        \begin{align*}
            \frac{1}{\widetilde c_\mu}\mu^2 &\leq R \\
            &\leq \phi_{21} \\
            &= \frac{d_2}{s_2^2}\log\big(1 + \frac{d_1}{s_1^2}\big) \\
            &\leq \Ct\frac{d_2}{s_2^2}\log\big(4e\big) \quad \text{(this follows from $\frac{d_1}{s_1^2}\leq 2e^{\Ct} \frac{d_2}{s_2^2}$)} \\
            &\leq \Ct\frac{d_2}{s_2^2}\log\big(4e\big)\log\big(\frac{\Ct}{2e}\vee \frac{d_1}{2es^2_1}\big).
        \end{align*}
        Letting $c'_\mu$ denote the constant from Lemma~\ref{lem_lb_s2d2big_maxlin}, we can now choose $\widetilde c_\mu = \frac{c_\mu'}{\Ct \log(4e)} $ sufficiently small so that, whenever $\mu^2 \leq \widetilde c_\mu R$, we have $\mu^2 \leq c_\mu' \frac{d_2}{s_2^2}\log\big(\frac{\Ct}{2e}\vee \frac{d_1}{2es^2_1}\big)$.         
        The conclusion follows by Lemma ~\ref{lem_lb_s2d2big_maxlin}, and the proof is complete. 
\end{proof}
The lemma below ensure that $R$ is a lower bound on the rate obtained in Lemma~\ref{lem_lb_s2d2big_truncchi2}.
\begin{lemma}\label{lem_simplify_s2d2big_truncchi2}
Suppose that $s_1^2 \leq \bar{c}d_1s_2$ for some $\bar{c} \in (0, (2e)^{-4})$ and $s_1 < c' d_1$, $s_2 \leq c' d_2$ for some sufficiently small constant $c'$. 
    Furthermore, suppose that $\frac{s^2_2}{d_2} \geq 2e\frac{s_1^2}{d_1}$. 
    \begin{enumerate}
        \item Then for any $\alpha > 0$, there exist constants $c_\mu > 0$ and $C_* \geq 1$ such that if $1 \vee \big\lceil \Ct \frac{s_1^2}{d_1}\big\rceil \vee \Big\lceil \frac{\Ct s_2^2/d_2}{\log\big(\frac{d_1s_2^2}{s_1^2d_2}\big)}\Big\rceil \leq s_1\wedge\big\lfloor\frac{s_2}{\log\big(\frac{d_1s_2}{s_1^2}\big)}\big\rfloor$and $\mu^2 \leq c_\mu R$, 
    \[\EE\Big[\exp(\mu^2XY)\one\Big( \Ct \frac{s_1^2}{d_1} \vee \frac{\Ct s_2^2/d_2}{\log\big(\frac{d_1s_2^2}{s_1^2d_2}\big)} \leq X \leq s_1\wedge\frac{s_2}{\log\big(\frac{d_1s_2}{s_1^2}\big)}\Big)\Big] < \alpha.\]
    \item Moreover, we have
    \begin{align*}
        R \geq \begin{cases}
            \big(\psi_{21} + \beta_{12}\big)\land \phi_{12} \land \phi_{21} & \text{ if } s_2 \leq s_1 \log\left(\frac{d_1}{s_1}\right)\\
            \big(\psi_{12} + \beta_{21}\big)\land \phi_{12} \land \phi_{21} & \text{ otherwise. } 
        \end{cases}
    \end{align*}
    \end{enumerate}
\end{lemma}
\begin{proof}[Proof of Lemma ~\ref{lem_simplify_s2d2big_truncchi2}]
\phantom{}

\begin{enumerate}
    \item 
    By the assumption $\frac{d_1s_2}{s_1^2} \geq \bar{c}^{-1} \geq 16e^4$, 
    we can repeat the steps leading to equation~\eqref{eq_psi12/4} to obtain 
    \begin{align*}
         \frac{\log\big(\frac{d_1s_2}{s_1^2}\big)}{s_2}\log\Big(1 + c_\mu \frac{d_2}{s_2\log\big(\frac{d_1s_2}{s_1^2}\big)}\log \big(\frac{d_1s_2}{2es_1^2\log(\frac{d_1s_2}{s_1^2})}\big)\Big)&\geq \frac{c_\mu}{4}\frac{\log\Big(1 + \frac{d_1s_2}{s_1^2}\log\big(\frac{d_2}{s_2}\big)\Big)}{s_2} \\
    &= \frac{c_\mu}{4} \psi_{21}.
    \end{align*}
    It immediately follows that, for some small enough constant $c$, we have 
    \begin{align*}
        \mu^2 &\leq cR \\
        &\leq c(\psi_{12} + \psi_{21}) \\
        &\leq \frac{1}{2}\frac{1}{s_1}\log\Big(1 + c_\mu \frac{d_2s_1}{s_2^2}\log(\frac{d_1}{2es_1})\Big) + \frac12 \frac{\log\big(\frac{d_1s_2}{s_1^2}\big)}{s_2}\log\Big(1 + c_\mu \frac{d_2}{s_2\log\big(\frac{d_1s_2}{s_1^2}\big)}\log \big(\frac{d_1s_2}{2es_1^2\log(\frac{d_1s_2}{s_1^2})}\big)\Big).
    \end{align*}
    The conclusion follows from Lemma ~\ref{lem_lb_s2d2big_truncchi2}.
    \item Assume now that $s_2 \leq s_1 \log\left(\frac{d_1}{s_1}\right)$. 
    Then we have
    \begin{align*}
        \psi_{12} &= \frac{1}{s_1} \log\left(1+\frac{d_2s_1}{s_2^2} \log\left(\frac{d_1}{s_1}\right)\right)\\
        & \geq \frac{1}{s_1} \log\left(1+\frac{d_2}{s_2} \right)\\
        & \geq \beta_{12}, 
    \end{align*}
    which yields the desired result. 
    Assume now that $s_2 > s_1 \log\left(\frac{d_1}{s_1}\right)$. 
    Then
    \begin{align*}
        \psi_{21} &= \frac{1}{s_2} \log\left(1+\frac{d_1s_2}{s_1^2} \log\left(\frac{d_2}{s_2}\right)\right)\\
        & > \frac{1}{s_2} \log\left(\frac{d_1}{s_1}\log\left(\frac{d_1}{s_1}\right)\log\left(\frac{d_2}{s_2}\right)\right)\\
        & \geq \frac{1}{s_2}\log\left(\frac{d_1}{s_1}\right)\\
        & \geq \beta_{21}.
    \end{align*}
    This concludes the proof.
    \end{enumerate}
\end{proof}

The lemma below ensure that $R$ is a lower bound on the rate obtained in Lemma~\ref{lem_max_test_lower_bound}.
\begin{lemma}\label{lem_simplification_rate_max_test}
    Suppose that $s_1^2 \leq \bar{c}d_1s_2$ for some $\bar{c} \in (0,(2e)^{-4})$ and $s_1 \leq c'd_1$ and $s_2 \leq c' d_2$ for some small enough constant $c'>0$. Furthermore, suppose that $s_2 < s_1\log\big((d_1s_2)/s_1^2\big)$ and $\frac{d_1}{s_1} \geq e \log\left(\frac{d_2}{s_2}\right)$. 
    \begin{enumerate}
        \item Then for any $\alpha_4 > 0$, there exist constants $c_\mu > 0$ and $C_* \geq 1$ such that if $\mu^2 \leq c_\mu R,$ then 
    \[\EE\Big[\exp(\mu^2XY)\one\Big(X \geq \frac{s_2}{\log\big(\frac{d_1s_2}{s_1^2}\big)} \lor  \Ct \frac{s_1^2}{d_1}\Big\rceil\Big)\Big] < \alpha_4.\]
    \item Moreover, it holds that 
    \begin{align*}
        R \geq \begin{cases}
            \frac{1}{4}\big((\psi_{21}+\beta_{12}) + (\psi_{12} + \beta_{21})\big) \land \phi_{12} \land \phi_{21}& \text{ if } s_1 \leq s_2\log\left(\frac{d_2}{s_2}\right)\\
            \frac{1}{2}\big(\psi_{21} + \beta_{12}\big) \land \phi_{12} \land \phi_{21} & \text{ otherwise.}
        \end{cases}
    \end{align*}
    \end{enumerate}
\end{lemma}
\begin{proof}[Proof of Lemma~\ref{lem_simplification_rate_max_test}]
\phantom{ }

\begin{enumerate}
    \item 

Note that the relation $s_2 < s_1\log\big((d_1s_2)/s_1^2\big)$ implies that 
\begin{align*}
    \frac{d_1s_2}{s_1^2} &< \frac{d_1}{s_1} \log\left(\frac{d_1s_2}{s_1^2}\right) \leq \frac{2d_1}{s_1} \log\left(\frac{2d_1}{s_1}\right) \quad \text{ by Lemma~\ref{lem_logs}.(i)}
\end{align*}
which yields 
\begin{align}
    s_2 \leq 4 s_1 \log(d_1/s_1). \label{eq_s2<s1_log(d1/s1)}
\end{align}
Now, we obtain 
\begin{align*}
    \frac{d_2s_1}{s_2^2} \log\left(\frac{d_1}{s_1}\right) = \frac{d_2}{s_2} \cdot \frac{s_1}{s_2} \log\left(\frac{d_1}{s_1}\right) \geq \frac{1}{4c'} \geq 1 
\end{align*}
provided $c' \leq 1/4$, and 
\begin{align*}
    \frac{d_1s_2}{s_1^2} \log\left(\frac{d_2}{s_2}\right) \geq \frac{1}{\bar c} \log\left(\frac{1}{c'}\right) \geq 1 
\end{align*}
provided $c'$ and $\bar c$ are small enough, which ensures that $\beta_{12 } = \frac{1}{s_1} \log\left(\frac{d_2}{s_2}\right)$ and $\beta_{21} = \frac{1}{s_2} \log\left(\frac{d_1}{s_1}\right)$. 

Assume first that  $s_1 < s_2\log(d_2/s_2)$. 
Then, we have 
\begin{align*}
    R &\leq \psi_{12} + \psi_{21}\\
    & = \frac{1}{s_1} \log\left(1+ \frac{d_2s_1}{s_2^2} \log\left(\frac{d_1}{s_1}\right)\right) + \frac{1}{s_2} \log\left(1+ \frac{d_1s_2}{s_1^2} \log\left(\frac{d_2}{s_2}\right)\right)\\
    & \leq \frac{1}{s_1} \log\left(2\frac{d_2}{s_2} \frac{s_1}{s_2} \log\left(\frac{d_1}{s_1}\right)\right) + \frac{1}{s_2} \log\left(1+ \frac{d_1}{s_1} 4\log\left(\frac{d_1}{s_1}\right)\log\left(\frac{d_2}{s_2}\right)\right)\\
    & \leq \frac{1}{s_1} \log\left(  \frac{d_2}{s_2} \right) + \frac{1}{s_1} \log\left(\frac{2s_1}{s_2} \log\left(\frac{d_1}{s_1}\right)\right) + \frac{1}{s_2} \log\left(\frac{8}{e} \left(\frac{d_1}{s_1} \right)^2\log\left(\frac{d_1}{s_1}\right)\right)\\
    & \leq \frac{1}{s_1} \log\left(  \frac{d_2}{s_2} \right) + \frac{1}{s_1} \log\left(\frac{2s_1}{s_2} \log\left(\frac{d_1}{s_1}\right)\right) + \frac{4}{s_2} \log\left(\frac{d_1}{s_1} \right)\\
    & \leq \frac{1}{s_1} \log\left(  \frac{d_2}{s_2} \right) + \frac{6}{s_2} \log\left(\frac{d_1}{s_1} \right)
    \\ &\leq 6 \beta_{12} + 6 \beta_{21}.
\end{align*}
In the second to last inequality, we used the fact that $\log\left(\frac{2s_1}{s_2} \log\left(\frac{d_1}{s_1}\right)\right) \leq \frac{2s_1}{s_2} \log\left(\frac{d_1}{s_1} \right)$ due to the inequality $\log(x) \leq x$ that holds true for any $x>0$. 
It follows from  Lemma~\ref{lem_max_test_lower_bound} that if $\mu^2 \leq c_\mu R$ for some sufficiently small constant $c_\mu>0$, then  
    $$\EE\Big[\exp(\mu^2XY)\one\Big(X \geq \frac{s_2}{\log\big(\frac{d_1s_2}{s_1^2}\big)}\lor  \Ct \frac{s_1^2}{d_1}\Big)\Big] < \alpha.$$

    Assume now that $s_1 > s_2\log\left(\frac{d_2}{s_2}\right)$. 
    Again, we have $R \leq \psi_{12} + \psi_{21}$. 
    We aim to show that
    \begin{align}
        \psi_{12} + \psi_{21} \leq \frac{4}{s_2}\log\big(\frac{s_2d_1\log(d_2/s_2)}{2es_1^2}\big),\label{eq_target_simplification_max_rate}
    \end{align}
    which will conclude the proof since it ensures that, if $\mu^2 \leq c_\mu R$ for some small enough $c_\mu>0$, then 
    $$\EE\Big[\exp(\mu^2XY)\one\Big(X \geq \frac{s_2}{\log\big(\frac{d_1s_2}{s_1^2}\big)}\lor  \Ct \frac{s_1^2}{d_1}\Big)\Big] < \alpha.$$ 
    by Lemma~\ref{lem_max_test_lower_bound}.
    We observe that the assumption $s_1^2 \leq \bar c d_1 s_2$ ensures that $\frac{d_1s_2}{s_1^2}\log(d_2/s_2) > 2$, which yields
    \begin{align*}
        \psi_{21} &= \frac{1}{s_2} \log\left(1+ \frac{d_1s_2}{s_1^2} \log\left(\frac{d_2}{s_2}\right)\right) \\
        & \leq \frac{2}{s_2} \log\left(\frac{d_1s_2}{s_1^2} \log\left(\frac{d_2}{s_2}\right)\right),
    \end{align*}   
    and further implies that 
    \begin{align}
        \frac{1}{s_2} \log\left(\frac{d_1s_2}{s_1^2} \log\left(\frac{d_2}{s_2}\right)\right) \leq \psi_{21}  \leq \frac{2}{s_2} \log\left(\frac{d_1s_2}{s_1^2} \log\left(\frac{d_2}{s_2}\right)\right).\label{eq_equiv_psi21}
    \end{align}
    Now, by the relations $s_2 \leq 4 s_1\log(d_1/s_1)$ and $s_j \leq c' d_j$ for some small enough constant $c'>0$, we also have $\frac{d_2s_1}{s_2^2} \log\left(\frac{d_1}{s_1}\right)>2$, which yields
    \begin{align*}
        \psi_{12} = \frac{1}{s_1} \log\left(1+ \frac{d_2s_1}{s_2^2} \log\left(\frac{d_1}{s_1}\right)\right) \leq \frac{2}{s_1} \log\left( \frac{d_2s_1}{s_2^2} \log\left(\frac{d_1}{s_1}\right)\right),
    \end{align*}
    so that
    \begin{align}
        \frac{1}{s_1} \log\left( \frac{d_2s_1}{s_2^2} \log\left(\frac{d_1}{s_1}\right)\right) \leq \psi_{12} \leq \frac{2}{s_1} \log\left( \frac{d_2s_1}{s_2^2} \log\left(\frac{d_1}{s_1}\right)\right) \label{eq_equiv_psi12}
    \end{align}
    To obtain the desired inequality~\eqref{eq_target_simplification_max_rate}, it therefore suffices to show that 
    \begin{align}
        \frac{1}{s_1} \log\left( \frac{d_2s_1}{s_2^2} \log\left(\frac{d_1}{s_1}\right)\right) \leq \frac{1}{s_2} \log\left( \frac{d_1s_2}{s_1^2} \log\left(\frac{d_2}{s_2}\right)\right)\label{eq_new_target}
    \end{align}
    under the assumptions $\frac{s_1}{\bar c d_1} \leq \frac{s_2}{s_1} \leq \frac{1}{\log(d_2/s_2)}$ and $\frac{d_j}{s_j} \geq \frac{1}{c'}$ for $j\in\{1,2\}$, as well as $\frac{d_1}{s_1} \geq e \log\left(\frac{d_2}{s_2}\right)$. 
    
    To prove this, we let $x = \frac{s_2}{s_1}$, and $a_j = \frac{d_j}{s_j}$ for $j=1,2$ and introduce the function 
\begin{align*}
    f(x) = \log\left(a_1 x \log(a_2)\right) - x \log\left(\frac{a_2\log(a_1)}{x}\right).
\end{align*}
We note that the desired inequality~\eqref{eq_new_target} is equivalent to showing that $f(x) \geq 0$ for any $x \in \left[\frac{1}{\bar c a_1}, \frac{1}{\log(a_2)}\right]$ where $a_1 \geq e \log(a_2)$ and $a_1, a_2 \geq \frac{1}{c'}$.
We have, for any $x \in \left[\frac{1}{\bar c a_1}, \frac{1}{\log(a_2)}\right]$
\begin{align*}
    f'(x) &= \frac{1}{x} - \log(a_2 \log(a_1)) + \log(x) + 1 \\
    & = \log(x) + \frac{1}{x} - \log\left(\frac{a_2 \log(a_1)}{e}\right),\\
    f''(x) &= \frac{1}{x} - \frac{1}{x^2} \leq 0,
\end{align*}
since $x \leq \frac{1}{\log(a_2)} \leq 1$. 
Therefore, $f$ is concave, and to prove that $f \geq 0$ over $\left[\frac{1}{\bar c a_1}, \frac{1}{\log(a_2)}\right]$, it suffice to show that $f\left(\frac{1}{\bar c a_1}\right) \geq 0$ and $f \left(\frac{1}{\log(a_2)}\right)\geq 0$. 
We have
\begin{align*}
    f\left(\frac{1}{\bar c a_1}\right) &= \log\left(\frac{\log(a_2)}{\bar c}\right) - \frac{\log(\bar c a_1)}{\bar c a_1}  - \frac{\log(a_2)}{\bar c a_1} - \frac{\log\log(a_1)}{\bar c a_1}\\
    &\geq \log\left(\frac{\log(a_2)}{\bar c}\right)   - \frac{\log(a_2)}{\bar c a_1} - 1\\
    & \geq \frac{1}{2}\log\left(\frac{\log(a_2)}{\bar c}\right) - 1 \geq 0.
\end{align*}
In the last step, we used the fact that, since $a_1 \geq e \log(a_2)$, we also have $\bar c a_1 \geq \frac{1}{2} \frac{\log(a_2)}{\log \log(a_2/\bar c)}$
provided $a_2$ is large enough, which can be enforced by choosing $c'>0$ small enough. 
Similarly, we have
\begin{align*}
    f\left(\frac{1}{\log(a_2)}\right) &= \log(a_1) - \frac{1}{\log(a_2)} \log\Big(a_2\log(a_2)\log(a_1)\Big)\\
    & = \log(a_1) - 1 - \frac{\log \log(a_2)}{\log(a_2)} - \frac{\log(a_1)}{\log(a_2)}\\
    & \geq \log(a_1) - 2 - \frac{\log(a_1)}{\log(a_2)}\\
    &
    \geq 0,
\end{align*}
provided $a_1$ and $a_2$ are larger than suitably large constants, which can be enforced by choosing $c'$ small enough. 
This concludes the proof of the first claim.
\item Assume first that $s_1 < s_2 \log \left(\frac{d_2}{s_2}\right)$. 
Then we have 
    $$\begin{aligned}
    \psi_{12} + \psi_{21}
    & = \frac{1}{s_1} \log\left(1+ \frac{d_2s_1}{s_2^2} \log\left(\frac{d_1}{s_1}\right)\right) + \frac{1}{s_2} \log\left(1+ \frac{d_1s_2}{s_1^2} \log\left(\frac{d_2}{s_2}\right)\right)\\
    & \geq \frac{1}{s_1} \log\left( 1+ \frac{d_2}{4s_2} \right) + \frac{1}{s_2} \log\left(  \frac{d_1s_2}{s_1^2} \log\left(\frac{d_2}{s_2}\right)\right)\\
    & > \frac{1}{4s_1} \log\left( 1+ \frac{d_2}{s_2} \right) + \frac{1}{s_2} \log\left(\frac{d_1}{s_1}\right)\\
    & \geq \frac{1}{4} \big(\beta_{12} + \beta_{21}\big),
\end{aligned}
$$
which yields $R \geq \frac{1}{8} \big((\psi_{21} + \beta_{12}) + (\psi_{12}+\beta_{21})\big) \land \phi_{12} \land \phi_{21}$, as desired. 

Assume now that $s_1 \geq s_2 \log \left(\frac{d_2}{s_2}\right)$. 
Combining~\eqref{eq_equiv_psi21}, \eqref{eq_equiv_psi12} and~\eqref{eq_new_target} yields $\psi_{12} + \psi_{21} \asymp \psi_{21}$. 
We now show that $\psi_{21} \asymp \psi_{21} + \beta_{12}$. 
By assumption, we have
\begin{align*}
    \frac{d_1s_2}{s_1^2} \log\left(\frac{d_2}{s_2}\right)\geq \frac{1}{\bar c} >1,
\end{align*}
which implies that $\beta_{12} = \frac{1}{s_1} \log\left(\frac{d_2}{s_2}\right)$. 
Therefore, we obtain
\begin{align*}
        \beta_{12} & =\frac{1}{s_1} \log\left( \frac{d_2}{s_2} \right) & \\
        & \leq \frac{1}{s_1} \log\left( \frac{d_2}{s_2} \right) + \frac{1}{s_1} \log\left(\frac{4s_1\log\left(\frac{d_1}{s_1}\right)}{s_2}\right) & \text{by equation~\eqref{eq_s2<s1_log(d1/s1)}}\\
        &\leq \frac{2}{s_1} \log\left( \frac{d_2s_1}{s_2^2} \log\left(\frac{d_1}{s_1}\right)\right) & \text{ using $\log(4x) \leq 2 \log(x), ~\forall x\geq 4$}\\
        &\leq \frac{2}{s_2} \log\left( \frac{d_1s_2}{s_1^2} \log\left(\frac{d_2}{s_2}\right)\right) & \text{by equation~\eqref{eq_new_target}}\\
        & \leq 2 \psi_{21} & \text{ by equation~\eqref{eq_equiv_psi21}.}
    \end{align*}
    Therefore, we have $\psi_{12} + \psi_{21} \geq \psi_{21} + \beta_{12}/2$, hence $R \geq \frac{1}{2}\big(\psi_{21} + \beta_{12} \big) \land \phi_{12} \land \phi_{21}$ and the proof is complete.
\end{enumerate}
\end{proof}

\begin{lemma}\label{lem_simplified_rate_for_UB}
        It holds that $R \gtrsim \big(\psi_{12} + \beta_{21}\big) \land \big(\psi_{21} + \beta_{12}\big) \land \phi_{12} \land \phi_{21}$.
    \end{lemma}

    \begin{proof}[Proof of Lemma~\ref{lem_simplified_rate_for_UB}]
    Assume first that, for some constant $c>0$, we have $s_1 \geq c d_1$ or $s_2 \geq c d_2$, and by symmetry, assume we have $s_1 \geq c d_1$. 
    Then 
    \begin{align*}
        \psi_{12} = \frac{1}{s_1} \log\left(1+ \frac{d_2}{s_2^2} \log\left(e {d_1 \choose s_1}\right)\right) \geq \frac{1}{s_1} \log\left(1+ \frac{d_2}{s_2^2} \right) \asymp \phi_{12},
    \end{align*}
    which yields that
    \begin{align*}
        R &= (\psi_{12} + \psi_{21}) \land \phi_{12} \land \phi_{21}\\
        & \asymp \phi_{12} \land \phi_{21} \\
        &\geq \big(\psi_{12} + \beta_{21}\big) \land \big(\psi_{21} + \beta_{12}\big) \land \phi_{12} \land \phi_{21}.
    \end{align*}
    From now on, assume that $s_1 \leq c d_1$ and $s_2 \leq c d_2$, and assume without loss of generality that $\frac{d_1}{s_1} \geq e\log\left(\frac{d_2}{s_2}\right)$, by lemma~\ref{lem_2.6}. 

    If $s_1^2 \geq \bar c d_1 s_2$, then the result follows by Lemma~\ref{lem_simplify_dense_allrates}. Now, assume that $s_1^2 < \bar c d_1 s_2$. If $\frac{k^2_2}{d_2} \leq 2e\frac{s_1^2}{d_1}$, then the result follows by Lemma~\ref{lem_simplify_s1d1big_truncchi2}. 
    Else, we have $s_1^2 \geq \bar c d_1 s_2$ and the result follows by Lemmas~\ref{lem_simplify_s2d2big_truncchi2} and~\ref{lem_simplification_rate_max_test}.
        
    \end{proof}

\subsection{Proof of Corollary 1}
\begin{proof}
        Recall that Theorem 1 establishes 
        \begin{equation}\label{eq_cor1_given}
            (\mu^*)^2 \asymp (\psi_{12} + \psi_{21}) \land \phi_{12} \land \phi_{21}.
        \end{equation}
        We aim to prove that
        \begin{equation}\label{eq_cor1_desired}
            (\mu^*)^2 \asymp \frac{d_1d_2}{s_1^2s_2^2} \land \frac{\log \big(\frac{d_2}{s_2}\big)}{s_1} + \frac{\log \big(\frac{d_1}{s_1}\big)}{s_2},
        \end{equation}
        under the following assumptions:
        $$e \leq \frac{d_1}{s_1} \land \frac{d_2}{s_2},\;\; \frac{\log \log\big( \frac{d_1}{s_1}\big)}{\log \big(\frac{d_2}{s_2}\big)} \lor \frac{\log \log\big( \frac{d_2}{s_2}\big)}{\log \big(\frac{d_1}{s_1}\big)} \leq 1,\;\; s_1 \log \big(\frac{d_1}{s_1}\big) \asymp s_2 \log \big(\frac{d_2}{s_2}\big).$$
        We remark that under the assumption $e \leq \frac{d_1}{s_1} \land \frac{d_2}{s_2}$, it holds
        \[\log {d_1 \choose s_1} \asymp s_1 \log\left(\frac{d_1}{s_1}\right), \;\;\log {d_2 \choose s_2} \asymp s_2 \log\left(\frac{d_2}{s_2}\right),\]
        which we will use throughout the proof to simplify the forms of $\psi_{12}$ and $\psi_{21}$ respectively. We proceed by cases. First, assume that $s_1\log(d_1/s_1) \gtrsim s_2^2 / d_2$. This implies
    \begin{align}
        \psi_{12} &\asymp \frac{1}{s_1} \log\left(\frac{d_2s_1}{s_2^2} \log\left(\frac{d_1}{s_1}\right)\right) \\
        &= \frac{\log \big(\frac{d_2}{s_2}\big)}{s_1} + \frac{\log \big(\frac{s_1}{s_2}\big)}{s_1} + \frac{\log \log \big(\frac{d_1}{s_1}\big)}{s_1} \\
        &\asymp \frac{\log \big(\frac{d_2}{s_2}\big)}{s_1} \quad \text{(since $\log \log \frac{d_1}{s_1} \leq \log \frac{d_2}{s_2}$ and $\frac{s_1}{s_2} \lesssim \log \frac{d_2}{s_2}$)}\\
        &\asymp \frac{\log \big(\frac{d_2}{s_2}\big)}{s_1} + \frac{\log \big(\frac{d_1}{s_1}\big)}{s_2} \quad \text{(since $s_1 \log \frac{d_1}{s_1} \asymp s_2 \log \frac{d_2}{s_2}$)}
    \end{align}
    where the second to last expression uses $\log (s_1/s_2) = o(s_1)$ and $\log \log(d_1/s_1) 
    \leq \log(d_2/s_2)$, and the final line uses $s_2 \log \big(d_2/s_2\big) \asymp s_1 \log \big(d_1/s_1\big)$. If $s_2\log(d_2/s_2) \gtrsim s_1^2 / d_1$, we may derive an analogous expression for $\psi_{21}$. Otherwise if $s_2\log(d_2/s_2) < s_1^2 / d_1$, we linearize the log and derive
    \begin{align}
        \psi_{21} &\asymp \frac{d_1}{s_1^2}\log\big(\frac{d_2}{s_2}\big) \\
        &\leq \frac{1}{s_2} \\
        &\leq \frac{\log \big(\frac{d_1}{s_1}\big)}{s_2} \quad \text{(since $es_1 \leq d_1$)} \\
        &\asymp \frac{\log \big(\frac{d_2}{s_2}\big)}{s_1} + \frac{\log \big(\frac{d_1}{s_1}\big)}{s_2},
    \end{align}
    with the final line using $s_2 \log \big(d_2/s_2\big) \asymp s_1 \log \big(d_1/s_1\big)$. In either case, we have
    \begin{equation}
        \psi_{12} + \psi_{21} \asymp \frac{\log \big(\frac{d_2}{s_2}\big)}{s_1} + \frac{\log \big(\frac{d_1}{s_1}\big)}{s_2},
    \end{equation}
    which matches the second term on the right hand side of~\eqref{eq_cor1_desired}. We now turn our attention to the term $\phi_{12}$; the analysis of $\phi_{21}$ is entirely analogous and we omit the details for brevity. If $d_2 < s_2^2$, then we linearize the logarithmic term and obtain
    \begin{equation}
        \phi_{12} \asymp \frac{d_1d_2}{s_1^2s_2^2},
    \end{equation}
    which clearly matches the first term in~\eqref{eq_cor1_desired}. Otherwise if $d_2 \gtrsim s_2^2$, it holds
    \begin{align}
        \phi_{12} &= \frac{d_1}{s^2_1}\log\Big(1 + \frac{d_2}{s_2^2}\Big) \\
        &\gtrsim \frac{d_1}{s_1^2} \\
        &\geq \frac{\log \big(\frac{d_2}{s_2}\big)}{s_1} \\
        &\asymp \frac{\log \big(\frac{d_2}{s_2}\big)}{s_1} + \frac{\log \big(\frac{d_1}{s_1}\big)}{s_2},
    \end{align}
    where the final two lines follow from $\log \log(d_1/s_1) \leq \log(d_2/s_2)$ and $s_2 \log \big(d_2/s_2\big) \asymp s_1 \log \big(d_1/s_1\big)$ respectively. Therefore, in this case both~\eqref{eq_cor1_given} and~\eqref{eq_cor1_desired} (following an identical calculation) evaluate to the order of $s_1^{-1}\log(d_2/s_2) + s_2^{-1}\log(d_1/s_1)$. This concludes the proof in the case $s_1\log(d_1/s_1) \gtrsim s_2^2 / d_2$.  If $s_1\log(d_1/s_1) < s_2^2 / d_2$ and $s_2\log(d_2/s_2) < s_1^2 / d_1$, it holds
    \begin{align}
        \psi_{12} + \psi_{21} &\asymp \frac{d_2}{s_2^2}\log\big(\frac{d_1}{s_1}\big) + \frac{d_1}{s_1^2}\log\big(\frac{d_2}{s_2}\big) \\
        &\geq \frac{d_1d_2}{s_1^2s_2^2} \\
        &= \phi_{12} \land \phi_{21}.
    \end{align}
    Furthermore, in this regime it holds
    \begin{align}
        \frac{d_1d_2}{s_1^2s_2^2} &\leq \frac{1}{s_1s_2\log(d_1/s_1)\log(d_2/s_2)} \\
        &\leq \frac{\log \big(\frac{d_2}{s_2}\big)}{s_1} + \frac{\log \big(\frac{d_1}{s_1}\big)}{s_2},
    \end{align}
    and therefore by Theorem 1 it holds
    \begin{equation}
        (\mu^{*})^2 \asymp \frac{d_1d_2}{s_1^2s_2^2} = \frac{d_1d_2}{s_1^2s_2^2} \land \frac{\log \big(\frac{d_2}{s_2}\big)}{s_1} + \frac{\log \big(\frac{d_1}{s_1}\big)}{s_2},
    \end{equation}
    and the proof is complete.
    \end{proof}
\section{Proof of upper bound}\label{app_ub}

\subsection{Analysis of the linear test}

\begin{lemma}\label{lem:linear-test}
    Let $\alpha \in (0,1)$ be given. There exist constants $C_\mu, C_h > 0$ such that if
    \[\mu^2 \geq C_\mu \frac{d_1d_2}{s^2_1s_2^2},\]
    then $\Delta^h_\lin$ with $h = C_h$ satisfies
    \[\cR(\Delta^h_\lin, \mu) \leq \alpha.\]
\end{lemma}
\begin{proof}
    We begin by controlling the Type I error. Recall that $t_\lin \sim N(0,1)$ under the null hypothesis. Then by an elementary calculation, for $C_h \geq \sqrt{2\log(2/\alpha)}$ it holds
    \begin{align*}
        \pr_0(\Delta^h_\lin = 1) &= \pr_{Z \sim N(0,1)}(Z  > h) \\
        &\leq e^{-h^2 / 2} \\
        &\leq \frac{\alpha}{2}.
    \end{align*}
    Now we consider the Type 2 error. Suppose that the mean matrix $ \EE[\Y] = \X$ is an arbitrary element of $\Theta(s_1, s_2, \mu)$. In this case, we have $t_\lin \sim N(\frac{s_1s_2}{\sqrt{d_1d_2}}\mu, 1)$. Note that
    \begin{align*}
        \EE_\X[t_\lin] &= \frac{s_1s_2}{\sqrt{d_1d_2}}\mu \\
        &\geq \sqrt{C_\mu} \quad \text{(by our assumption on $\mu$)} \\
        &\geq 2h,
    \end{align*}
    where the final inequality holds for $C_\mu$ taken sufficiently large. With this inequality in hand, we can control the Type II error as follows.
    \begin{align*}
        \pr_\X(\Delta^h_\lin = 0) &= \pr_\X(t_\lin \leq h) \\
        &= \pr_\X(t_\lin - \EE_X [t_\lin] \leq h - \EE_\X [t_\lin]) \\
        &= \pr_\X((t_\lin - \EE_X [t_\lin])^2 \geq (h - \EE_\X [t_\lin])^2) \\
        &\leq \pr_\X((t_\lin - \EE_X [t_\lin])^2 \geq (\EE_\X [t_\lin] / 2)^2) \\
        &\leq \frac{4}{(\EE_\X[t_\lin])^2} \\
        &\leq \frac{4}{C_\mu} \\
        &\leq \frac{\alpha}{2},
    \end{align*}
    where the final inequality holds for $C_\mu$ taken sufficiently large. Combining our bounds on the Type I and Type II errors yields
    \[\cR(\Delta^h_\lin, \mu) \leq \alpha,\]
    and the proof is complete.
\end{proof}

\subsection{Analysis of the Bonferroni corrected linear test}

\begin{lemma}\label{lem_maxlinear_test}
Let $\alpha \in (0,1)$ be given. There exists a constant $C_\mu > 0$ such that if
    \[\mu^2 \geq C_\mu \frac{d_2}{s_1s_2^2}\log\left(e {d_1 \choose s_1}\right),\]
    then the max-linear test with $h = \sqrt{\frac{2}{\alpha}} \vee \sqrt{2\log\Big(\frac{2e}{\alpha}{d_1 \choose s_1}\Big)}$ satisfies
    \[\cR(\Delta_{\linmax, 1}^h, \mu) \leq \alpha\]
\end{lemma}
\begin{proof}
    We begin with the Type I error.
    \begin{align*}
    \bbP_0(\Delta_{\linmax, 1}^h = 1) &= \bbP_0(t_{\linmax, 1} > h) \\
        &\leq {d_1 \choose s_1}\bbP_{Z \sim N(0,1)}(Z > h)\quad \text{(union bound)} \\
        &\leq {d_1 \choose s_1} e^{-h^2/2} \quad \text{(Chernoff bound)} \\
        &\leq {d_1 \choose s_1}\exp\left(-\log\left(\frac{2e}{\alpha}{d_1 \choose s_1}\right)\right) \quad \text{(by our choice of $h$)} \\
        &\leq \frac{\alpha}{2}
    \end{align*}
    Now suppose that $\X$ is an arbitrary matrix in $\Theta(s_1, s_2, \mu)$ with support $S_1 \times S_2$. Define the statistic
    \[t_{S_1} = \sum_{j = 1}^{d_2}\bar{Y}_{S_1, j} = \frac{1}{\sqrt{s_1d_2}}\sum_{j = 1}^{d_2}\sum_{i \in  S_1}Y_{ij}.\]
    Under the alternative hypothesis, $\EE_{\X} [t_{S_1}] = \mu s_2\sqrt{s_1 / d_2}$ and $\text{var}(t_{S_1}) = 1$. By our assumption on $\mu$, we have
    \begin{align*}
        \EE_{\X} [t_{S_1}] &= \mu s_2\sqrt{\frac{s_1}{d_2}} \\
        &\geq \sqrt{C_{\mu}\log\left(e {d_1 \choose s_1}\right)} \\
        &\geq 2h,
    \end{align*}
    where the final inequality holds for $C_\mu$ taken sufficiently large. We then have
    \begin{align*}
        \bbP_{\X}(\Delta_{\linmax}^h = 0)
        &= \bbP_{\X}(t_{\linmax, 1} \leq h) \\
        &\leq \bbP_\X(t_{S_1} \leq h) \\
        &= \bbP_\X(t_{S_1} -  \EE_{\X} [t_{S_1}]\leq h - \EE_{\X} [t_{S_1}]) \\
        &= \bbP_\X((t_{S_1} -  \EE_{\X} [t_{S_1 }])^2\geq (h - \EE_{\X} [t_{S_1 }])^2) \\
        &\leq \bbP_\X((t_{S_1} -  \EE_{\X} [t_{S_1}])^2\geq h^2) \\
        &\leq \frac{1}{h^2} \quad \text{(Markov's inequality)} \\
        &\leq \frac{\alpha}{2} \quad \text{(by our choice of $h$)}
    \end{align*}
    Combining these two bounds gives us 
    \[\cR(\Delta_{\linmax, 1}^h, \mu) \leq \alpha\]
    and the proof is complete.
\end{proof}

\subsection{Analysis of the linear truncated $\chi^2$ test}

\begin{lemma}\label{lem:linear-trunc-test}
    Let $\alpha \in (0,1)$ be given. Suppose that there exists a constant $c > 0$ such that $\frac{d_2}{s_2^2} \geq c$. Define 
    \[\tau^* = \sqrt{C\log\Big(1 + \frac{d_2}{s_2^2}\Big)}\]
    for $C = \big(\log(1 + c)\big)^{-1} \lor 2$. There exist constants $C_\mu, C_h > 0$ such that if 
    \[\mu^2 \geq C_\mu \frac{d_1}{s_1^2}\log\Big(1 + \frac{d_2}{s_2^2}\Big),\]
    then $\Delta_{\chisqlin, 1}^h$ with $\tau = \tau^*$ and $h = C_h s_2 \log\big(1 + \frac{d_2}{s_2^2}\big)$
    \[\cR(\Delta_{\chisqlin, 1}^h, \mu) \leq \alpha.\]
\end{lemma}
\begin{proof}
    We begin by considering the Type I error. Note that for any $x$ satisfying $x \geq b$ for a constant $b > 0$, there exists a constant $B > 0$ such that $B \log(1 + x) \geq \sqrt{x / (1 + x)}$. Therefore, for $C_h$ taken sufficiently large, it holds
    \begin{align*}
        h &= C_h s_2 \log\big(1 + \frac{d_2}{s_2^2}\big) \\
        &\geq 9\Big(\sqrt{\frac{d_2\log(2/\alpha)}{1 + \frac{d_2}{s_2^2}}} + \log(2/\alpha)\Big) \\
        &\geq 9\Bigg(\sqrt{d_2e^{-(\tau^*)^2/2}\log(2/\alpha)} + \log(2/\alpha)\Bigg).
    \end{align*}
    Under the null hypothesis, $\bar{Y}_j \sim N(0,1)$. Thus, directly using Lemma 5 of \cite{liu2021minimax} we have 
    \[\bbP_0(t_{\chisqlin, 1} > h) \leq \bbP_0\left(t_{\chisqlin, 1} > 9\Big(\sqrt{d_2e^{-(\tau^*)^2/2}\log(2/\alpha)} + \log(2/\alpha)\Big)\right) \leq  \frac{\alpha}{2}.\]
    Now we consider the Type II error. Suppose that the mean matrix $ \EE[\Y] = \X$ is an arbitrary element of $\Theta(s_1, s_2, \mu)$. Under this alternative, it holds that $\bar{Y}_j \sim N(\mu \frac{s_1}{\sqrt{d_1}}, 1)$. Recall that $\tau^* \geq 1$ by construction. Therefore, for $C_\mu \geq 8$, we can apply Lemmas 6 and 7 of \cite{liu2021minimax} to conclude
    \begin{align*}
        \EE_\X\big[t_{\chisqlin}\big] &\geq \frac{s_2s_1^2}{2d_1}\mu^2, \\
        \var_\X\big(t_{\chisqlin}\big) &\leq C_1s_2\frac{s_1^2}{d_1}\mu^2 + C_1 (\tau^*)^3e^{-(\tau^*)^2/2}(d_2 - s_2)
    \end{align*}
    where $C_1 > 0$ is the absolute constant obtained from Lemma 7 of \cite{liu2021minimax} Therefore, for $C_\mu \geq 4C_h$, we have 
    \begin{align*}
        \EE_\X\big[t_{\chisqlin}\big] &\geq \frac{s_2s_1^2}{2d_1}\mu^2 \\
        &\geq C_\mu \frac{s_2}{2}\log\Big(1 + \frac{d_2}{s^2_2}\Big) \\
        &\geq 2C_h s_2 \log\Big(1 + \frac{d_2}{s^2_2}\Big) \\
        &= 2h.
    \end{align*}
    Using this calculation, we can control the Type II error as follows. 
    \begin{align*}
        \bbP_\X \big(t_{\chisqlin} \leq h\big) &= \bbP_\X \big(t_{\chisqlin} - \EE_\X[t_{\chisqlin}] \leq h - \EE_\X[t_{\chisqlin}]\big) \\
        &= \bbP_\X \big((t_{\chisqlin} - \EE_\X[t_{\chisqlin}])^2 \geq (h - \EE_\X[t_{\chisqlin}])^2\big) \\
        &\leq \bbP_\X \big((t_{\chisqlin} - \EE_\X[t_{\chisqlin}])^2 \geq (\EE_\X[t_{\chisqlin}] / 2)^2\big) \\
        &\leq 4\frac{\var_\X(t_{\chisqlin})}{(\EE_\X[t_{\chisqlin}])^2} \quad \text{(Markov's inequality)} \\
        &\leq 4\frac{C_1s_2\frac{s_1^2}{d_1}\mu^2 + C_1 (\tau^*)^3e^{-(\tau^*)^2/2}(d_2 - s_2)}{\big(\frac{s_2s_1^2}{2d_1}\mu^2\big)^2} \\
        &= 4(\text{I} + \text{II}).
    \end{align*}
    We control each of these terms separately.
    \begin{align*}
        \text{I} &= \frac{C_1s_2\frac{s_1^2}{d_1}\mu^2}{\big(\frac{s_2s_1^2}{2d_1}\mu^2\big)^2} \\
        &= \frac{4C_1}{C_\mu s_2\log\big(1 + \frac{d_2}{s_2^2}\big)} \\
        &\leq \frac{4C_1}{C_\mu s_2\log\big(1 + c\big)} \\
        &\leq \frac{\alpha}{16}
    \end{align*}
    where the final inequality holds for $C_\mu$ taken sufficiently large. Finally, we have 
    \begin{align*}
        \text{II} &= \frac{C_1 (\tau^*)^3e^{-(\tau^*)^2/2}(d_2 - s_2)}{\big(\frac{s_2s_1^2}{2d_1}\mu^2\big)^2} \\
        &\leq \frac{C_1 (\tau^*)^3e^{-(\tau^*)^2/2}d_2}{\big(\frac{s_2s_1^2}{2d_1}\mu^2\big)^2} \\
        &\leq \frac{C_1 C^{3/2} (\log(1 + \frac{d_2}{s_2^2}))^{3/2}d_2}{\big(\frac{s_2s_1^2}{2d_1}\mu^2\big)^2(1 + \frac{d_2}{s_2^2})} \quad \text{(by our choice of $\tau^*$)} \\
        &\leq \frac{C'\big(\log(1 + \frac{d_2}{s_2^2})\big)^{3/2}}{C_\mu^2\big(\log(1 + \frac{d_2}{s_2^2})\big)^2}\frac{\frac{d_2}{s_2^2}}{1 + \frac{d_2}{s_2^2}} \quad \text{(by our lower bound on $\mu^2$)}\\ 
        &\leq \frac{C'}{C_\mu^2 \sqrt{\log(1 + \frac{d_2}{s_2^2})}} \\
        &\leq \frac{C'}{C_\mu^2 \sqrt{\log(1 + c)}} \\
        &\leq \frac{\alpha}{16},
    \end{align*}
    where $C' = 8C_1C^{3/2}$ and again the final inequality holds for $C_\mu$ sufficiently large. Combining these bounds on the Type I and Type II errors gives us
    \[\cR(\Delta_{\chisqlin}^h, \mu) \leq \alpha\]
    and the proof is complete.
\end{proof}
\subsection{Analysis of Bonferroni corrected truncated $\chi^2$ test}
\begin{lemma}\label{lem:max-trunc-test}
    Let $\alpha \in (0,1)$ be given. Suppose that there exists a constant $c > 0$ such that $\frac{d_2}{s_2^2}\log {d_1 \choose s_1} \geq c$. Define
    \[\tau^* = \sqrt{C\log\Big(1 + \frac{d_2}{s_2^2}\log {d_1 \choose s_1}\Big)}\]
    for $C = \big(\log(1 + c)\big)^{-1} \lor 2$. There exist constants $C_\mu, C_h > 0$ such that if
    \[\mu^2 \geq C_\mu\Bigg(\frac{1}{s_1}\log\Big(1 + \frac{d_2}{s_2^2}\log {d_1 \choose s_1}\Big) + \frac{1}{s_2s_1}\log {d_1 \choose s_1}\Bigg),\]
    then $\Delta_{\chisqmax, 1}^h$ with $\tau = \tau^*$ and $h = C_h\Big(s_2\log\Big(1 + \frac{d_2}{s_2^2}\log {d_1 \choose s_1}\Big) + \log {d_1 \choose s_1}\Big)$ satisfies 
    \[\cR(\Delta_{\chisqmax}^h, \mu) \leq \alpha.\]
\end{lemma}
\begin{proof}
    We begin by controlling the Type I error. 
    First, for $C_h$ taken sufficiently large note that
    \begin{align*}
        h &= C_h\Big(s_2\log\Big(1 + \frac{d_2}{s_2^2}\log {d_1 \choose s_1}\Big) + \log {d_1 \choose s_1}\Big) \\
        &\geq C_h\Big(s_2\log(1 + c) + \log {d_1 \choose s_1}\Big) + 
        \log(2/\alpha)\\
        &\geq C_h\Big(\log(1+c)\sqrt{\frac{d_2(\log {d_1 \choose s_1} + \log(2/\alpha))}{1 + \frac{d_2}{s_2^2}\log {d_1 \choose s_1}}} + \log {d_1 \choose s_1}\Big) + \log(2/\alpha) \\
        &\geq 9\Bigg(\sqrt{d_2 e^{-(\tau^*)^2/2}\Big(\log{d_1 \choose s_1} + \log(2/\alpha)\Big)} + \log {d_1 \choose s_1} + \log(2/\alpha)\Bigg) \\
        &:= \bar{h}.
    \end{align*}
    We have 
    \begin{align*}
        \bbP_0\big(\Delta_\chisqmax^h = 1\big) &= \bbP_0\big(t_{\chisqmax} > h\big) \\
        &\leq {d_1 \choose s_1}\bbP_{Z_j \iid N(0,1)}\Big(\sum_{j = 1}^{d_2}(Z_j^2 - \beta_{\tau^*})\one(|Z_j| > \tau^*) > h)\Big) \quad \text{(by union bound)} \\
        &\leq {d_1 \choose s_1}\bbP_{Z_j \iid N(0,1)}\Big(\sum_{j = 1}^{d_2}(Z_j^2 - \beta_{\tau^*})\one(|Z_j| > \tau^*) > \bar{h})\Big) \\
        &\leq {d_1 \choose s_1} e^{-\log(\frac{2}{\alpha}{d_1 \choose s_1})} \quad \text{(using Lemma 5 of \cite{liu2021minimax})} \\
        &= \frac{\alpha}{2}.
    \end{align*}
    Now we consider the Type 2 error. Let the mean matrix $\X$ be an arbitrary element of $\Theta(s_1, s_2, \mu)$ with support $S_1 \times S_2$. Note that $\tau^* \geq 1$ by definition. Therefore, for $C_\mu \geq 8$, we apply Lemmas 6 and 7 of \cite{liu2021minimax} to conclude
    \begin{align*}
        \EE_{\X}\big[t_{\chisqmax}\big] &\geq \frac{s_1s_2}{2}\mu^2 \\
        \var_\X(t_{\chisqmax}) &\leq C_1s_1s_2\mu^2+C_1 (\tau^*)^3e^{-(\tau^*)^2/2}(d_2 - s_2),
    \end{align*}
    where $C_1$ is the absolute constant from Lemma 7 of \cite{liu2021minimax}. For $C_\mu \geq 4C_h$, it holds
    \begin{align*}
        \EE_{\X}\big[t_{\chisqmax}\big] &\geq \frac{s_1s_2}{2}\mu^2 \\
        &\geq \frac{C_\mu}{2}\Big(s_2\log\Big(1 + \frac{d_2}{s_2^2}\log{d_1 \choose s_1}\Big) + \log {d_1 \choose s_1}\Big) \\ 
        &\geq 2h.
    \end{align*} 
    Thus, we can control the Type 2 error as follows.
    \begin{align*}
        \bbP_\X \big(t_{\chisqmax} \leq h\big) &= \bbP_\X \big(t_{\chisqmax} - \EE_\X[t_{\chisqmax}] \leq h - \EE_\X[t_{\chisqmax}]\big) \\
        &= \bbP_\X \big((t_{\chisqmax} - \EE_\X[t_{\chisqmax}])^2 \geq (h - \EE_\X[t_{\chisqmax}])^2\big) \\
        &\leq \bbP_\X \big((t_{\chisqmax} - \EE_\X[t_{\chisqmax}])^2 \geq (\EE_\X[t_{\chisqmax}] / 2)^2\big) \\
        &\leq 4 \frac{\var_\X(t_{\chisqmax})}{(\EE_\X[t_{\chisqmax}])^2} \quad \text{(Markov's inequality)} \\
        &\leq 4\frac{C_1s_2s_1\mu^2 + C_1 (\tau^*)^3e^{-(\tau^*)^2/2}(d_2 - s_2)}{\big(\frac{s_2s_1}{2}\mu^2\big)^2} \\
        &= 4(\text{I} + \text{II}).
    \end{align*}
    We now control these two terms separately.
    \begin{align*}
        \text{I} &= \frac{C_1s_2s_1\mu^2}{\big(\frac{s_2s_1}{2}\mu^2\big)^2} \\
        &\leq \frac{4C_1}{C_\mu^2s_2\log\big(1 + c\big) + C_\mu^2\log{d_1 \choose s_1}} \\
        &\leq \frac{\alpha}{16},
    \end{align*}
    where the final inequality holds for $C_\mu$ taken sufficiently large. Finally, we have 

    \begin{align*}
        \text{II} &= \frac{C_1 (\tau^*)^3e^{-(\tau^*)^2/2}(d_2 - s_2)}{\big(\frac{s_2s_1}{2}\mu^2\big)^2} \\
        &\leq \frac{C_1 (\tau^*)^3e^{-(\tau^*)^2/2}d_2}{\big(\frac{s_2s_1}{2}\mu^2\big)^2} \\
        &\leq \frac{C_1 (C \log(1 + \frac{d_2}{s_2^2}\log {d_1 \choose s_1}))^{3/2}d_2}{\big(\frac{s_2s_1}{2}\mu^2\big)^2(1 + \frac{d_2}{s_2^2}\log {d_1 \choose _1})} \quad \text{(by our choice of $\tau^*$)} \\
        &\leq \frac{C'\big(\log(1 + \frac{d_2}{s_2^2}\log {d_1 \choose s_1})\big)^{3/2}}{C_\mu^2\big(\log(1 + \frac{d_2}{s_2^2}\log {d_1 \choose s_1})\big)^2}\frac{\frac{d_2}{s_2^2}\log {d_1 \choose s_1}}{1 + \frac{d_2}{s_2^2}\log {d_1 \choose s_1}} \quad \text{(by our lower bound on $\mu^2$)}\\ 
        &\leq \frac{C'}{C_\mu^2 \sqrt{\log(1 + \frac{d_2}{s_2^2}\log {d_1 \choose s_1})}} \\
        &\leq \frac{C'}{C_\mu^2 \sqrt{\log(1 + c)}} \\
        &\leq \frac{\alpha}{16},
    \end{align*}
    where $C' = 8C_1C^{3/2}$ and again the final inequality holds for $C_\mu$ sufficiently large. Combining these bounds on the Type I and Type II errors gives us
    \[\cR(\Delta_{\chisqlin}^h, \mu) \leq \alpha\]
    and the proof is complete.
\end{proof}

\subsection{Proof of Proposition 1}
\begin{proof}
    
Here we give a proof of Proposition 1, which establishes the upper bound stated in Theorem 1. Recall the construction of our proposed optimal test. We let 
\begin{align*}
    \tilde R = \big(\psi_{12} + \beta_{21}\big) \land \big(\psi_{21} + \beta_{12}\big) \land \phi_{12} \land \phi_{21}.
\end{align*}
We then define our optimal test as
\begin{align*}
    \Delta^* = \begin{cases}
        \Delta_a^{h_1, h_2} & \text{ if } \tilde R = \phi_{12}, \\
        \Delta_b^{h'_1, h'_2} & \text{ if } \tilde R = \phi_{21}, \\
        \Delta_c^{h_3, h_4} & \text{ if } \tilde R = \psi_{12} + \beta_{21},\\
        \Delta_d^{h_3', h_4'} & \text{ if } \tilde R = \psi_{21} + \beta_{12}.
    \end{cases}
\end{align*}
where
\begin{align*}
    &\Delta_a^{h_1,h_2} = \begin{cases}
        \Delta^{h_1}_{\chisqlin, 1} & \text{if  $\frac{d_2}{s_2^2} \geq 1$,} \\ \Delta^{h_2}_{\lin} & \text{otherwise}
    \end{cases} \qquad \text{ and } 
    \qquad \Delta_b^{h_1',h_2'} = \begin{cases}
        \Delta^{h_1'}_{\chisqlin, 2} & \text{if  $\frac{d_1}{s_1^2} \geq 1$,} \\ \Delta^{h_2'}_{\lin} & \text{otherwise,}
    \end{cases}
\end{align*}
and 
\begin{align*}
    &\Delta_c^{h_3,h_4} = \begin{cases}
        \Delta^{h_3}_{\chisqmax, 1} & \text{if  $\frac{d_2}{s_2^2}\log\!\left(e{d_1 \choose s_1}\right) \geq 1$,} \\ \Delta^{h_4}_{\linmax, 1} & \text{otherwise}
    \end{cases} \quad \text{ and } 
    \quad \Delta_d^{h_3',h_4'} = \begin{cases}
        \Delta^{h_3'}_{\chisqmax, 2} & \text{if  $\frac{d_1}{s_1^2}\log\!\left(e{d_2 \choose s_2}\right) \geq 1$,} \\ \Delta^{h_4'}_{\linmax, 2} & \text{otherwise}.
    \end{cases} 
\end{align*}
By Lemma~\ref{lem_simplified_rate_for_UB}, we always have $(\psi_{12} + \psi_{21}) \land \phi_{12} \land \phi_{21} \gtrsim \tilde R$. 
Assume first that  $\tilde R = \big(\psi_{12} + \beta_{21}\big) \land \big(\psi_{21} + \beta_{12}\big)$, and, by symmetry, assume that we have  $\tilde R = \psi_{12} + \beta_{21}$. 
In this case, we have $\Delta^* = \Delta^{h_3, h_4}_{c}$. If $\frac{d_2}{s_2^2}\log\!\left(e{d_1 \choose s_1}\right) \geq 1$, the result follows by Lemma~\ref{lem:max-trunc-test}, and otherwise the result follows from Lemma \ref{lem_maxlinear_test}. We can proceed similarly in the case where $\tilde R = \psi_{21} + \beta_{12}$. 

Assume now that $\tilde R = \phi_{12}$ and $d_2 \geq s_2^2$. Then the result follows by Lemma~\ref{lem:linear-trunc-test}. Similarly, if  
$\tilde R = \phi_{21}$ and $d_1 \geq s_1^2$, the result follows by Lemma~\ref{lem:linear-trunc-test}.
Finally, if none of the conditions above are satisfied, then we have $\Delta^* = \Delta^{h_2}_{\lin}$ and the result follows by Lemma~\ref{lem:linear-test}. 
The proof is complete.
\end{proof}

\subsection{Analysis of adaptive tests}

\subsubsection{Adaptive Bonferroni corrected linear test}

\begin{lemma}\label{lem_ub_ada_maxlinear}
    Let $\alpha \in (0,1)$ be given. Define
    \[\Omega_1 = \left\{\frac{d_1}{2^m}: m \in \{0, \ldots, \lfloor \log_2(d_1) \rfloor\}\right\}.\]
    The adaptive Bonferroni corrected linear test is defined as
    \[\Delta^{\text{ada}}_{\linmax, 1} = \max_{s_1 \in \Omega_1}\one\big(t_{\linmax, 1}(s_1) > h(s_1)\big),\]
    for $h(s_1)$ to be specified later. Let $s_1^*$ and $s_2^*$ denote the true row and column sparsity under the alternative hypothesis. Suppose that $d_1 \geq 3$ and $s_1^* \leq d_1 - 1$. Then there exist constants $C_h, C_{\mu} > 0$ such that if
    \[\mu^2 \geq C_\mu \frac{d_2}{s^*_1(s^*_2)^2}\log\left(e {d_1 \choose s^*_1}\right),\]
    then the adaptive Bonferroni corrected linear test with cutoff $h(s_1) = C_h\sqrt{2\log\left(e {d_1 \choose s_1}\log_2(d_1)\right)}$ satisfies
    \[\cR(\Delta^{\text{ada}}_{\linmax, 1}, \mu) \leq \alpha.\]
\end{lemma}
\begin{proof}
    We consider the Type I error first. Under the null hypothesis, we have
    \begin{align*}
        \bbP_0\left(\Delta^{\text{ada}}_{\linmax, 1} = 1\right) &= \bbP_0\left(\max_{s_1 \in \Omega_1}\max_{J_1 \in \mathcal P_{s_1}\!(d_1)}\frac{1}{\sqrt{d_2}}\sum_{j = 1}^{d_2}\bar{Y}_{J_1, j} > h(s_1)\right) \\
        &\leq \sum_{s_1 \in \Omega_1}{d_1 \choose s_1}\bbP_{Z \sim N(0,1)}\big(Z > h(s_1)\big) \quad \text{(union bound)} \\
        &\leq \sum_{s_1 \in \Omega_1}{d_1 \choose s_1}e^{-h^2(s_1)/2} \quad \text{(Chernoff bound)} \\
        &= \sum_{s_1 \in \Omega_1}{d_1 \choose s_1} \exp\left[-C_h^2\log\left(e {d_1 \choose s_1}\log_2(d_1)\right)\right] \\
        &\leq \sum_{s_1 \in \Omega_1}{d_1 \choose s_1} \left(\frac{1}{{d_1 \choose s_1}\log_2(d_1)}\right)^{C_h^2} \\
        &= \left(\frac{1}{{d_1 \choose s_1}\log_2(d_1)}\right)^{C_h^2- 1} \quad \text{(for $C_h > 1$)} \\
        &\leq \left(\frac{1}{\log_2(3)}\right)^{C_h^2- 1} \quad \text{(for $d_1 \geq 3, {d_1 \choose s_1} \geq 1$)} \\
        &\leq \frac{\alpha}{2},
    \end{align*}
    where the final inequality holds for $C_h$ taken sufficiently large. Now we turn our attention to the Type II error. Let the mean matrix $\mathbf{X}$ be an arbitrary element of $\Theta(s^*_1, s^*_2, \mu)$ where $s^*_1$ and $s^*_2$ denote the true row and column sparsity levels respectively with support $S^*_1 \times S^*_2$. We have not assumed that $s_1^* \in \Omega_1$. However, it is true that $s_1^+ := d_12^{-\lceil \log_2(d_1/s_1)\rceil}$ belongs to $\Omega_1$. Since $s_1^* \leq s_1^+$, there exists a set $S_1^+ \in \mathcal{P}_{s_1^+}(d_1)$ such that $S_1^* \subset S_1^+$. Define the statistic
    \[t_{S^+_1} = \sum_{j = 1}^{d_2}\bar{Y}_{S^+_1, j} = \frac{1}{\sqrt{s^+_1d_2}}\sum_{j = 1}^{d_2}\sum_{i \in  S^+_1}Y_{ij}.\]
    Under the alternative hypothesis, $\EE_{\X} [t_{S^+_1}] = (\mu s_2s^*_1) / \sqrt{s_1^+d_2}$ and $\text{var}(t_{S_1}) = 1$. By our assumption on $\mu$ and the construction of $s_1^+$, we have
    \begin{align*}
        \EE_{\X} [t_{S_1}] &= \mu \frac{s_2s^*_1}{\sqrt{s_1^+d_2}} \\
        &\geq \mu s_2\sqrt{\frac{s^*_1}{2d_2}} \\
        &\geq \sqrt{\frac{C_{\mu}}{2}\log\left(e {d_1 \choose s^*_1}\right)} \\
        &\geq 2h,
    \end{align*}
    where the final inequality holds for $C_\mu$ taken sufficiently large, and relies on $s_1^* \leq d_1 - 1$ to obtain the inequality $\log\left(e{d_1 \choose s_1}\right) \geq \log \log_2 (d_1)$. The remainder of the proof follows exactly as that of Lemma \ref{lem_maxlinear_test}; we omit the details. The proof is complete.
\end{proof}

\subsubsection{Adaptive truncated $\chi^2$ test}
\begin{lemma}\label{lem_ub_ada_linear}
    Let $\alpha \in (0,1)$ be given. Suppose that there exists a constant $c > 0$ such that $\frac{d_2}{(s_2^*)^2} \geq c$ and that $s_2^* \geq 3$, where $s_2^*$ denotes the true column sparsity under $H_1$. Let $S_2 = \{s_2 : cs_2^2 \leq d_2, s_2 \geq 3\}$, and let
    \[\Omega_2 = \left\{\frac{d_2}{2^m} \in S_2 : m \in \{0, \ldots, \lfloor \log_2(d_2) \rfloor\}\right\}\]
    Given $s_2 \in \Omega_2$, we define
    \[\tau(s_2) = \sqrt{C\log\Big(1 + \frac{d_2}{s_2^2}\Big)},\]
where $C = \big(\log(1 + c)\big)^{-1}$. Letting $h(s_2) = C_hs_2\log\big(1 + \frac{d_2}{s_2^2}\big)$ for $C_h$ to be specified later, the adaptive truncated $\chi^2$ test is defined as
\[\Delta^{\text{ada}}_{\chisqlin} = \max_{s_2 \in \Omega_2}\one(t^{\tau(s_2)}_{\chisqlin} \geq h(s_2)).\]
There exist constants $C_h, C_\mu > 0$ such that if
\[\mu^2 \geq C_\mu \frac{d_1}{(s^*_1)^2}\log\big(1 + \frac{d_2}{(s^*_2)^2}\big),\]
then $\Delta^{\text{ada}}_{\chisqlin}$ with  satisfies
\[\cR(\Delta^{\text{ada}}_{\chisqlin}) \leq \alpha.\]
\end{lemma}
\begin{proof}

 We begin by controlling the Type I error. This proof follows the structure of the proof of Theorem 23 in \cite{kotekal2023minimax}. For every $s_2 \in S_2$, we define $u(s_2) = \frac{C^2_h}{81}\log^2(1 + \frac{d_2}{s_2^2}) \land \frac{C_h}{9} s_2\log(1 + \frac{d_2}{s_2^2})$. By direct calculation, we have
\begin{align*}
    9\Big(\sqrt{d_2 e^{-\tau(s_2)^2}u(s_2)} + u(s_2)\Big) 
    &\leq  9\Big(\sqrt{d_2 \frac{s_2^2}{s_2^2 + d_2} e^{-C}u(s_2)} + u(s_2)\Big) \\
    &\leq 9\Big(s_2e^{-C}\frac{C_h}{9}\log\big(1 + \frac{d_2}{s_2^2}\big) + \frac{C_h}{9}\log\big(1 + \frac{d_2}{s_2^2}\big)\Big) \\
    &\leq C_h s_2 \log\big(1 + \frac{d_2}{s_2^2}\big) \\
    &= h(s_2).
\end{align*}
Then it holds
\begin{align*}
    \bbP_0\Big(\Delta^{\text{ada}}_{\chisqlin} = 1\Big) &\leq \sum_{s_2 \in S_2}\bbP_0\Big(t^{\tau(s_2)}_{\chisqlin} \geq h(s_2)\Big) \\
    &\leq \sum_{s_2 \in S_2}\bbP_0\Big(t^{\tau(s_2)}_{\chisqlin} \geq 9\Big(\sqrt{d_2 e^{-\tau(s_2)^2}u(s_2)} + u(s_2)\Big)\Big) \\
    &\leq \sum_{s_2 \in S_2}e^{-u(s_2)} \quad \text{(by Lemma 5 of \cite{liu2021minimax})} \\
    &\leq \sum_{s_2 \in S_2}e^{-\frac{C^2_h}{81}\log^2\big(1 + \frac{d_2}{s_2^2}\big)} + \sum_{s_2 \in S_2}e^{-\frac{C_h}{9}s_2 \log\big(1 + \frac{d_2}{s_2^2}\big)}.
\end{align*}
We control the second sum first. Note that for $s_2 \geq 3$ and $C_h$ taken sufficiently large, it holds
\[\frac{C_h}{9}\log\big(1 + \frac{d_2}{s_2^2}\big) \geq \frac{C_h}{9}\log\big(1 + c\big) \geq \frac{\log(s_2^2)}{s_2} + \log\big(\frac{8\pi^2}{6\alpha}\big).\]
Then by taking $C_h$ sufficiently large, it holds
\begin{align*}
    \sum_{s_2 \in S_2}e^{-\frac{C_h}{9}s_2 \log\big(1 + \frac{d_2}{s_2^2}\big)} &\leq \frac{\alpha}{16} + \frac{\alpha}{16} + \sum_{s_2 \in S_2 / \{1, 2\}}e^{-\frac{C_h}{9}s_2 \log\big(1 + \frac{d_2}{s_2^2}\big)} \\
    &\leq \frac{\alpha}{8} +  \sum_{s_2 \in S_2 / \{1, 2\}}e^{-\log(s_2^2) - \log(8\pi^2/6\alpha)} \\
    &\leq \frac{\alpha}{8} + \frac{6\alpha}{8\pi^2}\sum_{s_2 = 1}^{\infty}\frac{1}{s_2^2} \\
    &= \frac{\alpha}{4}.
\end{align*}
To control the first sum, we have
\[\sum_{s_2 \in S_2}e^{-\frac{C^2_h}{81}\log^2\big(1 + \frac{d_2}{s_2^2}\big)} \leq \int_{-\infty}^{\infty}\exp\Big(-\frac{C_h^2}{81}x^2\Big)\text{d}x \leq \sqrt{\frac{81\pi}{C^2_h}} \leq \frac{\alpha}{4},\]
where the final inequality holds for $C_h$ taken sufficiently large. Therefore, we have shown
\[\bbP_0(\Delta^{\text{ada}}_{\chisqlin}) \leq \frac{\alpha}{2}.\]
We now turn our attention to the Type II error. Let the mean matrix $\mathbf{X}$ be an arbitrary element of $\Theta(s^*_1, s^*_2, \mu)$ where $s^*_1$ and $s^*_2$ denote the true row and column sparsity levels respectively. While we have assumed that $s_2^* \in S_2$, we have not assumed that $s_2^* \in \Omega_2$. However, since $s_2 \mapsto d_2 / s_2^2$ is decreasing in $s_2$, it is true that $s_2^- := d_22^{-\lceil \log_2(d_2/s_2)\rceil}$ belongs to $\Omega_2$. Note that $\tau(s_2^-) \geq 1$ by construction. Therefore, to apply Lemmas 6 and 7 of \citet{liu2021minimax} to the statistic $t^{\tau(s^-_2)}_{\trunc-\chi^2}$, we first to verify that $\EE[\bar{Y}_j] \geq 8\tau(s_2^-)$. Observe
\begin{align*}
    \EE[\bar{Y}_j] &= \frac{s_1^*}{\sqrt{d_1}}\mu \\
    &\geq \sqrt{C_\mu\log\left(1 + \frac{d_2}{(s_2^*)^2}\right)} \\
    &\geq \sqrt{C_\mu\log\left(1 + \frac{d_2}{2(s_2^-)^2}\right)} \\
    &\geq 8\tau(s_2^-),
\end{align*}
where the final inequality holds for $C_\mu$ taken sufficiently large. We may therefore invoke Lemmas 6 and 7 of \citet{liu2021minimax} to lower and upper bound the mean and variance of $t^{\tau(s^-_2)}_{\trunc-\chi^2}$ respectively, and then proceed according to the proof of Lemma \ref{lem:linear-trunc-test}. We omit the details for brevity. The proof is complete.
\end{proof}

\subsubsection{Adaptive Bonferroni-corrected truncated $\chi^2$ test}
We define
\begin{equation}\label{eq:adaptive-maxchisq-tau}
    \tau(s_1, s_2) = \sqrt{C\log\Big(1 + \frac{d_2}{s_2^2}\log \Big({d_1 \choose s_1}\log_2(d_1)\log_2(d_2)\Big)\Big)},
\end{equation}
and
\begin{equation}\label{eq:adaptive-maxchisq-u}
    u(s_1, s_2) = 9\Big(\sqrt{d_2e^{-(\tau(s_1, s_2))^2/2}\log(\frac{2}{\alpha}{d_1\choose s_1}\log_2(d_1)\log_2(d_2))} + \log\Big(\frac{2}{\alpha}{d_1 \choose s_1}\log_2(d_1)\log_2(d_2)\Big)\Big),
\end{equation}
where $\alpha > 0$ and $C > 0$ are constants to be specified later. For a constant $c > 0$, we define the set $\Omega$ as 
\[\Omega = \Big\{(s_1, s_2) : \frac{d_2}{s_2^2}\log {d_1 \choose s_1} \geq c\Big\}.\]
In practice, we should take $c$ small enough such that $(s_1^*, s_2^*) \in \Omega$, assuming that such a $c$ exists. We take $\bar{\Omega}$ as a dyadic partition of $\Omega:$
\begin{equation}
    \bar{\Omega} = \Big\{\big(\frac{d_1}{2^{m_1}}, \frac{d_2}{2^{m_2}}\big) \in 
\Omega : m_1 \in \{0, \ldots, \lfloor \log_2(d_1)\rfloor \}, m_2 \in \{0, \ldots, \lfloor \log_2(d_2) \rfloor\}  \Big\}.
\end{equation}
\begin{lemma}\label{lem_ub_ada_max}
    Let $\alpha \in (0,1)$ be given. Suppose that the pair of true sparsity levels satisfies $(s^*_1, s^*_2) \in \Omega$. Suppose that there exists a constant $C' > 0$ such that $s_2^* + \log \log d_1 \geq C' \log \log d_2$, $(d_1/s_1^*) \land  (d_2/s_2^*) \geq e$, and $d_1 \land d_2 \geq 8$. The adaptive Bonferroni-corrected truncated $\chi^2$ test is defined as
\[\Delta^{\text{ada}}_{\chisqmax} = \max_{(s_1, s_2) \in \bar{\Omega}} \one\Big(t^{(s_1, s_2)}_{\chisqmax} \geq u(s_1, s_2)\Big),\]
where $\tau(s_1, s_2)$ and $u(s_1, s_2)$ are defined as in (\ref{eq:adaptive-maxchisq-tau}) and (\ref{eq:adaptive-maxchisq-u}) respectively. We take the constant $C$ in the definition of $\tau(s_1, s_2)$ as $C = (\log(1 + c))^{-1} \lor 2$. Then there exists a constant $C_\mu > 0$ such that if
\[\mu^2 \geq C_\mu\Bigg(\frac{1}{s^*_1}\log\Big(1 + \frac{d_2}{(s^*_2)^2}\log {d_1 \choose s^*_1}\Big) + \frac{1}{s^*_2 s^*_1}\log {d_1 \choose s^*_1}\Bigg),\]
then $\Delta^{\text{ada}}_{\chisqmax}$  satisfies
\[\cR(\Delta^{\text{ada}}_{\chisqlin}) \leq \alpha.\]
\end{lemma}
\begin{proof}
    We begin by controlling the Type 1 error.
    \begin{align*}
\bbP_0\big(\Delta^{\text{ada}}_{\chisqmax} = 1\big) &\leq \sum_{(s_1, s_2) \in \bar{\Omega}} {d_1 \choose s_1}\bbP_{Z \iid N(0,1)}\Big(\sum_{j = 1}^{d_2}\big(Z^2_{j} - \beta_{\tau(s_1,s_2)}\big)\one(|Z_j| > \tau(s_1, s_2)) \geq u(s_1, s_2)\Big) \\
    &\leq \sum_{(s_1, s_2) \in \bar{\Omega}} {d_1 \choose s_1}e^{-\log(\frac{2}{\alpha}{d_1\choose s_1}\log_2(d_1)\log_2(d_2))} \quad \text{(by Lemma 5 of \cite{liu2021minimax})} \\
    &\leq \frac{\alpha}{2}.
    \end{align*}
    We remark that our application of Lemma 5 of \cite{liu2021minimax} relies on the fact that each candidate pair $(s_1, s_2)$ belongs to the set $\Omega$, and that the constant $C$ in the definition of $\tau(s_1, s_2)$ is sufficiently large. Now, we turn our attention to the Type 2 error. Let the mean matrix $\mathbf{X}$ be an arbitrary element of $\Theta(s^*_1, s^*_2, \mu)$ where $s^*_1$ and $s^*_2$ denote the true row and column sparsity levels respectively. While $(s_1^*, s_2^*) \in \Omega$ by assumption, we have no guarantee that $(s_1^*, s_2^*) \in \bar{\Omega}$. For $k \in \{1, 2\}$, we define $s_k^{-} = d_k 2^{-\lceil \log_2(d_k/s_k) \rceil}$ and $s_k^{+} = d_k 2^{-\lfloor \log_2(d_k/s_k) \rfloor}$ so that $s_k^{-} \leq s_k \leq s_k^{+}$. Since $(s_1, s_2) \mapsto d_2s_2^{-2}\log {d_1 \choose s_1}$ is decreasing in $s_2$, increasing in $s_1$ for $s_1 \leq d_1/2$, and decreasing in $s_1$ for $s_1 \geq d_1/2$, it always holds that either $(s_1^{+}, s_2^{-})$ or $(s_1^{-}, s_2^{-})$ belong to $\bar{\Omega}$ depending on the value of $s_1$. Since we have assumed that $d_1 / s_1^* \geq e$, we need only consider the case $s_1 \leq d_1/2$. We first verify that $\EE_{\mathbf{X}}[t^{(s_1^+, s_2^-)}_{\chisqmax}] \geq 2 u(s_1^+, s_2^-)$. We aim to apply Lemma 6 of \cite{liu2021minimax}. To do so, we first need to establish that $s_1^* s_2^{-}\mu^2 \geq \tau^2(s_1^+, s_2^-)$. By our assumption on the size of $\mu$, it is sufficient to show that $s_1^* s_2^{-}\mu^2 \geq C(\log \log d_1 + \log \log d_2)$ to establish this fact, where $C$ is the constant specified in the definition of $\tau(s_1, s_2)$. Note that
\begin{align*}
     \frac{(s_1^*)\cdot (s_2^-)}{2}\mu^2 &\geq \frac{s_1^*s_2^-}{2}C_\mu\Bigg(\frac{1}{s^*_1}\log\Big(1 + \frac{d_2}{(s^*_2)^2}\log {d_1 \choose s^*_1}\Big) + \frac{1}{s^*_2s^*_1}\log {d_1 \choose s^*_1}\Bigg) \\
    &\geq C'_\mu s_2^- + \frac{C_\mu}{2}\frac{s_2^-}{s_2^*}\log {d_1 \choose s_1^*} \\
    &= C'_\mu s_2^*\frac{s_2^-}{s_2^*} + \frac{C_\mu}{2}\frac{s_2^-}{s_2^*}\log {d_1 \choose s_1^*} \\
    &\geq \frac{C'_{\mu}}{4}s_2^* + \frac{C_\mu}{4}\log {d_1 \choose s_1^*},
\end{align*}
where $C'_\mu = \frac{C_\mu}{8}\log(1 + c)$ and the final inequality uses $2s_2^- \geq s_2^*$ which holds by the construction of $s_2^-$. Since $d_1/s_1^* \geq e$ by assumption, it is clear that
\begin{align*}
    \frac{C_\mu}{4}\log {d_1 \choose s_1^*} &\geq \frac{C_\mu}{4} s_1^* \log\left(\frac{d_1}{s_1^*}\right) \\
    &\geq C \log \log_2 d_1,
\end{align*}
Furthermore, invoking this logic again and the fact that $\log_2(x) \geq \log (x)$ for $x \geq 1$, we have
\begin{align*}
    \frac{C'_{\mu}}{4}s_2^* + \frac{C_\mu}{4}\log {d_1 \choose s_1^*} &\geq \frac{C'_{\mu}}{4}s_2^* + \frac{C_\mu}{4} \log \log d_1 \\
    &\geq C''_\mu \log \log d_2 \quad \text{(by hypothesis, with $C''_\mu = 4^{-1}(C_\mu \land C'_\mu) \cdot C'$)} \\
    &\geq C \log \log d_2,
\end{align*}
where the final inequality holds for $C_\mu$ taken sufficiently large. Thus, it holds that $s_1^* s_2^- \mu^2 \geq \tau^2(s_1^+, s_2^-)$ and we may invoke Lemma 6 of \cite{liu2021minimax} to obtain
\begin{align*}
    \EE_{\mathbf{X}}[t^{(s_1^+, s_2^-)}_{\chisqmax}] &\geq \frac{s_1^*s_2^-}{2}\mu^2.
\end{align*}
Additionally, it holds
\begin{align*}
    u(s^+_1, s^-_2) &= 9\Big(\sqrt{d_2e^{-(\tau(s^+_1, s^-_2))^2/2}\log(\frac{2}{\alpha}{d_1\choose s^+_1}\log_2(d_1)\log_2(d_1))} + \log\Big(\frac{2}{\alpha}{d_1 \choose s^+_1}\log_2(d_1)\log_2(d_1)\Big)\Big) \\
    &\leq \frac92Cs^*_2 + 9\log\Big(\frac{2}{\alpha}{d_1 \choose s^+_1}\log_2(d_1)\log_2(d_2)\Big)\Big) \quad \text{(by the form of $\tau(s_1, s_2)$ and using $s_2^- \leq s_2^*$)} \\
    &= \frac92Cs^*_2 + 9\log\Big(\frac{2}{\alpha}{d_1 \choose s^+_1}\Big) + 9 \log \log_2 d_1 + 9 \log \log_2 d_2.
\end{align*}
By our construction of $s_1^+$, it holds
\begin{align*}
    \frac{\log {d_1 \choose s^+_1}}{\log {d_1 \choose s^*_1}} &\leq \frac{s_1^+\log(ed_1 / s_1^+)}{s_1^*\log(d_1 / s_1^*)} \\
    &\leq 2\frac{\log(ed_1/s_1^*)}{\log(d_1 / s_1^*)} \\
    &= 2\Big(1 + \frac{e}{\log(d_1/s_1^*)}\Big) \\
    &\leq 2(1 + e) \leq 8,
\end{align*}
where the second to final inequality uses $d_1 / s_1^* \geq e$. Therefore, $u(s_1^+, s_2^-)$ satisfies
\[u(s_1^+, s_2^-) \leq \frac{9}{2}Cs_2^* + 9\log(2/\alpha) + 72\log {d_1 \choose s_1^*} + 9\log \log_2 d_1 + 9 \log \log_2 d_2.\]
Using our previously established lower bound on $\EE_{\mathbf{X}}[t^{(s_1^+, s_2^-)}_{\chisqmax}]$ and the assumption $\mu$, it is clear that
\[\frac12\EE_{\mathbf{X}}[t^{(s_1^+, s_2^-)}_{\chisqmax}] \geq \frac{9}{2}Cs_2^* + 9\log(2/\alpha) + 72\log {d_1 \choose s_1^*} + 9\log \log_2 d_1 + 9 \log \log_2 d_2,\]
for $C_\mu$ taken sufficiently large. We have thus shown that $\EE_{\mathbf{X}}[t^{(s_1^+, s_2^-)}_{\chisqmax}] \geq 2 u(s_1^{+}, s_2^{-})$. With this result in hand, we control the Type 2 error using Chebyshev's inequality.
\begin{align*}
    \pr_\X \left(\Delta^{\text{ada}}_{\chisqmax} = 0\right) &= \pr\left(\bigcap_{(s_1, s_2 \in \bar{\Omega})} \bigg\{t_{\chisqmax}^{(s_1, s_2)} \leq u(s_1, s_2) \bigg\}\right) \\
    &\leq \pr\left(t_{\chisqmax}^{(s^+_1, s^-_2)} \leq u(s^+_1, s^-_2)\right) \\
    &= \pr\left(t_{\chisqmax}^{(s^+_1, s^-_2)} - \EE[t_{\chisqmax}^{(s^+_1, s^-_2)}]\leq u(s^+_1, s^-_2) - \EE[t_{\chisqmax}^{(s^+_1, s^-_2)}]\right) \\
    &= \pr\left(\bigg(t_{\chisqmax}^{(s^+_1, s^-_2)} - \EE[t_{\chisqmax}^{(s^+_1, s^-_2)}]\bigg)^2\geq \bigg(u(s^+_1, s^-_2) - \EE[t_{\chisqmax}^{(s^+_1, s^-_2)}]\bigg)^2\right) \\
    &\leq \pr\left(\bigg(t_{\chisqmax}^{(s^+_1, s^-_2)} - \EE[t_{\chisqmax}^{(s^+_1, s^-_2)}]\bigg)^2\geq \bigg(\frac12\EE[t_{\chisqmax}^{(s^+_1, s^-_2)}]\bigg)^2\right) \\
    &\leq 4\frac{\var \Big(t_{\chisqmax}^{(s^+_1, s^-_2)}\Big)}{\Big(\EE[t_{\chisqmax}^{(s^+_1, s^-_2)}]\Big)^2}.
\end{align*}
By Lemma 7 of \cite{liu2021minimax}, it holds
\[\var_\X(t^{(s_1^+, s_2^-)}_{\chisqmax}) \leq\frac{C_1}{4}s^*_1s^-_2\mu^2 + d_2C_1 (\tau(s_1^+, s_2^-))^3e^{-(\tau(s_1^+, s_2^-))^2/2},\]
where $C_1 \geq 1$ is universal constant. With this inequality in hand, alongside the established fact that $\EE[t_{\chisqmax}^{(s^+_1, s^-_2)}] \geq 2^{-1}s_1^*s_2^-\mu^2$, we can bound the Type 2 error as follows:
\begin{align*}
    \pr_\X \left(\Delta^{\text{ada}}_{\chisqmax} = 0\right) &\leq 4\frac{\var \Big(t_{\chisqmax}^{(s^+_1, s^-_2)}\Big)}{\Big(\EE[t_{\chisqmax}^{(s^+_1, s^-_2)}]\Big)^2} \\
    &\leq 4\frac{\frac{C_1}{4}s^*_1s^-_2\mu^2 + d_2C_1 (\tau(s_1^+, s_2^-))^3e^{-(\tau(s_1^+, s_2^-))^2/2}}{\Big(\EE[t_{\chisqmax}^{(s^+_1, s^-_2)}]\Big)^2} \\
    &\leq \frac{4C_1}{s_1^*s_2^-\mu^2} + 4\sqrt{2}C_1\frac{d_2(\tau(s_1^+, s_2^-))^3e^{-(\tau(s_1^+, s_2^-))^2/2}}{\mu^4(s_1^*s_2^-)^2} \\
    &= \text{I} + \text{II}.
\end{align*}
We control each of these terms separately.
\begin{align*}
    \text{I} &= \frac{4C_1}{s_1^*s_2^-\mu^2} \\
    &\leq \frac{4C_1}{C_\mu s_2^-\log\Big(1 + \frac{d_2}{(s_2^*)^2}\log {d_1 \choose s_1^*}\Big)} \\
    &\leq \frac{4C_1}{C_\mu s_2^-\log(1 + c)} \\
    &\leq \frac{\alpha}{4},
\end{align*}
with the final inequality holding for $C_\mu$ taken sufficiently large. Turning our attention to II, it holds
\begin{align*}
    \text{II} &= 4\sqrt{2}C_1\frac{d_2(\tau(s_1^+, s_2^-))^3e^{-(\tau(s_1^+, s_2^-))^2/2}}{\mu^4(s_1^*s_2^-)^2} \\
    &\leq 4\sqrt{2}C_1C^{3/2}\frac{d_2\big(\log(1 + \frac{d_2}{(s_2^{-})^2}\log {d_1 \choose s_1^+ }\log_2(d_1)\log_2(d_2))\big)^{3/2}}{\mu^4(s_1^*s_2^-)^2\big(1 + \frac{d_2}{(s_2^{-})^2}\log {d_1 \choose s_1^+ }\log_2(d_1)\log_2(d_2)\big)} \\
    &\leq 4\sqrt{2}C_1C^{3/2}\frac{\big(\log(1 + \frac{d_2}{(s_2^{-})^2}\log {d_1 \choose s_1^+ }\log_2(d_1)\log_2(d_2))\big)^{3/2}}{\mu^4(s_1^*)^2\log {d_1 \choose s_1^+ }\log_2(d_1)\log_2(d_2)} \\
    &\leq \frac{4\sqrt{2}C_1C^{3/2}}{C^4_\mu}\frac{\big(\log(1 + \frac{d_2}{(s_2^{-})^2}\log {d_1 \choose s_1^+ }\log_2(d_1)\log_2(d_2))\big)^{3/2}}{\big(\log(1 + \frac{d_2}{(s_2^{*})^2}\log {d_1 \choose s_1^* })\big)^{2}\log_2(d_1)\log_2(d_2)} \\
    &\leq \frac{4\sqrt{2}C_1C^{3/2}}{C^4_\mu}\frac{\big(\log(1 + \frac{d_2}{(s_2^{-})^2}\log {d_1 \choose s_1^+ }\log_2(d_1)\log_2(d_2))\big)^{3/2}}{\Big(\log\big(1 + \frac{d_2}{(s_2^{*})^2}\log {d_1 \choose s_1^* } \log_2(d_1)\log_2(d_2) \big)\Big)^{2}} \quad \text{(since $d_1 \land d_2 \geq 8 \geq 2^e$)} \\
    &\leq \frac{4\sqrt{2}C_1C^{3/2}}{C^4_\mu} \frac{1}{\sqrt{\log\big(1 + \frac{d_2}{(s_2^{*})^2}\log {d_1 \choose s_1^* } \log_2(d_1)\log_2(d_2)\big)}} \\
    &\leq \frac{4\sqrt{2}C_1C^{3/2}}{C^4_\mu \sqrt{\log(1 + c)}} \\
    &\leq \frac{\alpha}{4},
\end{align*}
where the final inequality holds with $C_\mu$ taken sufficiently large. Assembling these results, we conclude
\[\pr_\X \left(\Delta^{\text{ada}}_{\chisqmax} = 0\right) \leq \text{I} + \text{II} \leq \frac{\alpha}{2}.\]
Combining this result with our Type 1 error bound, we have shown that
\[\cR\left(\Delta^{\text{ada}}_{\chisqmax}\right) \leq \alpha,\]
and the proof is complete.
\end{proof}

\subsubsection{Proof of Proposition 2}
\begin{proof}
Recall that $\Omega$ denotes the dyadic partition of the set $[d_1] \times [d_2]$:
\[
    \Omega = \Bigg\{\Big(\frac{d_1}{2^{m_1}}, \frac{d_2}{2^{m_2}}\Big): m_1 \in [\log_2(d_1)], m_2 \in [\log_2(d_2)]  \Bigg\}.
\]
Let us define the following index sets:
    \begin{enumerate}
        \item $\Omega_1 = \Big\{(s_1, s_2) \in \Omega : \tilde{R}(s_1, s_2) = \psi_{12} + \beta_{21} \Big\}$
        \item $\Omega_2 = \Big\{(s_1, s_2) \in \Omega : \tilde{R}(s_1, s_2) = \psi_{21} + \beta_{12} \Big\}$
        \item $\Omega_3 = \Big\{(s_1, s_2) \in \Omega : \tilde{R}(s_1, s_2) = \phi_{12} \Big\}$ 
        \item $\Omega_3 = \Big\{(s_1, s_2) \in \Omega : \tilde{R}(s_1, s_2) = \phi_{21} \Big\}$
    \end{enumerate}
    where we define $\tilde{R}(s_1, s_2) = \tilde{R}$ with the dependence on $s_1$ and $s_2$ rendered explicit. Using this notation, we simply decompose the minimax risk of $\Delta^{*, \text{ada}}$ as follows:
    \begin{align*}
        \cR(\Delta^{*, \text{ada}}) &= \cR\left(\max_{(s_1, s_2) \in \Omega}\Delta^*(s_1, s_2)\right) \\
        &\leq \cR\left(\max_{(s_1, s_2) \in \Omega_1
         }\Delta^*(s_1, s_2)\right) + \cR\left(\max_{(s_1, s_2) \in \Omega_2
         }\Delta^*(s_1, s_2)\right) \\
         &\quad + \cR\left(\max_{(s_1, s_2) \in \Omega_3
         }\Delta^*(s_1, s_2)\right) + \cR\left(\max_{(s_1, s_2) \in \Omega_4
         }\Delta^*(s_1, s_2)\right).
    \end{align*}
    The first two summands are controlled by Lemmas \ref{lem_ub_ada_max} and \ref{lem_ub_ada_maxlinear} according to cases, and the latter two by Lemmas \ref{lem_ub_ada_linear} and \ref{lem:linear-test} similarly by cases. Since $(\psi_{12} + \psi_{21}) \land \phi_{12} \land \phi_{21} \gtrsim \tilde R(s_1^*, s_2^*)$ by Lemma \ref{lem_simplified_rate_for_UB}, the proof is complete.
\end{proof}
\section{Additional technical lemmas}

\begin{lemma}\label{lem_2.6}
    For any two real numbers $x,y>1$, at least one of the two inequalities $x \geq e \log(y)$ or $y \geq e\log(x)$ holds.
\end{lemma}
\begin{proof}[Proof of Lemma ~\ref{lem_2.6}]
    We first prove a preliminary result: For any $x > 1$, it holds that $x \geq e\log(e \log(x))$. 
    To see this, define the function $f: (1,\infty) \to \R;~ x \mapsto x - e\log(e \log(x))$. We have for any $x>1$
    \begin{align*}
        f'(x) = 1 - \frac{e}{x\log(x)}.
    \end{align*}
    The only value $x^*$ for which $f(x^*)=0$ is $x^* = e$, which implies that $f$ is decreasing over $(1,e)$ and increasing over $(e,\infty)$. Hence, $f$ is minimized at $e$ and its corresponding minimum value is
    \begin{align*}
        f(e) = e - e\log(e\log(e)) = 0.
    \end{align*}

    This fact being established, assume now for the sake of contradiction that there exist two real numbers $x,y$ such that $x<e \log(y)$ and $y< e\log(x)$. 
    Then we obtain $x< e\log(e\log(x))$ which is a contradiction. This concludes the proof.
\end{proof}

\begin{lemma}\label{lemma:probXk}

Let $X \sim \text{Bin}(n, p)$ with $p \leq \frac12$. Then for any $k \in \{1, ..., n\}$, we have 

\[\Pr(X = k) \leq \Big(\frac{2enp}{k}\Big)^k\exp\big(-np\big)\]
    
\end{lemma}

\begin{proof}
    Using the bound ${n \choose k } \leq \big(\frac{ne}{k}\big)^k$ (Appendix A in \cite{roch2024modern}), we have 

    \begin{align*}
        \Pr(X = k) &= {n \choose k}p^k\big(1 - p\big)^{n - k} \\
        &\leq \Big(\frac{npe}{k}\Big)^k\big(1 - p\big)^{n - k} \\
        &=  \Big(\frac{npe}{k}\Big)^k\big(1 - p\big)^{n}\big(1 - p\big)^{-k} \\ 
        &\leq \Big(\frac{npe}{k}\Big)^k\big(1 - p\big)^{n}\big(1/2\big)^{-k} \\
        &= \Big(\frac{2enp}{k}\Big)^k\Big(1 - p\Big)^{n} \\
        &\leq \Big(\frac{2enp}{k}\Big)^k\exp\big(-np\big)
    \end{align*}
where the final inequality uses $(1 - x)^b \leq \exp(-xb)$ for any $x, b \geq 0$.
\end{proof}

\begin{lemma}\label{lemma:log-plus-one}
    Let $f : S \rightarrow (0,\infty)$ be differentiable, where $S \subseteq (0,\infty)$ is an interval. For $x \in S$, define $g(x) = \frac1x\log\big(f(x)\big)$ and $g^{(1)}(x) = \frac1x \log\big(1 + f(x)\big)$. If $g$ is decreasing on $S$, then $g^{(1)}$ is decreasing on $S$ as well.
\end{lemma}

\begin{proof}
    Let $f'(x) = \frac{\D f}{\D x}(x)$ for any $x \in S$. By direct calculation, we have 

    \[\frac{\D g}{ \D x}(x) = \frac{f'(x)}{x f(x)} - \frac{\log \big(f(x)\big)}{x^2}.\]

    If $g$ is decreasing on $S$, then for each $x \in S$ it holds

    \[\frac{f'(x)}{x f(x)} < \frac{\log \big(f(x)\big)}{x^2}.\]

    Now calculating the derivative of $g^{(1)}$ at any $x \in S$, we have 

    \begin{align*}
        \frac{\D g^{(1)}}{\D x}(x) &= \frac{f'(x)}{x \big(1 + f(x)\big)} - \frac{\log \big(1 + f(x)\big)}{x^2} \\
        &= \frac{f'(x)}{x \big(1 + f(x)\big)}\frac{f(x)}{f(x)} - \frac{\log \big(1 + f(x)\big)}{x^2} \\
        &< \frac{\log \big(f(x)\big)}{x^2}\frac{f(x)}{\big(1 + f(x)\big)}- \frac{\log \big(1 + f(x)\big)}{x^2} \\
        &\leq 0
    \end{align*}

    where the final inequality uses $f(x) \leq 1 + f(x)$ and $\log\big(f(x)\big) \leq \log\big(1 + f(x)\big)$. This completes the proof.
\end{proof}

\begin{lemma}\label{lem_logs}
    Let $x,y > e$.
    \begin{enumerate}
        \item[(i)] If $x/\log(x) \leq y/2$, then we have $x \leq y \log(y)$.
        \item[(ii)] If $x/\log(x) \geq y$, then we have $x \geq y \log(y)$.
    \end{enumerate}
\end{lemma}
\begin{proof}
\begin{enumerate}
    \item[(i)] We prove the statement by contrapositive. Suppose that $x > y \log (y)$. Then since $t \mapsto \frac{t}{\log(t)}$ is increasing over $(e, \infty)$, we have
    \begin{align*}
        \frac{x}{\log(x)} \geq \frac{y \log(y)}{\log(y \log (y))} = \frac{y}{ 1 +\frac{\log \log(y)}{\log(y)}} > \frac{y}{2}.
    \end{align*}
    \item[(ii)] Again, we prove the statement by contrapositive. 
    Suppose that $x /\log(x) < y.$ Then since $t \mapsto \frac{t}{\log(t)}$ is increasing over $(e, \infty)$, we have
     \begin{align*}
        \frac{x}{\log(x)} < \frac{y \log(y)}{\log(y \log (y))} = \frac{y}{ 1 +\frac{\log \log(y)}{\log(y)}} < y.
    \end{align*}
    This completes the proof.
\end{enumerate}

\end{proof}

\begin{lemma}\label{lem_logs_constants}
    The following two properties hold.
    \begin{enumerate}
        \item[(i)] For any $x \geq 0$ and $c \in [0,1]$, we have $\log(1+cx) \geq c\log(1+x)$.
        \item[(ii)] For any $x \geq 0$ and $C\geq 1$, we have $C\log(1+x) \geq \log(1+Cx)$.
    \end{enumerate}
\end{lemma}

\begin{proof}[Proof of Lemma~\ref{lem_logs_constants}]
\begin{enumerate}
    \item[(i)] By concavity of the logarithm, we have $\log(1+cx) = \log\left((1-c)\cdot 1 + c(1+x)\right) \geq (1-c) \log(1) + c\log(1+x) = c \log(1+x)$ for any $c \in [0,1]$. 
    \item[(ii)] We have seen above that $\log(1+cy) \geq c \log(1+y)$ for any $y \geq 0$ and $c \in [0,1]$. 
Applying this with $y = Cx$ and $c = 1/C$ yields the result. 
\end{enumerate}
    
\end{proof}

\section{Additional justifications: imbalanced regimes}\label{app_additional_justifications}

\subsection{Proof of Corollaries in Section 4}

\begin{proposition}[Corollary 2 in the main text.]
        Suppose that $s_1^2 \geq \bar{c}d_1 s_2$ for a constant $\bar{c} > 0$ and $s_j \leq c_j d_j$ for $j \in \{1, 2\}$ where $c_1,c_2 > 0$ are sufficiently small constants. Additionally, suppose that $\frac{d_1}{s_1} \geq e \log(\frac{d_2}{s_2})$ and that there exists a constant $\alpha > 0$ such that $d_2 \geq s_2^{2 + \alpha}$. Then it holds
\begin{equation}
    (\mu^*)^2 \asymp \frac{1}{s_2}\log\left(1 + \frac{d_1s_2}{s_1^2}\log(d_2)\right).
\end{equation}
    \end{proposition}
    \begin{proof}
        Theorem 2 grants us that 
        \[(\mu^*)^2 \asymp (\psi_{12} + \psi_{21}) \land \phi_{12} \land \phi_{21}.\]
        By Lemma \ref{lem_simplify_dense_allrates}, we have that
        \[(\psi_{12} + \psi_{21}) \land \phi_{12} \asymp \psi_{21} \land \phi_{12}.\]
        Using the inequality $\log(1+x) \le x$ which holds for any $x \ge 0$, we have
\begin{align*}
\psi_{21} &= \frac{1}{s_2} \log\left(1 + \frac{d_1}{s_1^2} \log\left(e\binom{d_2}{s_2}\right)\right) \\
&\asymp \frac{1}{s_2} \log\left(1 + \frac{d_1 s_2}{s_1^2} \log(d_2/s_2)\right) \\
&\asymp \frac{1}{s_2} \log\left(1 + \frac{d_1 s_2}{s_1^2} \log(d_2)\right) \quad \text{(since $d_2 \geq s_2^{2 + \alpha}$)}\\
&\le \frac{1}{s_2} \frac{d_1 s_2}{s_1^2} \log(d_2) \\
&\asymp \frac{d_1}{s_1^2} \log\left(1 + \frac{d_2}{s_2^2}\right) \quad \text{(since $d_2 \geq s_2^{2 + \alpha}$)}  \\
&= \phi_{12}.
\end{align*}
This implies that $\psi_{21} \land \phi_{12} = \psi_{21}$. It thus remains to show that $\psi_{21} \land \phi_{21} = \psi_{21}$. Observe
\begin{align*}
    \phi_{21} &= \frac{d_2}{s^2_2}\log\left(1 + \frac{d_1}{s_1^2}\right) \\
    &\asymp \frac{d_2d_1}{s^2_2s_1^2} \quad \text{(since $s_1^2 \gtrsim d_1s_2$)} \\
    &\geq \frac{d_1}{s_1^2}\log\left(1 + \frac{d_2}{s_2^2}\right) \quad \text{(again using $\log(1+x) \leq x$ for $x \geq 0$)} \\
    &= \phi_{12}.
\end{align*}
Since we have also demonstrated that $\phi_{12} \geq \psi_{21} \asymp s_2^{-1}\log\left(1 + \frac{d_1 s_2}{s_1^2} \log(d_2)\right)$, the proof is complete.
    \end{proof}

    \begin{proposition}[Corollary 3 in the main text.]
        Suppose that $s_1 = s_2^2$ and $d_1 = d_2^2$, and that there exists a constant $\alpha > 0$ such that $d_1 \geq s_1^{2 + \alpha}$. Then it holds
       \begin{equation}
            (\mu^*)^2 \asymp \frac{\log \left(d_1\right)}{\sqrt{s_1}}.
        \end{equation}
    \end{proposition}
    \begin{proof}
        Under the conditions $s_1 = s_2^2$ and $d_1 = d_2^2$, it holds
        \begin{align*}
            \psi_{12} + \psi_{21} &= \frac{1}{s_1} \log\!\Big(1+ \frac{d_2}{s_2^2} \log\!\Big(\!e\scalebox{.92}{$\displaystyle{d_1 \choose s_1}$}\Big)\,\Big) + \frac{1}{s_2} \log\!\Big(1+ \frac{d_1}{s_1^2} \log\!\Big(\!e\scalebox{.92}{$\displaystyle{d_2 \choose s_2}$}\Big)\,\Big) \\
            &\asymp \frac{1}{s_1} \log\!\Big(1+ \sqrt{d_1}\log\!\Big(\frac{d_1}{s_1}\Big)\,\Big) + \frac{1}{\sqrt{s_1}} \log\!\Big(1+ \frac{d_1}{s_1^{3/2}} \log\!\Big(\frac{d_1}{s_1}\Big)\,\Big) \\
            &\asymp \frac{1}{s_1} \log(d_1) + \frac{1}{\sqrt{s_1}} \log(d_1) \\
            &\asymp \frac{1}{\sqrt{s_1}} \log(d_1)
        \end{align*}
        where the second-to-last expression uses $d_1 \geq s_1^{2 + 
        \alpha}$ to simplify the second summand. It thus remains to show that 
        \[\frac{1}{\sqrt{s_1}} \log(d_1) \leq \phi_{12} \land \phi_{21}.\]
        Note that the conditions $s_1 = s_2^2$, $d_1 = d_2^2$, and $d_1 \geq s_1^{2 + \alpha}$ imply that $d_2 \geq s_2^{2 + \alpha}$. Thus observe
        \begin{align*}
            \phi_{12} &= \frac{d_1}{s_1^2}\log\left(1 + \frac{d_2}{s_2^2}\right) \\
            &\gtrsim  s_1^{\alpha}\log\left( \frac{d_2}{s_2}\right) \\
            &\asymp s_1^{\alpha}\log\left( \frac{d_1}{s_1}\right) \\
            &\asymp s_1^{\alpha}\log\left( d_1\right) \\
            &> \frac{1}{\sqrt{s_1}} \log(d_1).
        \end{align*}
        The proof that $\phi_{21} \geq s_1^{-1/2}\log(d_1)$ is symmetric and thus omitted. The proof is complete.
    \end{proof}

\subsection{Case $s_1 = 1$}

Assume $s_1 = 1$. We distinguish between the following cases.

\begin{enumerate}
    \item First case: Assume $\log(d_1) > s_2$. Then the quantities involved in the rate simplify as follows
    \begin{itemize}
        \item $\psi_{12} = \log\left(1+\frac{d_2}{s_2^2}\log(ed_1)\right) \asymp \log\left(\frac{ed_2}{s_2}\right) + \log\left(\frac{e\log(d_1)}{s_2}\right)$
        \item $\psi_{21} = \frac{1}{s_2} \log\left(1+d_1 \log(e{d_2 \choose s_2})\right) \asymp \frac{\log d_1}{s_2} + \frac{\log(\log(e{d_2 \choose s_2}))}{s_2},$
    \end{itemize}
    hence 
    \begin{align*}
        \psi_{12} + \psi_{21} \asymp \log\left(\frac{ed_2}{s_2}\right) + \frac{\log d_1}{s_2}.
    \end{align*}
    Moreover, we also have
    \begin{align*}
        \phi_{21} = \frac{d_2}{s_2^2}\log(1+d_1) \gtrsim \log\left(\frac{ed_2}{s_2}\right) + \frac{\log d_1}{s_2} \asymp \psi_{12} + \psi_{21}.
    \end{align*}
    As for $\phi_{12}$, assume first that $s_2^2 \leq d_2$. Then, we have
      \begin{align*}
        \phi_{12} &= d_1 \log\left(1+\frac{d_2}{s_2^2}\right) \\
        &\asymp \frac{d_1}{s_2} \frac{d_2}{s_2} \geq \frac{d_1}{s_2} \lor \frac{d_2}{s_2} \qquad  \text{ since } d_1 \geq e^{s_1} \geq s_1 \text{ and } d_2 \geq s_2\\
        &\geq \frac{\log(s_1)}{s_2} \lor \log(e\frac{d_2}{s_2}) \\
        &\asymp \psi_{12} + \psi_{21}.
    \end{align*}
    Next, assume that $s_2^2 > d_2$. Then $\log(ed_2/s_2) \asymp \log(ed_2) \asymp \log(es_2)$ and we know that $\log(d_1) \geq s_2$, which implies $d_1 \geq \log(s_2)$. Therefore, 
    \begin{align*}
        \phi_{12} \gtrsim d_1 \geq \frac{\log(d_1)}{s_2} \lor \log(s_2) \asymp \psi_{12} + \psi_{21}.
    \end{align*}
    It follows that $(\mu^*)^2 \asymp \psi_{12} + \psi_{21} \asymp \log\left(\frac{ed_2}{s_2}\right) + \frac{\log d_1}{s_2}$, as claimed. 

\item Assume now that $\log(d_1)< s_2$. We have $\psi_{12} = \log\left(1+\frac{d_2}{s_2^2}\log(ed_1)\right)$, and 
\begin{itemize}
    \item $\phi_{12} = d_1 \log\left(1+\frac{d_2}{s_2^2}\right) \geq \psi_{12}$ by Lemma~\ref{lem_logs_constants}.(ii).
    \item $\phi_{21} = \frac{d_2}{s_2^2}\log (1+d_1) \gtrsim \psi_{12}$.
\end{itemize}
Assume first that $\frac{d_2}{s_2^2}\log (ed_1)<1$. It follows that $\psi_{12} \asymp \phi_{21}$, so that $(\mu^*)^2 \asymp(\psi_{12} + \psi_{21}) \land \phi_{12} \land \phi_{21} \asymp \phi_{12}$, as claimed. 

Now, assume that $\frac{d_2}{s_2^2}\log (ed_1) \geq 1$. Then by Lemma~\ref{lem_logs_constants}.(ii), we have
\begin{align*}
    s_2 \psi_{12} \geq \log\left(1+ \frac{d_2\log(ed_1)}{s_2}\right) \gtrsim \log(1+s_2) \lor \log\log\left(\frac{ed_2}{s_2}\right) \asymp \log\left(e{d_2 \choose s_2}\right),
\end{align*}
hence $\psi_{12} \gtrsim \psi_{21}$. It follows that $(\mu^*)^2 \asymp \psi_{12}$, which concludes the proof. 
\end{enumerate}

\end{document}